\documentclass[final,hidelinks,onefignum,onetabnum]{siamart220329}

\usepackage{lipsum}
\usepackage{amsfonts}
\usepackage{amsmath}
\usepackage{graphicx}
\usepackage{epstopdf}
\usepackage{algorithmic}
\usepackage{bm}
\usepackage{hyperref}

\usepackage{booktabs}
\usepackage{multirow} 
\usepackage{subcaption}
\usepackage{caption}

\usepackage{tikz}
\usetikzlibrary{positioning}
\usetikzlibrary{arrows.meta}
\usepackage{extarrows}

\usepackage{enumitem}

\newtheorem{ex}{Example}[section]
\ifpdf
  \DeclareGraphicsExtensions{.eps,.pdf,.png,.jpg}
\else
  \DeclareGraphicsExtensions{.eps}
\fi

\newsiamremark{rem}{Remark}
\newsiamremark{hypothesis}{Hypothesis}
\crefname{hypothesis}{Hypothesis}{Hypotheses}
\newsiamthm{claim}{Claim}
\newsiamthm{prop}{Proposition}
\newsiamthm{thm}{Theorem}

\headers{IDP Framework for PAMPA}{R\'emi Abgrall, Miaosen Jiao, Yongle Liu, Kailiang Wu}

\title{A Novel and Simple Invariant-Domain-Preserving Framework for PAMPA Scheme: 1D Case	
	\thanks{
			Kailiang Wu and Miaosen Jiao were partially supported by Shenzhen Science and Technology Program (No.~RCJC20221008092757098) and 
			National Natural Science Foundation of China (No.~12171227). Yongle Liu was supported by UZH Postdoc Grant, 2024 / Verf\"{u}gung Nr. FK-24-110 and SNSF grant 200020$\_$204917.}} 
\author{
	R\'{e}mi Abgrall\thanks{Institute of Mathematics,
		University of Z\"{u}rich, 8057 Z\"{u}rich, Switzerland (\email{remi.abgrall@math.uzh.ch}).}
	\and 
	Miaosen Jiao\thanks{Department of Mathematics, Southern University of Science and Technology, Shenzhen 518055, China
  (\email{12332859@mail.sustech.edu.cn}).}
\and Yongle Liu\thanks{Institute of Mathematics,
		University of Z\"{u}rich, 8057 Z\"{u}rich, Switzerland (\email{yongle.liu@math.uzh.ch}).}
\and Kailiang Wu\thanks{Corresponding author. Department of Mathematics and Shenzhen International Center
for Mathematics, Southern University of Science and Technology, Shenzhen
518055, China
  (\email{wukl@sustech.edu.cn}).}
}

\usepackage{amsopn}

\ifpdf
\hypersetup{
  pdftitle={IDP Framework for PAMPA},
  pdfauthor={Remi Abgrall, Miaosen Jiao, Yongle Liu, Kailiang Wu}
}
\fi

\allowdisplaybreaks

\begin{document}

\maketitle

\begin{abstract} 
	The PAMPA (Point-Average-Moment PolynomiAl-interpreted) method was proposed in [R. Abgrall, Commun. Appl. Math. Comput., 5(1): 370--402, 2023], as an innovative approach effectively combining the conservative and non-conservative formulations of an hyperbolic system of conservation laws to evolve cell averages and point values. Solutions to hyperbolic conservation laws typically admit an invariant domain, and preserving numerical solutions within this domain is essential yet nontrivial. In this paper, we propose a novel framework for designing efficient Invariant-Domain-Preserving (IDP) PAMPA schemes. We first provide a rigorous theoretical analysis of the IDP property for the updated cell averages in the original PAMPA scheme, revealing the critical roles of cell average decomposition and the enforcement of midpoint values within the invariant domain. This analysis also highlights the challenges of relying solely on continuous fluxes to ensure the updated cell averages within the invariant domain. Building on these insights, we introduce a simple IDP limiter for cell midpoint values and {\bf construct a provably IDP PAMPA scheme} that always maintains the IDP property for updated cell averages {\bf by theoretical proof without the need for any post-processing limiters}. This approach contrasts with existing bound-preserving PAMPA schemes, which typically require extra convex limiting to blend high-order and low-order schemes. Most notably, inspired by the Softplus and Clipped ReLU functions from machine learning, we {\bf propose an innovative, automatic IDP reformulation} of the governing equations. Thanks to this new formulation, we {\bf design an unconditionally limiter-free IDP scheme} for evolving point values. We also introduce novel techniques to suppress spurious oscillations in the IDP PAMPA scheme, allowing for effective capture of strong shocks. Numerical experiments in 1D, including tests on the linear convection equation, Burgers’ equation, the compressible Euler equations and MHD equations, demonstrate the accuracy and robustness of the proposed IDP PAMPA scheme.

\end{abstract}

%

\section{Introduction}

Hyperbolic conservation laws are essential for modeling diverse phenomena, including gas dynamics, traffic flow, magnetohydrodynamics, water wave propagation, elastodynamics, shallow water flows, population dynamics, and atmospheric flows. 
Over the past few decades, numerous numerical methods for hyperbolic conservation laws have been developed, including finite volume, finite difference, discontinuous Galerkin (DG) \cite{ShuRKDG1998,Shu1989TVBRKDG}, and spectral difference (SD) methods \cite{SD2005}, among others. Finite volume and finite difference methods typically improve accuracy by extending the computational stencil and employing higher-order reconstructions, such as Essentially Non-Oscillatory (ENO) \cite{HARTEN1987231} and Weighted ENO (WENO) \cite{LIU1994200,shu1989efficient} schemes. DG and SD methods achieve high accuracy by increasing the degrees of freedom within each cell. 
However, traditional numerical methods, while enhancing accuracy, also exhibit certain limitations. For example, the large computational stencil in WENO schemes can complicate parallelization, while the DG method is often memory-intensive.

To address these limitations, Eymann and Roe \cite{EymannRoeJan.2011,EymannRoeJune.2011,EymannRoe2013} proposed the active flux (AF) method, a hybrid finite element–finite volume approach inspired by van Leer \cite{VANLEER1977276}. Unlike traditional methods, the AF method enhances accuracy by introducing point values at cell interfaces as additional degrees of freedom, alongside traditional cell averages. The cell averages are updated through a finite volume approach, while point values evolve using exact or approximate evolution operators with a quadrature rule in time, typically the Simpson rule. 
The original AF method employs continuous reconstructions to represent the numerical solution and uses values of the flux function at quadrature nodes to approximate fluxes across cell interfaces. This design inherently produces a third-order accurate method, eliminating the need for time integration schemes, such as Runge--Kutta methods. In this framework, the evolution operators for point values are critical. Exact evolution operators based on characteristic methods have been developed for linear hyperbolic equations (see, e.g., \cite{EymannRoe2013,Roe2015,Barsukow2019-wx,Calhoun2023-ms}). For nonlinear systems, several approximate evolution operators have been introduced, including those for Burgers’ equation \cite{EymannRoeJan.2011,EymannRoeJune.2011,Barsukow2020-rt}, the 1D compressible Euler equations \cite{EymannRoeJune.2011,Helzel2019-ix,Barsukow2020-rt}, and hyperbolic balance laws \cite{Barsukow2020-rt,Barsukow2024-tc}.

For nonlinear systems, constructing exact or approximate evolution operators becomes significantly more complex, particularly in multiple spatial dimensions. A recent active flux-inspired method, introduced by Abgrall \cite{abgrall2023combination}, offers a streamlined approach to evolving point values and combining different formulations (conservative and non-conservative) of a nonlinear hyperbolic system. This method uses point values and the cell average to construct a quadratic polynomial approximation of the solution within each cell. The updates for both the cell average and point values are formulated in a semi-discrete form and advanced in time using standard Runge--Kutta methods. This approach achieves third-order accuracy and has been extended to arbitrary high-order accuracy in \cite{ABGRALL2023hybrid} and \cite{ABGRALL2023highorder}. 
More recently, this method has been generalized to multidimensional triangular meshes in compressible flow problems \cite{abgrall2023activefluxtriangularmeshes}, to one-dimensional hyperbolic systems of balance laws \cite{liu2024arbitrarily,abgrall2024newapproach}, and to multidimensional hyperbolic balance laws \cite{liu2024pampa}, where this method is named the PAMPA (Point-Average-Moment PolynomiAl-interpreted) scheme.

Solutions of hyperbolic equations often satisfy specific bounds, typically forming a convex invariant domain $G$. For example, solutions to scalar conservation laws adhere to a strict maximum principle (see \eqref{IRofscalar}), while the compressible Euler equations require that density and pressure remain positive (see \eqref{IRofEuler}). Similarly, in the shallow water equations, the water height must stay positive. 
It is essential to explore invariant-domain-preserving (IDP) or bound-preserving schemes that ensure the numerical solution satisfies these bounds (i.e., remains within the invariant domain $G$). For example, in the Euler equations, non-physical negative pressures and densities produced numerically can undermine the equation’s hyperbolicity, leading to nonlinear instabilities and potentially causing the breakdown of the computation. First-order accurate monotone schemes (such as the Godunov, Lax–Friedrichs, and Engquist–Osher schemes) are known to be IDP for scalar conservation laws. These schemes are also adapted to many hyperbolic systems. However, developing high-order accurate IDP schemes is nontrivial.  
In \cite{zhang2010maximum,zhang2010positivity}, Zhang and Shu introduced a general framework for constructing high-order IDP finite volume and discontinuous Galerkin schemes for hyperbolic conservation laws. A key aspect of this approach is rewriting a high-order scheme as a convex combination of first-order schemes through cell average decomposition (CAD), which determines the IDP Courant–Friedrichs–Lewy (CFL) condition. Recently, in \cite{CUI2023ocad,cui2024ocad}, the authors developed a generic approach to achieve an optimal CAD that maximizes the IDP CFL number. 
Another approach for IDP schemes in hyperbolic systems involves flux-correction limiters \cite{hu2013positivity,Liang2014-cj,Xu2014,wu2015high,Xiong2016} or convex limiting \cite{Guermond2015,ConvexLimiting2018,KUZMIN2020,Kuzmin2022}, which use a prepared  IDP (lower-order) scheme to modify high-order schemes via convex combinations. Rigorously proving the IDP properties in the presence of nonlinear constraints presents a considerable challenge, even for first-order schemes. The geometric quasilinearization (GQL) framework \cite{Wu2023GQLreview} addresses this challenge using an equivalent linear representation of any convex invariant domain, motivated by earlier research on  IDP schemes \cite{Wu2017PPRHD,Wu2018PPMHD,Wu2019PPMHD,Wu2021ppmultidimensionalrelativisticMHD,Wu2021RHD} for magnetohydrodynamics (MHD) and relativistic hydrodynamics systems.

Research on developing IDP or bound-preserving PAMPA schemes is crucial but still in its early stages. In \cite{CHUDZIK2021125501}, a limiter was introduced for the point values of the original AF scheme, where the bound-preserving property was shown for the linear advection equation. The MOOD procedure is applied to the PAMPA scheme in \cite{abgrall2023combination} to enforce the positivity of density and pressure when solving the Euler equations. In \cite{duan2024activefluxmethodshyperbolic}, the authors refer to the PAMPA scheme as a generalized AF scheme, employing several flux vector splittings to address transonic issues in nonlinear problems. They also design bound-preserving limiters for evolving both cell averages and point values, using convex limiting and scaling limiting, respectively. 
Recently, another bound-preserving effort for the PAMPA scheme was presented in \cite{abgrall2024BPPAMPA}, which utilized a convex blending of fluxes computed by the high-order PAMPA scheme and a low-order local Lax–Friedrichs-type scheme. Instead of using a scaling limiter as in \cite{duan2024activefluxmethodshyperbolic}, the bound-preserving approach in \cite{abgrall2024BPPAMPA} for point values is distinct, as it formulates the scheme as a convex combination of two residual terms. By blending higher-order residuals with lower-order residuals, each residual term remains within the invariant region $G$, ensuring bound preservation for the point values \cite{abgrall2024BPPAMPA}. Using the geometric quasilinearization (GQL) technique \cite{Wu2023GQLreview}, the bound-preserving PAMPA approach in \cite{abgrall2024BPPAMPA} has been extended to hyperbolic systems such as the Euler equations.

The aim of this paper is to develop a novel, simple, yet efficient IDP framework for PAMPA scheme. 
The novelty and contributions of  this work is summarized as follows. 

\begin{itemize}[leftmargin=*]
	\item {\bf Theoretical Analysis of IDP Property:} The paper first provides a rigorous theoretical analysis of the IDP (Invariant-Domain-Preserving) property for the original PAMPA scheme. This analysis reveals the importance of cell average decomposition and enforcing midpoint values within the invariant domain to maintain the IDP property.
	\item {\bf Identification of Challenges with Continuous Flux:} The analysis highlights the limitations of relying solely on continuous fluxes to ensure that updated cell averages remain within the invariant domain, motivating a new approach to enforce the IDP property.
	\item {\bf  Introduction of a Simple IDP Limiter for Midpoint Values:} 
	A simple IDP limiter, motivated by \cite{zhang2010maximum,zhang2010positivity} is introduced to enforce the IDP property at midpoint values. 
	\item {\bf Provably IDP PAMPA Scheme for Cell Averages:} The paper constructs a provably IDP PAMPA scheme with standard bound-preserving numerical fluxes to maintain the IDP property for updated cell averages without the need for post-processing limiters, contrasting with other bound-preserving PAMPA schemes \cite{duan2024activefluxmethodshyperbolic,abgrall2024BPPAMPA} that typically require the extra convex limiting procedure to blend high-order and low-order schemes.
	\item {\bf Unconditionally Limiter-Free IDP scheme for Point Values:} 
	The PAMPA approach provides great flexibility to use various non-conservative formulations for scheme of the point values. 
	Inspired by the Softplus and Clipped ReLU functions in machine learning, the paper proposes a novel, automatic IDP reformulation of the governing equations.  Thanks to this reformulation, we successfully develop an innovative, unconditionally limiter-free IDP scheme for evolving point values. 
	\item {\bf Oscillation-Control Techniques:} 
	A new  oscillation-eliminating (OE) technique and the monotonicity-preserving (MP) limiter are introduced to suppress spurious oscillations, enabling effective capture of strong shocks by IDP PAMPA scheme.
	\item {\bf Comprehensive Numerical Validation:} 
	The paper presents a series of 1D numerical experiments, including tests on the linear convection equation, Burgers' equation, the compressible Euler equations, and MHD equations, to demonstrate the accuracy and robustness of the proposed IDP PAMPA scheme.
\end{itemize}
The extension of our scheme to multidimensions will be discussed in a future paper.

\section{Invariant Domains in Hyperbolic Conservation Laws}

This paper focuses on one-dimensional (1D) hyperbolic conservation laws:
\begin{align}\label{HCL}
	\begin{cases}
		\displaystyle 
		\frac{\partial \mathbf{U}}{\partial t} + \frac{\partial \mathbf{F}(\mathbf{U})}{\partial x} = {\bf 0}, \quad &(x,t) \in \mathbb{R} \times \mathbb{R}^+, \\
		\mathbf{U}(x,0) = \mathbf{U}_0(x), \quad &x \in \mathbb{R}^d,
	\end{cases}
\end{align}
where $\mathbf{U} \in \mathbb{R}^d$ is the conservative vector, $\mathbf{F} \in \mathbb{R}^d$ represents the flux, and $\mathbf{U}_0 \in \mathbb{R}^d$ is the initial data.

Typically, the exact physical solutions to \eqref{HCL} remain within a convex invariant domain $G \subset \mathbb{R}^d$. Specifically, if the initial solution $\mathbf{U}_0(x) \in G$, then the solution 
$\mathbf{U}(x, t) \in G$ for all $t > 0$. This invariant domain is often determined by fundamental physical bounds or principles. Three illustrative examples are discussed below.

\begin{ex}[Scalar Conservation Laws]
	Consider a scalar conservation law in the form:
	\begin{align}\label{Scalar}
		\begin{cases}
			\displaystyle
			\frac{\partial u}{\partial t} + \frac{\partial f(u)}{\partial x} = 0, \quad &(x,t) \in \mathbb{R} \times \mathbb{R}^+, \\
			u(x,0) = u_0(x), \quad &x \in \mathbb{R},
		\end{cases}
	\end{align}
	where the exact (entropy) solutions satisfy the maximum principle:
	\begin{align}\label{IRofscalar}
		u(x,t) \in G := [U_{\min}, U_{\max}] \qquad \forall t \geq 0,
	\end{align}
	with $U_{\min} = \min_{x} u_0(x)$ and $U_{\max} = \max_{x} u_0(x)$. Clearly, $G = [U_{\min}, U_{\max}]$ is a convex invariant domain for \eqref{Scalar}.
\end{ex}

\begin{ex}[Compressible Euler Equations]
	This system in 1D is given by
	\begin{align}\label{Eulerconservation}
		\frac{\partial}{\partial t} \begin{pmatrix} \rho \\ \rho v \\ E \end{pmatrix} + \frac{\partial}{\partial x} \begin{pmatrix} \rho v \\ \rho v^2 + p \\ v(E + p) \end{pmatrix} = {\bf 0},
	\end{align}
	where $ \rho $ is the fluid density, $ v $ is the velocity, and $ p $ is the pressure. The total energy $ E = \frac{1}{2} \rho v^2 + \rho e $, where $ e = \frac{p}{(\gamma - 1)\rho} $ is the specific internal energy. 
	The physical solutions to this system must ensure the positivity of both density and pressure, namely, 
	\begin{align}\label{IRofEuler}
		\mathbf{U} \in G := \left\{ \mathbf{U} = (\rho, \rho v, E)^T: \rho > 0, \, p = (\gamma - 1) \left( E - \frac{1}{2} \rho v^2 \right) > 0 \right\}.
	\end{align}
	This set $G$ forms a convex invariant domain for the Euler equations \eqref{Eulerconservation}; see \cite{zhang2010positivity,ConvexLimiting2018}.
\end{ex}

\begin{ex}[Ideal MHD Equations]
    The 1D MHD system reads
  \begin{align}\label{MHD}
\frac{\partial }{\partial t}
\begin{pmatrix}
\rho \\ 
\rho v_x \\ 
\rho v_y \\ 
\rho v_z \\ 
B_y \\ 
B_z \\ 
E
\end{pmatrix}
+
\frac{\partial}{\partial x}
\begin{pmatrix}
\rho v_x \\
\rho v_x^2 + p^* - B_x^2 \\
\rho v_x v_y - B_x B_y \\
\rho v_x v_z - B_x B_z \\
B_y v_x - B_x v_y \\
B_z v_x - B_x v_z \\
(E + p^*) v_x - B_x \big( B_x v_x + B_y v_y + B_z v_z \big)
\end{pmatrix} = {\bf 0}.
\end{align}
Here, $\mathbf{B}=(B_x,B_y,B_z)$ represents the magnetic field with $B_x$ being a given constant, and $\mathbf{v}=(v_x,v_y,v_z)$ denotes the velocity vector. The total pressure is $p^* = p + \frac{1}{2} |\mathbf{B}|^2$, where $p = \left(\gamma - 1\right)\left(E - \frac{1}{2}\rho |\mathbf{v}|^2 - \frac{1}{2}|\mathbf{B}|^2\right)$. The remaining physical quantities are defined as in the Euler equations described above. It is essential for the physical solutions of this system to maintain positive values for both density and pressure:
        \begin{align}\label{IRofMHD}
		\mathbf{U} \in G := \left\{ \mathbf{U} = (\rho, \rho v_x,\rho v_y, \rho v_z,B_y,B_z, E)^T: \rho > 0, \, p > 0 \right\}.
	\end{align}
The set $G$ is a convex invariant domain for MHD equations \eqref{MHD}; see \cite{Wu2018PPMHD,Wu2018PPDGMHD,Wu2019PPMHD}.
\end{ex}

Developing IDP schemes that keep the numerical solution within the invariant domain $G$ is essential for maintaining both physical accuracy and numerical stability. Preserving $ \mathbf{U} \in G $ is often a basic  requirement for the hyperbolicity and well-posedness of the system \eqref{HCL}, as in many cases, the Jacobian matrix $ \frac{\partial \mathbf{F}(\mathbf{U})}{\partial \mathbf{U}} $ is real diagonalizable only if $ \mathbf{U} \in G $. For example, in the Euler equations, nonphysical negative pressures or densities generated numerically can lose hyperbolicity, leading to ill-posed discrete problems, nonlinear instabilities, and eventually the breakdown of the code.

Consider the exact solution, denoted as $\mathbf{U}^{\rm RP}\left( \frac{x}{t}; \mathbf{U}_L, \mathbf{U}_R \right)$, to the Riemann problem of \eqref{HCL} with initial data
\begin{align}\label{RP:init}
	\mathbf{U}_0(x) = \begin{cases}
		\mathbf{U}_L, \quad x < 0, \\
		\mathbf{U}_R, \quad x > 0.
	\end{cases}
\end{align}
Assuming $\mathbf{U}_L, \mathbf{U}_R \in G$, we have $\mathbf{U}^{\rm RP}\left( \frac{x}{t}; \mathbf{U}_L, \mathbf{U}_R \right) \in G$ for all $x$ and $t > 0$. Let $\lambda_{\max} (\mathbf{U}_L, \mathbf{U}_R)$ represent the maximum wave speed in this Riemann problem. Then it follows that
\begin{align*}
	\mathbf{U}^{\rm RP}\left( \frac{x}{t}; \mathbf{U}_L, \mathbf{U}_R \right) & = \mathbf{U}_L, \qquad \forall x < - \lambda_{\max}  t, \\
	\mathbf{U}^{\rm RP}\left( \frac{x}{t}; \mathbf{U}_L, \mathbf{U}_R \right) & = \mathbf{U}_R, \qquad \forall x > \lambda_{\max}  t.
\end{align*}
For any $\lambda \ge \lambda_{\max}$, integrating \eqref{HCL} with \eqref{RP:init} over the spacetime domain $[-\lambda, \lambda] \times [0,1]$, and applying the divergence theorem, we obtain 
\begin{align*}
	\frac{ \mathbf{U}_L + \mathbf{U}_R }{2} - \frac{1}{2\lambda} \left(  \mathbf{F}( \mathbf{U}_R ) - \mathbf{F} (\mathbf{U}_L)  \right) = \frac{1}{2\lambda} \int_{-\lambda}^{\lambda} \mathbf{U}^{\rm RP}(x; \mathbf{U}_L, \mathbf{U}_R ) \, dx \in G,
\end{align*}
which lies in the invariant domain $G$ due to the convexity of $G$. This leads to the following proposition.

\begin{prop}[Generalized Lax--Friedrichs Splitting Property]
	If $\mathbf{U}_L, \mathbf{U}_R \in G$, there exists a suitable wave speed estimate $\lambda_{\max} (\mathbf{U}_L, \mathbf{U}_R)$ such that
	\begin{align}\label{LFsplitting}
		\frac{ \mathbf{U}_L + \mathbf{U}_R }{2} - \frac{1}{2\lambda} \left(  \mathbf{F}( \mathbf{U}_R ) - \mathbf{F} (\mathbf{U}_L)  \right) \in G, \qquad \forall \lambda \ge \lambda_{\max}. 
	\end{align}
\end{prop}

This property, often called the generalized Lax--Friedrichs splitting property \cite{Wu2017PPRHD,Wu2018PPMHD,Wu2021ppmultidimensionalrelativisticMHD}, provides a basis for the IDP property in Lax--Friedrichs-like schemes and can be viewed as an extension of the usually expected Lax--Friedrichs splitting property \cite{zhang2010positivity,wu2015high}. 

Notice that, to ensure the above property \eqref{LFsplitting}, taking $\lambda_{\max}$ as the maximum wave speed in the Riemann problem is {\em a sufficient but not a necessary} requirement. 
In practice, computing the exact maximum wave speed in the Riemann problem can be challenging. Fortunately, for many systems, it is possible to derive a simple, explicit, and easily computable IDP wave speed (still denoted as $\lambda_{\max}$ for convenience). While this estimate may not always serve as an upper bound for the maximum wave speed, it can still rigorously guarantee the validity of property \eqref{LFsplitting}. Examples include:
\begin{itemize}[leftmargin=*]
	\item For scalar conservation laws \eqref{Scalar} with a convex or concave flux $f(u)$, we can set 
$$
\lambda_{\max} = \max\{ |f'(u_L)|,~ |f'(u_R)|  \}
$$
and use monotonicity techniques to verify \eqref{LFsplitting}.
\item For the Euler equations with the invariant domain $G$ in \eqref{IRofEuler}, the IDP wave speed can be chosen as
\begin{equation}\label{eq:EulerIDPspeed}
\lambda_{\max} = \max \left\{ |v_L| + \sqrt{\frac{\gamma p_L}{\rho_L}}, |v_R| + \sqrt{\frac{\gamma p_R}{\rho_R}} \right\}.
\end{equation}
By applying the GQL approach \cite{Wu2023GQLreview}, one can rigorously prove \eqref{LFsplitting}. Although this estimated speed \eqref{eq:EulerIDPspeed} is not generally an upper bound of the maximum wave speed \cite{FastWaveGuermondPopov}, it suffices to ensure the generalized Lax--Friedrichs splitting property \eqref{LFsplitting} for $G$ in \eqref{IRofEuler} (see \cite{Wu2023GQLreview}). It is worth noting that if the minimum entropy principle \cite{ConvexLimiting2018} is incorporated into $G$, the estimated speed in \eqref{eq:EulerIDPspeed} may not always guarantee the property \eqref{LFsplitting}, and a larger wave speed (e.g., \cite{FastWaveGuermondPopov}) would be required. 
    \item For the MHD equations \eqref{MHD}, the IDP wave speed can be taken as
    \begin{align*}
    \lambda_{\max}=&\max\left\{|v_{x,L}|+c_{f,L},|v_{x,R}|+c_{f,R},v_{\rm ROE}+\max\left\{c_{f,L},c_{f,R}\right\}\right\} 
    + \frac{ |{\bf B}_L-{\bf B}_R|  }{\sqrt{\rho_L}+\sqrt{\rho_R}}
    \end{align*}
    with $v_{\rm ROE}:= \frac{\sqrt{\rho_L} v_{x,L}+\sqrt{\rho_R} v_{x,R}}{\sqrt{\rho_L}+\sqrt{\rho_R}}$ and 
    the fast magnetoacoustic wave speed given by 
    $$
    c_f=\frac{1}{\sqrt{2}}\left[ \frac{\gamma p+|\mathbf{B}|^2}{\rho}+\sqrt{\left(\frac{\gamma p+|\mathbf{B}|^2}{\rho}\right)^2-4\frac{\gamma p B_{x}^2}{\rho^2}}\right]^{\frac{1}{2}}.
    $$
    If $B_{x,L}=B_{x,R}$, then we can employ the GQL approach \cite{Wu2023GQLreview} to prove \eqref{LFsplitting}; see \cite{Wu2018PPMHD}.
\end{itemize}

\section{Review of the PAMPA Scheme}

This section provides a brief review of the original PAMPA scheme \cite{abgrall2023combination} for the 1D conservation law \eqref{HCL}. For simplicity, we use the forward Euler method for time discretization, noting that all discussions extend directly to high-order strong-stability-preserving (SSP) time discretization methods, which are convex combinations of forward Euler steps.

The PAMPA scheme \cite{abgrall2023combination} combines both conservative and non-conservative formulations of hyperbolic conservation laws \eqref{HCL}. Consider a non-conservative formulation (equivalent in smooth regions) of \eqref{HCL}:
\begin{align}\label{HCLequivalent}
	\begin{cases}
		\displaystyle
		\frac{\partial{\mathbf W} }{\partial t}+\mathbf{J}\frac{\partial{\mathbf W}}{\partial x}={\bf 0}, \quad &( x,t)\in\mathbb{R}\times \mathbb{R}^+,\\
		{\mathbf W}(x,0)={\mathbf W}_0(x), \quad & x \in \mathbb{R},
	\end{cases}
\end{align}
where $ {\mathbf W} = \Psi(\mathbf{U}) $, with $ \Psi $ being a one-to-one, continuously differentiable vector function. The Jacobian matrix $ \mathbf{J} $ is defined as
\begin{align}\label{eq:JJ} 
	\mathbf{J} = \frac{\partial \Psi(\mathbf{U})}{ \partial \mathbf{U} }  \frac{\partial {\bf F}(\mathbf{U})}{ \partial \mathbf{U} }  \left(  \frac{\partial \Psi(\mathbf{U})}{ \partial \mathbf{U} }  \right)^{-1}. 
\end{align}

For instance, a scalar conservation law \eqref{Scalar}, the mapping $\Psi$ simplifies to the identity, $\Psi(u) = u$, with the Jacobian $f'(u)$, and 
the non-conservative formulation 
is given by
$$
\frac{\partial u}{\partial t}+ f'(u) \frac{\partial u}{\partial x}=0.
$$
For the compressible Euler system 
\eqref{Eulerconservation}, 
a classical non-conservative formulation in primitive variables is 
\begin{align}\label{Eulerprim}
	\frac{\partial}{\partial t} \begin{pmatrix} \rho\\v\\p \end{pmatrix}+ \mathbf{J} \frac{\partial}{\partial x} \begin{pmatrix} \rho\\v\\p \end{pmatrix}
	= {\bf 0} \quad \text{with} \quad \mathbf{J}:=\begin{pmatrix}v&\rho&0\\0&v&\frac{1}{\rho}\\0&\gamma p&v\end{pmatrix},
\end{align}
where the Jacobian matrix $\mathbf{J}$ is real diagonalizable only if $ \mathbf{U} \in G $. 
For this system, $ \mathbf{U} = (\rho, \rho v, E)^T $ represents the conservative variables, and $ {\mathbf W} = (\rho, v, p)^T $ denotes the primitive variables. The mapping $ \Psi $ is the transformation from conservative to primitive variables.

The 1D spatial domain is divided into non-overlapping cells, denoted by $I_{j+\frac{1}{2}} = [x_j, x_{j+1}]$, where $\Delta x_{j+\frac{1}{2}} = x_{j+1} - x_j$ is the cell size, and $x_{j+\frac{1}{2}} = \frac12(x_j + x_{j+1})$ is the cell center. Similarly, the time interval is discretized into time levels $t^n$, with the time step size $\Delta t^n = t^{n+1} - t^{n}$ determined by an appropriate CFL condition.

The PAMPA scheme \cite{abgrall2023combination} simultaneously evolves the cell averages of the variables $\mathbf{U}$ and the point values ${\mathbf W}$. The degrees of freedom at time step $t^n$ are given by
\begin{align}\label{DoFs}
	\overline{\mathbf{U}}_{j+\frac{1}{2}}^n\approx \frac{1}{\Delta x_{j+\frac12}}\int_{x_j}^{x_{j+1}}\mathbf{U} (x,t^n)\, dx, \quad {\mathbf W}_j^n\approx {\mathbf W}(x_j,t^n).
\end{align}
The PAMPA scheme updates the cell averages by integrating \eqref{HCL} over each cell $I_{j+\frac{1}{2}}$, yielding
\begin{align}\label{CellEvolve}
	\overline{\mathbf{U}}_{j+\frac12}^{n+1} = \overline{\mathbf{U}}_{j+\frac12}^{n} - \frac{\Delta t^n} {\Delta x_{j+\frac12}} \Big( \mathbf{F}( \mathbf{U}_{j+1}^n ) -  \mathbf{F}( \mathbf{U}_{j}^n ) \Big),
\end{align}
where the point values $\mathbf{U}_j^n$ are computed by $\mathbf{U}_j^n = \Psi^{-1}({\mathbf W}_j^n)$. The point values of ${\mathbf W}$ are evolved according to
\begin{align}\label{pointevolve}
	{\mathbf W}_j^{n+1} = {\mathbf W}_j^{n} - \Delta t^n \left( {\bf \Phi}_{j+\frac12}^{n,-} + {\bf \Phi}_{j-\frac12}^{n,+} \right),
\end{align}
where the residuals $\Phi_{j\mp\frac12}^{n,\pm}$ can be chosen as
\begin{align}\label{wkl32}
	{\bf \Phi}_{j+\frac12}^{n,-} =\mathbf{J}^-({\mathbf W}_{j}^n)\frac{\delta_j^- {\mathbf W}^n}{\Delta x_{j+\frac{1}{2}} / 2}, 
	\quad {\bf \Phi}_{j-\frac12}^{n,+} =\mathbf{J}^+({\mathbf W}_{j}^n)\frac{\delta_{j}^{+} {\mathbf W}^n}{\Delta x_{j-\frac{1}{2}} / 2}. 
\end{align}
Here, $\mathbf{J}^{\pm}$ can be taken as the Jacobian splitting $\mathbf{J}^{\pm} = ( \mathbf{J} \pm \vert\mathbf{J}\vert)/2$ (see \cite{abgrall2023combination}) or the Local Lax--Friedrichs splitting $\mathbf{J}^{\pm} = ( \mathbf{J} \pm \alpha {\bf I} )/2$, with $\alpha$ denoting a suitable upper bound of the spectral radius of $\mathbf{J}$. In equation \eqref{wkl32}, $\delta_{j}^{\pm} {\mathbf W}^n$ is taken as 
\begin{align}\label{deltaV}
	\delta_{j}^{-} {\mathbf W}^n = -\frac{3}{2} {\mathbf W}_j^n + 2 {\mathbf W}^n_{j+\frac{1}{2}} - \frac{{\mathbf W}^n_{j+1}}{2}, \quad 
	\delta_{j}^{+} {\mathbf W}^n = \frac{{\mathbf W}^n_{j-1}}{2} - 2 {\mathbf W}^n_{j-\frac{1}{2}} + \frac{3}{2} {\mathbf W}^n_{j},
\end{align}
where ${\mathbf W}^n_{j+\frac{1}{2}}$ represents a third-order approximation to the midpoint value ${\mathbf W}(x_{j+\frac12}, t^n)$ in cell $I_{j+\frac12}$, given by 
$
{\mathbf W}_{j+\frac{1}{2}}^n = \Psi(\mathbf{U}_{j+\frac{1}{2}}^n)
$ 
with 	
\begin{align}\label{midpoint}
	\mathbf{U}_{j+\frac{1}{2}}^n  = \frac{3}{2} \overline{\mathbf{U}}_{j+\frac{1}{2}}^n - \frac{1}{4}(\mathbf{U}_{j}^n + \mathbf{U}_{j+1}^n).
\end{align}

\section{Invariant-Domain-Preserving Analysis}

This section presents a rigorous  IDP analysis of the PAMPA schemes, along with a discussion of the challenges in achieving  IDP properties when using only continuous flux. To streamline the analysis, we primarily focus on scalar conservation laws, while our findings readily extend to systems of hyperbolic conservation laws. 

We concentrate on ensuring the  IDP property of the updated cell averages during time-stepping, with point values enforced through a simple IDP limiter later.

\subsection{A Sufficient Condition for Ensuring  IDP Cell Averages in PAMPA}

The relation \eqref{midpoint} implies the following cell average decomposition (CAD):
\begin{align}\label{CAD}
	\overline{\mathbf{U}}_{j+\frac12}^{n} = \frac{1}{6} \mathbf{U}_j^{n} + \frac{4}{6} \mathbf{U}_{j+\frac12}^{n} + \frac{1}{6} \mathbf{U}_{j+1}^{n},
\end{align}
which is critical in ensuring the IDP property of the updated cell averages. 
By substituting the CAD into \eqref{CellEvolve}, we have
\begin{align}\notag
	\overline{\mathbf{U}}_{j+\frac12}^{n+1} &= \frac{1}{6} \mathbf{U}_j^{n} + \frac{4}{6} \mathbf{U}_{j+\frac12}^{n} + \frac{1}{6} \mathbf{U}_{j+1}^{n} - \frac{\Delta t^n} {\Delta x_{j+\frac12}} \Big( {\bf F}( \mathbf{U}_{j+1}^n ) - {\bf F}( \mathbf{U}_{j}^n ) \Big)\\
	\label{eq:convex}
	& = \frac{2}{3} \mathbf{U}_{j+\frac12}^{n} + \frac{1}{3} \left(  \frac{ \mathbf{U}_j^{n} + \mathbf{U}_{j+1}^{n} }{2} - \frac{6 \Delta t^n}{2 \Delta x_{j+\frac12}} \left( {\bf F}( \mathbf{U}_{j+1}^n ) - {\bf F}( \mathbf{U}_{j}^n ) \right) \right),
\end{align}
which is a convex combination of ${\bf U}_{j+\frac12}^n$ and a Lax–Friedrichs-like scheme with a time step of $6 \Delta t^n$. For scalar conservation laws, this Lax–Friedrichs-like scheme is monotone and thus IDP under the CFL condition:
\begin{align}\label{eq:CFL}
	\max\{ |f'( u_j^{n} )|, |f'( u_{j+1}^{n} )| \}\frac{\Delta t^n}{\Delta x_{j+\frac12}} \le \frac{1}{6}.
\end{align}
For general hyperbolic systems with the generalized Lax--Friedrichs splitting property \eqref{LFsplitting}, this Lax–Friedrichs-like scheme is IDP under the CFL condition:
\begin{align}\label{eq:CFL2}
	\lambda_{\max} ( {\bf U}_j^n, {\bf U}_{j+1}^n )   \frac{\Delta t^n}{\Delta x_{j+\frac12}}  \le \frac{1}{6}.
\end{align}
If $\mathbf{U}_j^{n} \in G$ for all $j$, then 
\begin{align}\label{eq:LFBP}
	\frac{ \mathbf{U}_j^{n} + \mathbf{U}_{j+1}^{n} }{2} - \frac{6 \Delta t^n}{2 \Delta x_{j+\frac12}} \left( {\bf F}( \mathbf{U}_{j+1}^n ) - {\bf F}( \mathbf{U}_{j}^n ) \right) \in G 
\end{align}
under the CFL condition \eqref{eq:CFL2}. Thus, by the convex decomposition \eqref{eq:convex} and the convexity of $G$, we conclude that
\begin{align*}
	\overline{\mathbf{U}}_{j+\frac12}^{n+1} \in G,
\end{align*}
whenever the CFL condition \eqref{eq:CFL2} is satisfied.

In summary, we obtain the following theoretical results.

\begin{thm}\label{thm:Kailiang1}
	If $\overline{\mathbf{U}}_{j+\frac12}^{n} \in G$ and the point values satisfy
	$$
	\mathbf{U}_j^n, \mathbf{U}^n_{j+\frac12}, \mathbf{U}^n_{j+1} \in G
	$$
	for every cell $j$, then under the CFL condition \eqref{eq:CFL2}, we have  
	$$\overline{\mathbf{U}}_{j+\frac12}^{n+1} \in G.$$
\end{thm}

Since the midpoint value is fully determined by relation \eqref{midpoint} (or equivalently the CAD \eqref{CAD}) using the degrees of freedom $\{\overline{\mathbf{U}}_{j+\frac12}^{n}, \mathbf{U}_j^{n}, \mathbf{U}_{j+1}^{n}\}$,  the above theorem can also be restated as follows.

\begin{thm}\label{thm:Kailiang2}
	If
	$$
	\overline{\mathbf{U}}_{j+\frac12}^{n} \in G, \qquad \mathbf{U}_j^n, \mathbf{U}^n_{j+1} \in G, \qquad \frac{3}{2} \overline{\mathbf{U}}_{j+\frac12}^{n} - \frac{1}{4} \left( \mathbf{U}_j^{n} + \mathbf{U}_{j+1}^{n} \right) \in G,
	$$
	for every cell $j$, then under the CFL condition \eqref{eq:CFL2}, we have 
	$$
	\overline{\mathbf{U}}_{j+\frac12}^{n+1} \in G. 
	$$
\end{thm}

\subsection{Challenges for  IDP Cell Averages Using Only Continuous Flux}\label{sec:challenges}

In this subsection, we explain the challenges in achieving  IDP cell averages with only continuous flux, even if all the (cell endpoint) point values are BP. 
This underscores the necessity of adopting appropriate numerical fluxes at cell interfaces to ensure the  IDP property of the updated cell averages during time-stepping.

Theorems above suggest that, to ensure IDP for the updated cell averages during time evolution, all point values---including the midpoint value $\mathbf{U}_{j+\frac12}^{n}$---must satisfy the IDP condition. However, guaranteeing $\mathbf{U}_{j+\frac12}^{n} \in G$ is nontrivial, as this value is uniquely determined by \eqref{midpoint} once $\overline{\mathbf{U}}_{j+\frac12}^{n}$, $\mathbf{U}_j^n$, and $\mathbf{U}_{j+1}^n$ are specified.

This raises a natural question:

\textbf{\em Is it possible to remove the requirement $\mathbf{U}_{j+\frac12}^{n} \in G$ from Theorem \ref{thm:Kailiang1}?}

\noindent
The answer is no, as shown in the following theorem.

\begin{thm}\label{thm:Kailiang3} 
	If 
	$$
	\overline{\mathbf{U}}_{j+\frac12}^{n},~ \mathbf{U}_j^n,~ \mathbf{U}_{j+1}^n \in G \quad \forall j,
	$$
	but $\mathbf{U}_{j+\frac12}^{n} \notin G$, then, in general, 
	$$
	\overline{\mathbf{U}}_{j+\frac12}^{n+1} \quad \text{may not necessarily belong to}~ G, 
	$$
	and there is no constant CFL number that can always ensure $\overline{\mathbf{U}}_{j+\frac12}^{n+1} \in G$.
\end{thm}

\begin{proof}
	For clarity, consider the 1D advection equation $u_t + u_x = 0$ with $G = [0, 1]$. Suppose there is a cell $I_{j+\frac12} = [x_{j}, x_{j+1}]$ such that
	$$
	\overline{u}_{j+\frac12}^{n} = 1 - \frac{2}{3} \epsilon \quad \text{with} \quad 0 < \epsilon \ll \frac{1}{4}, \quad u_j^n = 1, \quad u_{j+1}^n = 0,
	$$
	which all lie within $G = [0, 1]$. However, the midpoint value
	$$
	u_{j+\frac12}^{n} = \frac{3}{2} \overline{u}_{j+\frac12}^{n} - \frac{1}{4} \left( u_j^{n} + u_{j+1}^{n} \right) = \frac{5}{4} - \epsilon \notin G.
	$$
	The PAMPA scheme \eqref{CellEvolve} with continuous flux for updating the cell averages gives
	\begin{align*}
		\overline{u}_{j+\frac12}^{n+1} &= \overline{u}_{j+\frac12}^{n} - \frac{\Delta t^n} {\Delta x_{j+\frac12}} \Big( f( u_{j+1}^n ) - f( u_{j}^n ) \Big)\\
		&= \overline{u}_{j+\frac12}^{n} - \frac{\Delta t^n} {\Delta x_{j+\frac12}} \Big( u_{j+1}^n - u_{j}^n \Big)\\
		&= 1 - \frac{2}{3} \epsilon + \frac{\Delta t^n} {\Delta x_{j+\frac12}},
	\end{align*}
	which lies within $G$ if and only if 
	$
	\frac{\Delta t^n} {\Delta x_{j+\frac12}} \le \frac{2}{3} \epsilon.
	$ 
	However, the above counterexample can be constructed with arbitrarily small $\epsilon > 0$. Hence, in general, if $u_{j+\frac12}^{n} \notin G$, there is no constant CFL number that can always guarantee $\overline{u}_{j+\frac12}^{n+1} \in G$ in all cases.
\end{proof}

\begin{rem}
	In \cite{duan2024activefluxmethodshyperbolic,abgrall2024BPPAMPA}, two bound-preserving limiting techniques were developed to enforce the IDP property for the endpoint values $\mathbf{U}_j^n, \mathbf{U}_{j+1}^n \in G$, but these techniques do not ensure that the midpoint value $\mathbf{U}_{j+\frac12}^{n} \in G$. As a result, they cannot theoretically guarantee that $\overline{\mathbf{U}}_{j+\frac12}^{n+1} \in G$ is preserved automatically during time-stepping. To address this, an additional limiter (i.e., convex limiting) is required to enforce IDP cell averages \cite{duan2024activefluxmethodshyperbolic,abgrall2024BPPAMPA}. 
	The convex limiting blends the high-order continuous flux in PAMPA with a first-order IDP numerical flux (typically the Lax–Friedrichs flux). {\bf Note that when the convex limiting is activated, the resulting flux used to update cell averages actually becomes a blended numerical flux rather than a continuous flux.}
\end{rem}

\begin{rem}
	In \cite{duan2024activefluxmethodshyperbolic}, Duan et al.~employed a power limiter for the reconstruction polynomial in each cell. Once this limiter is applied, the reconstructed solution is no longer a polynomial, so the Simpson rule is no longer exact for its integration. Consequently, the crucial CAD (cell average decomposition) \eqref{CAD} does not hold for their approach. The CAD is essential in our theoretical proof for automatically ensuring  IDP cell averages during time-stepping.
\end{rem}

\begin{rem}\label{rem:kailiang1102}
	Theorem \ref{thm:Kailiang3} demonstrates that maintaining the  IDP property for both the endpoint and midpoint values is essential for automatically preserving  IDP cell averages during time-stepping. This raises the question:
	
	\textbf{Is it possible to construct a point-value scheme that also preserves  IDP midpoint values?} 
	
	\noindent
	In other words, can the high-order point-value scheme \eqref{pointevolve} be modified to ensure both IDP endpoint and IDP midpoint values:
	\begin{align}\label{Kailiang2}
		\mathbf{U}_j^n,~~ \mathbf{U}_{j+1}^n \in G, \qquad \frac{3}{2} \overline{\mathbf{U}}_{j+\frac12}^{n} - \frac{1}{4} \left( \mathbf{U}_j^{n} + \mathbf{U}_{j+1}^{n} \right) \in G, \qquad \forall j.
	\end{align}
	This is a challenging task if one persists in maintaining continuity of the endpoint values across cell interfaces. The conditions in \eqref{Kailiang2} are globally coupled for all cells $j$, meaning that modifying or limiting the high-order point-value scheme \eqref{pointevolve} to meet \eqref{Kailiang2} becomes a globally coupled problem for all cells.
\end{rem}

Due to the difficulties in simultaneously ensuring  IDP midpoint values while persisting in the continuity of endpoint values at cell interfaces, we will propose using appropriate numerical fluxes at cell interfaces to ensure the  IDP property of the updated cell averages.

\section{Novel Efficient Invariant-Domain-Preserving PAMPA Framework}\label{IDPscheme}

This section introduces our IDP PAMPA framework. Unlike the convex limiting approach \cite{duan2024activefluxmethodshyperbolic,abgrall2024BPPAMPA} for IDP cell averages, we first apply a local scaling limiter to enforce the IDP midpoint values, which subsequently ensure that the IDP property of the updated cell averages is {\bf preserved by proof without any post-processing limiter}. 
The unconditionally IDP scheme for updating point values is based on a novel non-conservative form of the hyperbolic system, providing an innovative approach for {\bf automatically ensuring the IDP property without any limiter}. 

Assume that all point values $\{ {\bf U}_j^n \}$ and cell averages $\{\overline {\mathbf{U}}^n_{j+\frac12}\}$ at time level $n$ lie within $G$. Our objective is to develop a scheme that ensures the updated cell averages $\overline{\mathbf{U}}_{j+\frac{1}{2}}^{n+1} \in G$ under a suitable CFL condition and that the updated point values ${\bf U}_j^{n+1} \in G$ automatically and unconditionally. 

For simplicity, we use the forward Euler method for time discretization, noting that the IDP properties remain valid for higher-order SSP time discretization methods. These methods, as convex combinations of forward Euler steps, retain the IDP properties due to the convexity of $G$.

Our IDP PAMPA scheme is designed through the following three steps.

\subsection{Step 1: Local Scaling IDP Limiter for Midpoint Values}

According to Theorem \ref{thm:Kailiang1}, in order to ensure the IDP property of the updated cell averages at the next time level $t=t^{n+1}$, we need to first enforce all the point values, including the endpoint and midpoint values, $\mathbf{U}_j^n, \mathbf{U}^n_{j+\frac12}, \mathbf{U}^n_{j+1}$ within $G$.

The local scaling bound-preserving limiter proposed by Zhang and Shu \cite{zhang2010maximum,zhang2010positivity} is adapted here. 
Assume that $\overline{\mathbf{U}}_{j+\frac12}^n \in G$, which is ensured by the IDP property of the updated cell averages in the previous time step. We aim to use $\overline{\mathbf{U}}_{j+\frac12}^n$ to limit $\mathbf{U}^n_{j+\frac12}$ such that the limited endpoint values lie within $G$.

\subsubsection{Scalar Conservation Laws}
For a scalar conservation law with $G = [U_{\min}, U_{\max}]$, the local scaling IDP limiter in each cell $j$ is defined as follows:
\begin{align}
	\label{mid}
	& \widehat{u}_{j+\frac12}^{n} := (1-\theta_{j+\frac12}) \overline{u}_{j+\frac12}^{n} + \theta_{j+\frac12} u_{j+\frac12}^n, 
\end{align}
where
\begin{align*}
	\theta_{j+\frac12} = \begin{cases}
		\displaystyle 
		\frac{  \overline{u}_{j+\frac12}^{n} - U_{\min}   }{  \overline{u}_{j+\frac12}^{n} - {u}_{j+\frac12}^{n} }, \qquad & {\rm if}~~ {u}_{j+\frac12}^{n} < U_{\min},
		\\
		\displaystyle
		\frac{ U_{\max}-  \overline{u}_{j+\frac12}^{n}    }{ {u}_{j+\frac12}^{n}- \overline{u}_{j+\frac12}^{n}  }, \qquad & {\rm if}~~ {u}_{j+\frac12}^{n} > U_{\max},
		\\
		1, \qquad & {\rm otherwise.} 
	\end{cases}
\end{align*}
Although the endpoint values $\mathbf{U}_j^n$ and $\mathbf{U}^n_{j+1}$ are already preserved in $G$ from the previous time step, 
as analyzed in Remark \ref{rem:kailiang1102}, enforcing the midpoint values in $G$ and keeping the CAD \eqref{CAD} may require sacrificing the continuity of the endpoint values.  
To ensure the CAD \eqref{CAD}, which is crucial for the IDP property of the updated cell averages, we modify the endpoint values as follows:
\begin{align}
	\label{left}
	& \text{Left endpoint value} \quad \widehat{u}_{j}^{n,R} := (1-\theta_{j+\frac12}) \overline{u}_{j+\frac12}^{n} + \theta_{j+\frac12} u_{j}^{n}, \\
	\label{right}
	& \text{Right endpoint value} \quad \widehat{u}_{j+1}^{n,L} := (1-\theta_{j+\frac12}) \overline{u}_{j+\frac12}^{n} + \theta_{j+\frac12} u_{j+1}^{n},
\end{align}
It can be verified that the limited values $\widehat{u}_{j}^{n,R}, \widehat{u}_{j+\frac12}^{n}, \widehat{u}_{j+1}^{n,L}$ all lie within $G$ and satisfy the CAD \eqref{CAD}. Moreover, this IDP limiter preserves the high-order accuracy of the PAMPA scheme, as justified in \cite{zhang2010maximum,zhang2017positivity} and confirmed by our numerical experiments. This IDP limiter is efficient, as it is applied separately and independently within each cell, enabling effective parallelization.

\subsubsection{Hyperbolic Systems}
This local scaling IDP limiter can be extended to systems. Taking the 1D compressible Euler equations \eqref{Eulerconservation} as an example, assuming that the cell averages $\overline{\mathbf{U}}_{j+\frac12}^n = (\overline{\rho}_{j+\frac12}^n, \overline{\rho v}_{j+\frac12}^n, \overline{E}_{j+\frac12}^n)$ lie within the invariant region $G$, we use each cell average $\overline{\mathbf{U}}_{j+\frac{1}{2}}^n$ to limit the midpoint value $\mathbf{U}_{j+\frac12}^n$ to $G$. This modification ensures that the resulting point values satisfy $\rho > \epsilon_\rho := \min_j\{10^{-13}, \overline{\rho}_{j+\frac12}^{n}\}$ and $p > \epsilon_p := \min_j\{10^{-13}, \overline{p}_{j+\frac12}^n \}$.

The local scaling IDP limiter in each cell $j$ is defined as follows:
\begin{align}
	\label{left2}
	&{\bf U}^{*} := (1 - \theta_{\rho, j+\frac12}) \overline{\mathbf{U}}_{j+\frac12}^{n} + \theta_{\rho,j+\frac12} {\bf U}_{j+\frac12}^n, 
	\\ 
	&\widehat{\mathbf{U}}_{j+\frac12}^{n} := (1 - \theta_{p,j+\frac12}) \overline{\mathbf{U}}_{j+\frac12}^{n} + \theta_{p,j+\frac12} \mathbf{U}^{*}, 
\end{align}
where 
$$
\theta_{\rho, j+\frac12}= \begin{cases}
	\displaystyle 
	\frac{\overline{\rho}_j^n - \epsilon_\rho}{\overline{\rho}_j^n - \rho_{j+\frac{1}{2}}^n },  & {\rm if}~~  \rho_{j+\frac{1}{2}}^n< \epsilon_\rho,
	\\
	1,   & {\rm otherwise},
\end{cases} \quad 
\theta_{p, j+\frac12}= \begin{cases}
	\displaystyle 
	\frac{p(\overline{\mathbf{U}}_j^{n}) - \epsilon_p}{p(\overline{\mathbf{U}}_j^n) - p(\mathbf{U}^{*})},  & {\rm if}~~  p(\mathbf{U}_j^{*}) < \epsilon_p,
	\\
	1,   & {\rm otherwise}.
\end{cases}
$$
To ensure the CAD \eqref{CAD} for the provable IDP property of the updated cell averages, we modify the endpoint values as follows:
\begin{align}
	& \text{Left endpoint value} \quad \widehat{\mathbf{U}}_{j}^{n,R} := (1 - \theta_{j+\frac12}) \overline{\mathbf{U}}_{j+\frac12}^{n} + \theta_{j+\frac12}\mathbf{U}_j^{n}, \\ 
	& \text{Right endpoint value} \quad \widehat{\mathbf{U}}_{j+1}^{n,L} := (1 - \theta_{j+\frac12}) \overline{\mathbf{U}}_{j+\frac12}^{n} + \theta_{j+\frac12} \mathbf{U}_{j+1}^{n}, 
\end{align}
where $\theta_{j+\frac12} = \theta_{\rho, j+\frac12} \theta_{p, j+\frac12}$. It can be proven that the limited point values of the conservative variables still satisfy the CAD:
\begin{align}\label{CAD2}
	\overline{\mathbf{U}}_{j+\frac12}^{n} = \frac{1}{6} \widehat{ \mathbf{U}}_j^{n,R} + \frac{4}{6} \widehat{ \mathbf{U}}_{j+\frac12}^{n} + \frac{1}{6} \widehat{\mathbf{U}}_{j+1}^{n,L},
\end{align}
which is crucial for proving the IDP property of the updated cell averages later.

\subsection{Step 2: Provably IDP Scheme for Updating Cell Averages without Any Extra Limiter}

When the local scaling IDP limiter is activated in Step 1 (i.e., $\theta_{j+\frac12} < 1$), each cell interface may yield two distinct point values, i.e., $\widehat {\bf U}_{j}^{n,L} \neq \widehat {\bf U}_{j}^{n,R}$. This, coupled with the challenge of maintaining the IDP property of updated cell averages using a continuous flux (as discussed in Section \ref{sec:challenges}), leads us to replace the continuous flux ${\bf F}({\bf U}_{j}^n)$ with a numerical flux $\widehat {\bf F} ( \widehat {\bf U}{j}^{n,L}, \widehat {\bf U}_{j}^{n,R} )$.

If the IDP limiter is not activated near the interface $x_j$, then $\widehat {\bf U}_{j}^{n,L} = \widehat {\bf U}_{j}^{n,R}$, and by consistency, the numerical flux $\widehat {\bf F} ( \widehat {\bf U}_{j}^{n,L}, \widehat {\bf U}_{j}^{n,R} )$ will automatically reduce to the continuous flux ${\bf F}({\bf U}_{j}^n)$.

To achieve the IDP property for the updated cell averages, we proceed with the modified point values in a manner consistent with classical finite volume methods:
\begin{align}\label{cellmodifiedscalar}
	\overline{\bf U}_{j+\frac12}^{n+1} = \overline{\bf U}_{j+\frac12}^{n} - \frac{\Delta t^n}{\Delta x_{j+\frac12}} \left( \widehat{\bf F}( \widehat{\bf U}_{j+1}^{n,L}, \widehat{\bf U}_{j+1}^{n,R} ) - \widehat{\bf F}( \widehat{\bf U}_{j}^{n,L}, \widehat{\bf U}_{j}^{n,R} ) \right),
\end{align}
where $\widehat{\bf F}({\bf U}^L, {\bf U}^R)$ is an IDP numerical flux satisfying consistency $\widehat{\bf F}({\bf U}, {\bf U}) = {\bf F}({\bf U})$. 

A numerical flux is called IDP if its corresponding 1D three-point first-order scheme is IDP, i.e., for any ${\bf U}_1, {\bf U}_2, {\bf U}_3 \in G$, it holds that
\begin{align}\label{eq:IDPflux_1storder}
	{\bf U}_2 - \frac{\Delta t}{\Delta x} \left( \widehat{\bf F}({\bf U}_2, {\bf U}_3) - \widehat{\bf F}({\bf U}_1, {\bf U}_2) \right) \in G,
\end{align}
under a suitable CFL condition $\lambda_{\max} \Delta t \le c_0 \Delta x$, where $\lambda_{\max}$ denotes the maximum characteristic speed, and $c_0$ is the maximum allowable CFL number for the 1D first-order scheme to be IDP. For example, typically $c_0 = 1$ for the local Lax--Friedrichs flux. There are many suitable IDP numerical fluxes, such as the Lax--Friedrichs flux, Godunov flux, and Harten--Lax--van Leer flux. In our numerical experiments, we simply use the local Lax--Friedrichs flux.

Using a standard convex decomposition technique \cite{zhang2010maximum}, 
it is straightforward to verify that the updated cell averages given by the modified PAMPA scheme \eqref{cellmodifiedscalar} are IDP, i.e., $\overline{\mathbf{U}}_{j+\frac12}^{n+1} \in G$, as long as the CFL number is less than or equal to $\frac{1}{6}$.

\begin{thm}
	The cell averages $\overline{\mathbf{U}}_{j+\frac12}^{n+1}$, evolved by the modified PAMPA scheme \eqref{cellmodifiedscalar} with an IDP numerical flux, are preserved within the invariant domain $G$, under the CFL condition 
	\begin{align}\label{eq:PAMPA_CFL}
		\lambda_{\max}  \frac{\Delta t^n}{\Delta x_{j+\frac12}} \le \frac{1}{6}.
	\end{align}
\end{thm}

\begin{proof}
	The local scaling IDP limiter ensures that $\widehat{\mathbf{U}}_j^{n,R}, \widehat{\mathbf{U}}_{j+\frac12}^{n}, \widehat{\mathbf{U}}_{j+1}^{n,L} \in G$, and moreover, retains the CAD \eqref{CAD2}. Substituting the CAD \eqref{CAD2} into the modified PAMPA scheme \eqref{cellmodifiedscalar}, we obtain 
	\begin{align*}
		\overline{\mathbf{U}}_{j+\frac12}^{n+1} &= \overline{\mathbf{U}}_{j+\frac12}^n - \frac{\Delta t^n}{\Delta x_{j+\frac12}} \Big( \widehat{\bf F}( \widehat{\mathbf{U}}_{j+1}^{n,L}, \widehat{\mathbf{U}}_{j+1}^{n,R} ) - \widehat{\bf F}( \widehat{\mathbf{U}}_{j}^{n,L}, \widehat{\mathbf{U}}_{j}^{n,R} ) \Big) \\
		&= \frac{1}{6} \widehat{\mathbf{U}}_j^{n,R} + \frac{4}{6} \widehat{\mathbf{U}}_{j+\frac12}^{n} + \frac{1}{6} \widehat{\mathbf{U}}_{j+1}^{n,L} - \frac{\Delta t^n}{\Delta x_{j+\frac12}} \Big( \widehat{\bf F}( \widehat{\mathbf{U}}_{j+1}^{n,L}, \widehat{\mathbf{U}}_{j+1}^{n,R} ) - \widehat{\bf F}( \widehat{\mathbf{U}}_{j}^{n,L}, \widehat{\mathbf{U}}_{j}^{n,R} ) \Big) \\
		&= \frac{1}{6} \widetilde{\mathbf{U}}_{j}^{n,R} + \frac{1}{6} \widetilde{\mathbf{U}}_{j+1}^{n,L} + \frac{4}{6} \widehat{\mathbf{U}}_{j+\frac12}^{n},
	\end{align*}
	with 
	\begin{align}\label{eq:AA}
		\widetilde{\mathbf{U}}_{j}^{n,R} :=  \widehat{\mathbf{U}}_{j}^{n,R} - 6 \frac{\Delta t^n}{\Delta x_{j+\frac12}} \Big( \widehat{\bf F}( \widehat{\mathbf{U}}_{j}^{n,R}, \widehat{\mathbf{U}}_{j+1}^{n,L} ) - \widehat{\bf F}( \widehat{\mathbf{U}}_{j}^{n,L}, \widehat{\mathbf{U}}_{j}^{n,R} ) \Big),
	\end{align}
	and
	\begin{align}\label{eq:BB}
		\widetilde{\mathbf{U}}_{j+1}^{n,L} := \widehat{\mathbf{U}}_{j+1}^{n,L} - 6 \frac{\Delta t^n}{\Delta x_{j+\frac12}} \Big( \widehat{\bf F}( \widehat{\mathbf{U}}_{j+1}^{n,L}, \widehat{\mathbf{U}}_{j+1}^{n,R} ) - \widehat{\bf F}( \widehat{\mathbf{U}}_{j}^{n,R}, \widehat{\mathbf{U}}_{j+1}^{n,L} ) \Big).
	\end{align}
	Note that \eqref{eq:AA} and \eqref{eq:BB} are formally the first-order schemes with the IDP numerical flux and a larger time step size of $6 \Delta t^n$. Thanks to the IDP property of the adopted numerical flux \eqref{eq:IDPflux_1storder}, we have 
	$$
	\widetilde{\mathbf{U}}_{j}^{n,R},~~\widetilde{\mathbf{U}}_{j+1}^{n,L} \in G,
	$$
	under the CFL condition \eqref{eq:PAMPA_CFL}. Therefore, the updated cell average $\overline{\mathbf{U}}_{j+\frac12}^{n+1}$ remains in $G$ because it is a convex combination of elements in $G$.
\end{proof}

\begin{rem}
The IDP limiter, when activated in a cell, modifies the continuous point value at the interface into two distinct values, automatically detecting that the interface may represent a discontinuity (e.g., shock). If the local scaling IDP limiter is not activated in cells $I_{j-\frac12}$ and $I_{j+\frac12}$, then $\widehat{\mathbf{U}}_{j}^{n,L} = \widehat{ \mathbf{U}}_{j}^{n,R} = \mathbf{U}_{j}^{n}$, and the scheme reduces to the original cell-average scheme in PAMPA with continuous flux. Thus, the numerical flux is applied in regions where physical bounds are violated, typically near strong shocks. {\bf The numerical viscosity in the flux depends on the jumps of the limited endpoint values: the more severe the bound violation, the stronger the limiter and the greater the numerical viscosity.}
\end{rem}

\subsection{Step 3: Unconditionally Limiter-Free IDP Scheme for Updating Point Values}

Thanks to the powerful of the PAMPA framework in combining conservative and non-conservative formulations of the same equations, 
this framework provide a great flexibility for us to use various non-conservative formulations for scheme of the point values. 
Typically, we can find a suitable set of variables ${\bf W}$, such that 
the range of the mapping $ {\mathbf W} = \Psi(\mathbf{U})$ is $\mathbb R^d$, i.e, $\Psi(G) = \mathbb R^d$. 
In the words, for any ${\bf W} \in \mathbb R^d$, we always have ${\mathbf U} = \Psi^{-1}(\mathbf{W}) \in G$. 
This automatically ensures the IDP property of the point values. 

To be more specific, let us take the compressible Euler equations as an example for illustration. 
For the Euler system, there are many choices for the variables \(\mathbf{W}\). 
Instead of using the standard primitive variables, we choose 
a density-like quantity, the fluid velocity, and the specific entropy:
\begin{align}
	\mathbf{W} : = \left(  q, v, s  \right)^T  \qquad \text{with} \quad q:=\ln \left( e^{\frac{\rho}{\rho_{\rm ref}}} -1 \right), \quad s = \ln p - \gamma \ln \rho,  
\end{align}
where $\gamma$ is the adiabatic index, and $\rho_{\rm ref}$ is a reference density taken as one in our computations (it can also be adjusted to accommodate the scale of density). The associated non-conservative formulation in these variables is given by
\begin{align}\label{Eulerentropy0}
	\frac{\partial \mathbf{W}}{\partial t}  + \mathbf{J} \frac{\partial \mathbf{W}   }{\partial x} 
	= {\bf 0} \quad \text{with} \quad \mathbf{J}:=\begin{pmatrix}v&   \frac{ e^{\frac{\rho}{\rho_{\rm ref}}-q}\rho }{\rho_{\rm ref}}  &0\\  
   \frac{ \gamma \rho_{\rm ref} p  }{e^{\frac{\rho}{\rho_{\rm ref}}-q}\rho^2  }     &v& \frac{p}{\rho}  \\0&0&v\end{pmatrix}.
\end{align}
For any $\mathbf{W} \in \mathbb R^3$, we always have 
$$
\rho = \rho_{\rm ref} \ln\left( e^q + 1 \right) > 0, \qquad p = \rho^\gamma e^s > 0. 
$$
This reformulation of density is motivated by the Softplus function $\ln (1+e^x)$ in machine learning, which is often used as an activation function in neural networks.

The point values are evolved by 
\begin{align}\label{eq:pointPAMPA}
	\mathbf{W}_{j}^{n+1} = \mathbf{W}_j^{n} - \Delta t^n (\mathbf{\Phi}_{j+\frac{1}{2}}^{n,-} + \mathbf{\Phi}_{j-\frac{1}{2}}^{n,+}),   \qquad   {\mathbf U}_{j}^{n+1} = \Psi^{-1}(\mathbf{W}_{j}^{n+1})
\end{align}
with
\begin{align*}
	&{\bf \Phi}_{j+\frac12}^{n,-} = \mathbf{J}_j^{n,-} \frac{\delta_j^- {\mathbf W}^n}{\Delta x_{j+\frac12} / 2}, \qquad {\bf \Phi}_{j-\frac12}^{n,+} = \mathbf{J}_{j}^{n,+} \frac{\delta_{j}^{+} {\mathbf W}^n}{\Delta x_{j-\frac12} / 2},
	\\
	&	\delta_{j}^{-} {\mathbf W}^n = -\frac{3}{2} {\mathbf W}_j^n + 2 \widehat {\mathbf W}^n_{j+\frac{1}{2}} - \frac{{\mathbf W}^n_{j+1}}{2}, \qquad 
	\delta_{j}^{+} {\mathbf W}^n = \frac{{\mathbf W}^n_{j-1}}{2} - 2 \widehat {\mathbf W}^n_{j-\frac{1}{2}} + \frac{3}{2} {\mathbf W}^n_{j}.
\end{align*}
where $\widehat {\mathbf W}^n_{j-\frac{1}{2}} = \Psi ( \widehat {\mathbf U}^n_{j-\frac{1}{2}} )$ is the modified midpoint values, and 
\begin{align*}
	\mathbf{J}_j^{n,\pm} = \frac{ \mathbf{J}( {\mathbf U}_j^{n}) \pm \alpha_j {\bf I}}{2},
\end{align*}
with $\alpha_j$ being an appropriate upper bound for the maximum wave speed, defined as the maximum of the spectral radii of the Jacobian matrices $\mathbf{J}( {\mathbf U}_j^{n})$, $\mathbf{J}(\widehat {\mathbf U}_{j-\frac12}^{n})$, and $\mathbf{J}(\widehat {\mathbf U}_{j+\frac12}^{n})$. To ensure the robustness of the scheme \eqref{eq:pointPAMPA} and the well-posedness of the discrete problem, we must ensure the fundamental hyperbolicity, i.e., that the Jacobian matrix $\mathbf{J}({\mathbf U}_j^{n})$ is real diagonalizable. Since $\mathbf{J}$ is similar to $\frac{\partial \mathbf{F}(\mathbf{U})}{\partial \mathbf{U}}$, which is real diagonalizable if ${\mathbf U}_j^{n} \in G$, it follows that $\mathbf{J}({\mathbf U}_j^{n})$ is real diagonalizable.

For the scalar conservation law \eqref{IRofEuler} with the invariant domain $G=[U_{\min}, U_{\max}]$, we consider the following mapping: 
\begin{align}\label{eq:scalarmapping}
	u = \Psi^{-1}(w) =\left( U_{\max}-U_{\min} \right) \min\{ {\rm ReLU}\left(w \right),1 \}+U_{\min}, \qquad \forall w \in \mathbb R, 
\end{align}
where ReLU is an activation function in machine leaning defined by
\begin{align*}
    {\rm ReLU}(w)=\max\{0,w\}, \qquad  \forall w \in \mathbb R.
\end{align*}
This mapping ensures that, for any $w \in \mathbb R$, we always have $u \in [U_{\min}, U_{\max}]$. 
It can map the variable \(w\) in \([0,1]\) to the variable \(u\) injectively, though the entire mapping is not globally one-to-one. In practice, we define the inverse of mapping \eqref{eq:scalarmapping} as
$$w = \Psi(u) = 
\begin{cases} 
1, & {\rm if} ~~ u = U_{\max}, \\ 
\frac{u - U_{\min}}{U_{\max} - U_{\min}}, & {\rm if} ~~ U_{\min} < u < U_{\max}, \\ 
0, & {\rm if} ~~ u = U_{\min}.
\end{cases}$$
The non-conservative formulation of \eqref{IRofEuler} in terms of the new variable $w$ is given by 
$$
\frac{\partial w}{\partial t}+ f'( \Psi^{-1}(w) ) \frac{\partial w}{\partial x}=0.
$$
In the scalar case, 
the modified PAMPA scheme for point values reduces to  
\begin{align}\label{scalarpoint}
	w_{j}^{n+1} = w_j^{n} - \Delta t^n \left( \Phi_{j+\frac{1}{2}}^{n,-} +  \Phi_{j-\frac{1}{2}}^{n,+} \right), \qquad u_{j}^{n+1} = \Psi^{-1}(w_{j}^{n+1}). 
\end{align}
where
$$
\Phi_{j+\frac12}^{n,-} = f_j^{n,-} \frac{\delta_j^- w^n}{\Delta x_{j+\frac12} / 2}, \qquad \Phi_{j-\frac12}^{n,+} = f_{j}^{n,+} \frac{\delta_{j}^{+} w^n}{\Delta x_{j-\frac12} / 2},
$$
with the midpoint value 
in 
$\delta_j^\pm w^n$ replaced with the IDP midpoint value $\widehat{w}_{j+\frac12}^n = \Psi ( \widehat{u}_{j+\frac12}^n )$, 
$
f_j^{n,\pm} = \frac12 \left(f'({u}_{j}^{n}) \pm \alpha_{j}\right)
$, and $\alpha_j := \max\{ |f'({u}_{j}^{n})|,|f'(\widehat {u}_{j-\frac12}^{n})|,|f'(\widehat {u}_{j+\frac12}^{n})| \}$.

\section{Suppression of Spurious Oscillations}

Even with the IDP technique, the numerical solutions of PAMPA may exhibit spurious oscillations near strong discontinuities. To achieve essentially non-oscillatory results, techniques commonly used in discontinuous Galerkin methods can be effectively applied. These include the total variation diminishing (TVD) or total variation bounded (TVB) limiters, the oscillation-eliminating (OE) procedure \cite{peng2024oedg}, and the monotonicity-preserving (MP) limiter \cite{SURESH199783}.

In this paper, we suggest two techniques: a new OE procedure and the MP limiter \cite{SURESH199783}. Either technique can suppress spurious oscillations and is applied directly before the local scaling IDP limiter.

\subsection{New OE Procedure for PAMPA Scheme}

The primary advantage of the OE procedure is its ability to retain the compactness and high-order accuracy of the PAMPA scheme. We define
\begin{align}\label{OEtheta}
    \theta_{j+\frac12}^{\text{OE}} = \exp \left( -\frac{\beta_{j+\frac12} \Delta t^n}{\Delta x_{j+\frac{1}{2}}} \sigma_{j+\frac12} \right),
\end{align}
where $\beta_{j+\frac12}$ is a suitable estimate of the local maximum wave speed on $I_{j+\frac12}$, and 
\begin{align}\label{OEsigma}
    \sigma_{j+\frac12} = \begin{cases}
    0,& {\rm if} ~~ \frac{S_R}{S_R-S_L} d_{j+\frac12}^{n,L} - \frac{S_L}{S_R-S_L} d_{j+\frac12}^{n,R}=0,\\
    \frac{\frac{S_R}{S_R-S_L}\eta_{j+\frac12}^{n,L} - \frac{S_L}{S_R-S_L}\eta_{j+\frac12}^{n,R}}{\frac{S_R}{S_R-S_L} d_{j+\frac12}^{n,L} - \frac{S_L}{S_R-S_L} d_{j+\frac12}^{n,R}}. &{\rm otherwise.}
    \end{cases}
\end{align}
The estimated wave speeds are
$ 
S_L = \min \{ \lambda_{\min}(\overline{\bf U}_{j-\frac12}^n), \lambda_{\min}(\overline{\bf U}_{j+\frac12}^n), \lambda_{\min}(\overline{\bf U}_{j+\frac32}^n), 0 \}$ and $
S_R = \max \{ \lambda_{\max}(\overline{\bf U}_{j-\frac12}^n), \lambda_{\max}(\overline{\bf U}_{j+\frac12}^n), \lambda_{\max}(\overline{\bf U}_{j+\frac32}^n), 0 \}$ 
and 
\begin{align*}
\eta_{j+\frac12}^{n,L} &=  \int_{I_{j+\frac12}} \left( p_{j+\frac12}^n(x) - p_{j+\frac12}^{n,L}(x) \right)^2 dx + \frac{1}{3} \Delta x_{j+\frac12}^{5} \left( \partial_{xx} p_{j+\frac12}^n(x) - \partial_{xx} p_{j+\frac12}^{n,L}(x) \right)^2, 
\\
\eta_{j+\frac12}^{n,R} &= \int_{I_{j+\frac12}} \left( p_{j+\frac12}^n(x) - p_{j+\frac12}^{n,R}(x) \right)^2 dx + \frac{1}{3} \Delta x_{j+\frac12}^{5}  \left( \partial_{xx} p_{j+\frac12}^n(x) - \partial_{xx} p_{j+\frac12}^{n,R}(x) \right)^2 , 
\\
d_{j+\frac12}^{n,L} &= \int_{I_{j+\frac12}} \left( p_{j+\frac12}^{n,L}(x) - \overline{\mathbf{U}}^n_{j+\frac12} \right)^2 +\left( p_{j+\frac12}^{n}(x) - \overline{\mathbf{U}}^n_{j+\frac12} \right)^2 dx + \frac{1}{3} \Delta x_{j+\frac12}^{5} (\partial_{xx} p_{j+\frac12}^{n,L}(x))^2, 
\\
d_{j+\frac12}^{n,R} &= \int_{I_{j+\frac12}} \left( p_{j+\frac12}^{n,R}(x) - \overline{\mathbf{U}}^n_{j+\frac12} \right)^2 +\left( p_{j+\frac12}^{n}(x) - \overline{\mathbf{U}}^n_{j+\frac12} \right)^2 dx + \frac{1}{3} \Delta x_{j+\frac12}^{5} (\partial_{xx} p_{j+\frac12}^{n,R}(x))^2 .
\end{align*}
Here, $p_{j+\frac12}^{n,L}(x)$ and $p_{j+\frac12}^{n,R}(x)$ are obtained by extending the reconstructed polynomials in $I_{j-\frac12}$ and $I_{j+\frac32}$, respectively. The point values in each cell $j$ after applying the OE procedure are updated as follows:
\begin{align}\label{OEprocedure}
    & \text{Left endpoint value:} \quad \widetilde{\mathbf{U}}_{j}^{n,R} := (1 - \theta_{j+\frac12}^{\text{OE}}) \overline{\mathbf{U}}_{j+\frac12}^{n} + \theta_{j+\frac12}^{\text{OE}} \mathbf{U}_j^{n}, \\
    & \text{Right endpoint value:} \quad \widetilde{\mathbf{U}}_{j+1}^{n,L} := (1 - \theta_{j+\frac12}^{\text{OE}}) \overline{\mathbf{U}}_{j+\frac12}^{n} + \theta_{j+\frac12}^{\text{OE}} \mathbf{U}_{j+1}^{n}.
\end{align}

\subsection{MP Limiter for PAMPA Scheme}

The MP limiter is advantageous for its high resolution and ability to retain the original high-order accuracy, whether applied globally to all cells or adaptively to troubled cells identified by a shock indicator. However, it loses the compactness of the PAMPA scheme. The MP limiter essentially serves as the estimate and enforcement of local bounds. 
For scalar conservation laws, the MP limiter modifies the point values as follows:
\begin{align}
    \widetilde{u}_j^{n,L} = \text{median}\left( u_j^n, u^{n,\text{min}}_j, u^{n,\text{max}}_j \right), \label{mplimiter}
\end{align}
where
\begin{equation}\label{eq:6.6}
\begin{aligned}
    u^{n,\text{min}}_j &= \max \left[ \min\left( \overline{u}_{j-\frac12}^n, \overline{u}_{j-\frac32}^n, u^{n,\text{MD}}_j \right), \min\left( \overline{u}_{j-\frac12}^n, u^{n,\text{UL}}_j, u^{n,\text{LC}}_j \right) \right], \\
    u^{n,\text{max}}_j &= \min \left[ \max\left( \overline{u}_{j-\frac12}^n, \overline{u}_{j-\frac32}^n, u^{n,\text{MD}}_j \right), \max\left( \overline{u}_{j-\frac12}^n, u^{n,\text{UL}}_j, u^{n,\text{LC}}_j \right) \right].
\end{aligned}
\end{equation}
In \eqref{eq:6.6}, $u^{n,\text{MD}}_{j} = \frac{1}{2}\left(\overline{u}_{j-\frac12}^n + \overline{u}_{j+\frac12}^n + d^{n,\text{M4}}_{j}\right)$, $u^{n,\text{UL}}_{j} = \overline{u}_{j-\frac12}^n + \alpha \left( \overline{u}_{j-\frac12}^n -  \overline{u}_{j-\frac32}^n\right)$, and $u^{n,\text{LC}}_{j} =  \overline{u}_{j-\frac12}^n + \frac{1}{2} \left( \overline{u}_{j-\frac12}^n - \overline{u}_{j-\frac32}^n\right) + \frac{\beta}{3} d^{n,\text{M4}}_{j-1}$. The quantity $d^{n,\text{M4}}_{j}$ is defined as 
$$d^{n,\text{M4}}_{j}= {\rm minmod}\left(4d_{j-\frac12}^n - d_{j+\frac12}^n, 4d_{j+\frac12}^n - d_{j-\frac12}^n, d_{j-\frac12}^n, d_{j+\frac12}^n\right)$$ with $d_{j+\frac12}^n = \overline{u}_{j+\frac32}^n - 2\overline{u}_{j+\frac12}^n + \overline{u}_{j-\frac12}^n$. 
The point value $\widetilde{u}_{j}^{n,R}$ is modified symmetrically. The parameters $\alpha = 2$ and $\beta = 4$ are used in the numerical tests, as suggested in \cite{SURESH199783}. In each cell $I_{j+\frac12}$, we compute the modified endpoint values $\widetilde{u}_{j}^{n,R}$, $\widetilde{u}_{j+1}^{n,L}$, and the modified midpoint $\widetilde{u}_{j+\frac12}^n = \frac{3}{2} \overline{u}_{j+\frac12}^n - \frac{1}{4} (\widetilde{u}_{j}^{n,R} + \widetilde{u}_{j+1}^{n,L})$. For hyperbolic systems, the MP limiter is applied independently to each component of the variables ${\mathbf W}$.

\subsection{Combining IDP and Oscillation Suppression Techniques}

The OE or MP procedure  is applied before the IDP limiter to suppress spurious oscillations. Specifically, the point values $\{ u_j^n, u^n_{j+\frac12}, u^n_{j+1} \}$ in \eqref{mid}--\eqref{right} are replaced with the OE- or MP-modified values $\{ \widetilde{u}_j^{n,R}, \widetilde{u}^n_{j+\frac12}, \widetilde{u}_{j+1}^{n,L} \}$. The resulting PAMPA scheme is both IDP and effectively suppresses spurious oscillations without compromising accuracy.

\section{Numerical Examples}

This section presents several 1D numerical examples to validate the accuracy and robustness of the proposed IDP PAMPA scheme. The examples cover the advection equation, Burgers' equation, the compressible Euler equations, and the ideal MHD equations. Unless otherwise specified, the third-order SSP multi-step method is used for scalar equations, while the third-order SSP Runge--Kutta method is applied to systems of conservation laws. The CFL number is set to 0.1.

 \subsection{Advection equation}

This subsection presents two examples of the advection equation $u_t + u_x = 0$ with periodic boundary conditions.

\begin{ex}[Smooth problem]
\rm 
The first example evaluates the impact of the IDP limiter and oscillation suppression techniques on the accuracy of the PAMPA scheme. The initial condition is chosen as $u_0(x) = 1 + \sin^4(2\pi x)$ on the domain $[0,1]$. Table \ref{advorderone} lists the $l^1$ errors and corresponding convergence rates for the numerical solution at $t=1$, obtained using PAMPA scheme with different limiters on $N$ uniform cells. The results demonstrate that the IDP limiter and oscillation suppression techniques maintain the desired third-order accuracy.
\end{ex}

\begin{table}[ht]
\centering
\caption{Errors and convergence rates for cell averages (top) and point values (bottom) computed with $N$ uniform cells.}\label{advorderone}
\begin{tabular}{lcccccccc}
\toprule
 &   & \multicolumn{2}{c}{{IDP PAMPA}} & \multicolumn{2}{c}{{IDP PAMPA with OE}} & \multicolumn{2}{c}{{IDP PAMPA with MP}} \\
\cmidrule(lr){3-4} \cmidrule(lr){5-6}  \cmidrule(lr){7-8} 
 &$N$ & Error & Order& Error & Order& Error & Order \\
\midrule
\multirow{4}{*}{} 
& 20   & 2.66e-2 & - & 4.12e-2 & - & 2.68e-2 & - \\
& 40  & 5.60e-3 & 2.25 & 6.60e-3 & 2.63& 5.40e-3 & 2.30 \\
& 80  & 8.62e-4 & 2.69 &8.76e-4 & 2.92 & 8.57e-4 & 2.67 \\
& 160  & 1.09e-4 & 2.98 & 1.09-4 & 3.00& 1.09e-4 & 2.97 \\
& 320  & 1.37e-5 &3.00 & 1.37e-5 & 3.00 & 1.37e-5 &3.00\\
& 640  & 1.71e-6 &3.00 & 1.71e-6 &3.00 & 1.71e-6 &3.00\\
& 1280 & 2.14e-7&3.00 & 2.14e-7 &3.00 & 2.14e-7&3.00\\
\midrule
\multirow{4}{*}{} 
& 20   & 2.94e-2 & - & 4.34e-2 & - & 2.93e-2 & - \\
& 40  & 5.50e-3 & 2.42 & 6.50e-3 & 2.73 & 5.30e-3 & 2.47 \\
& 80  & 8.49e-4 & 2.69 & 8.73e-4 & 2.90 & 8.39e-4 & 2.65 \\
& 160 & 1.09e-4 & 2.96 & 1.09e-4 & 3.00& 1.09e-4 & 2.95 \\
& 320  & 1.36e-5 &3.00 & 1.36e-5 & 3.00 & 1.36e-5 &3.00\\
& 640 & 1.71e-6 &3.00& 1.71e-6 &3.00 & 1.71e-6 &3.00\\
& 1280  & 2.14e-7&3.00 & 2.14e-7 &3.00 & 2.14e-7&3.00\\

\bottomrule
\end{tabular}
\end{table}

\begin{ex}[Jiang--Shu problem]
\rm
The initial conditions are given by
$$
u_0(x) = \begin{cases}
\frac{1}{6}\left(G_1(x, \beta, z-\delta) + G_1(x, \beta, z+\delta) + 4 G_1(x, \beta, z)\right), & \text{if } -0.8 \leq x \leq -0.6, \\
1, & \text{if } -0.4 \leq x \leq -0.2, \\
1 - |10(x - 0.1)|, & \text{if } 0 \leq x \leq 0.2, \\
\frac{1}{6}\left(G_2(x, \alpha, a-\delta) + G_2(x, \alpha, a+\delta) + 4 G_2(x, \alpha, a)\right), & \text{if } 0.4 \leq x \leq 0.6, \\
0, & \text{otherwise,}
\end{cases}
$$
for \(x \in [-1, 1]\), where 
$$
G_1(x, \beta, z) = e^{-\beta(x - z)^2}, \quad G_2(x, \alpha, a) = \sqrt{\max \left(1 - \alpha^2(x - a)^2, 0\right)},
$$
and the constants are \(a = 0.5\), \(z = -0.7\), \(\delta = 0.005\), \(\alpha = 10\), and \(\beta = \frac{\ln 2}{36 \delta^2}\). This problem is solved for one period, i.e., until \(t = 2\). The results of the PAMPA scheme, both with and without the IDP technique, are compared in Figure \ref{adv1fig}, computed with 400 uniform cells. 
When the IDP limiter is applied, the numerical solution consistently remains strictly within the bounds \(G = [0, 1]\), and spurious oscillations are effectively suppressed. This demonstrates the robustness of the IDP PAMPA scheme in maintaining both stability and adherence to the maximum principle, highlighting its effectiveness in accurately capturing multiple wave phenomena.
\end{ex}

\begin{figure}[h]
    \centering
    \begin{subfigure}[b]{0.48\textwidth} 
        \centering
        \includegraphics[width=\textwidth]{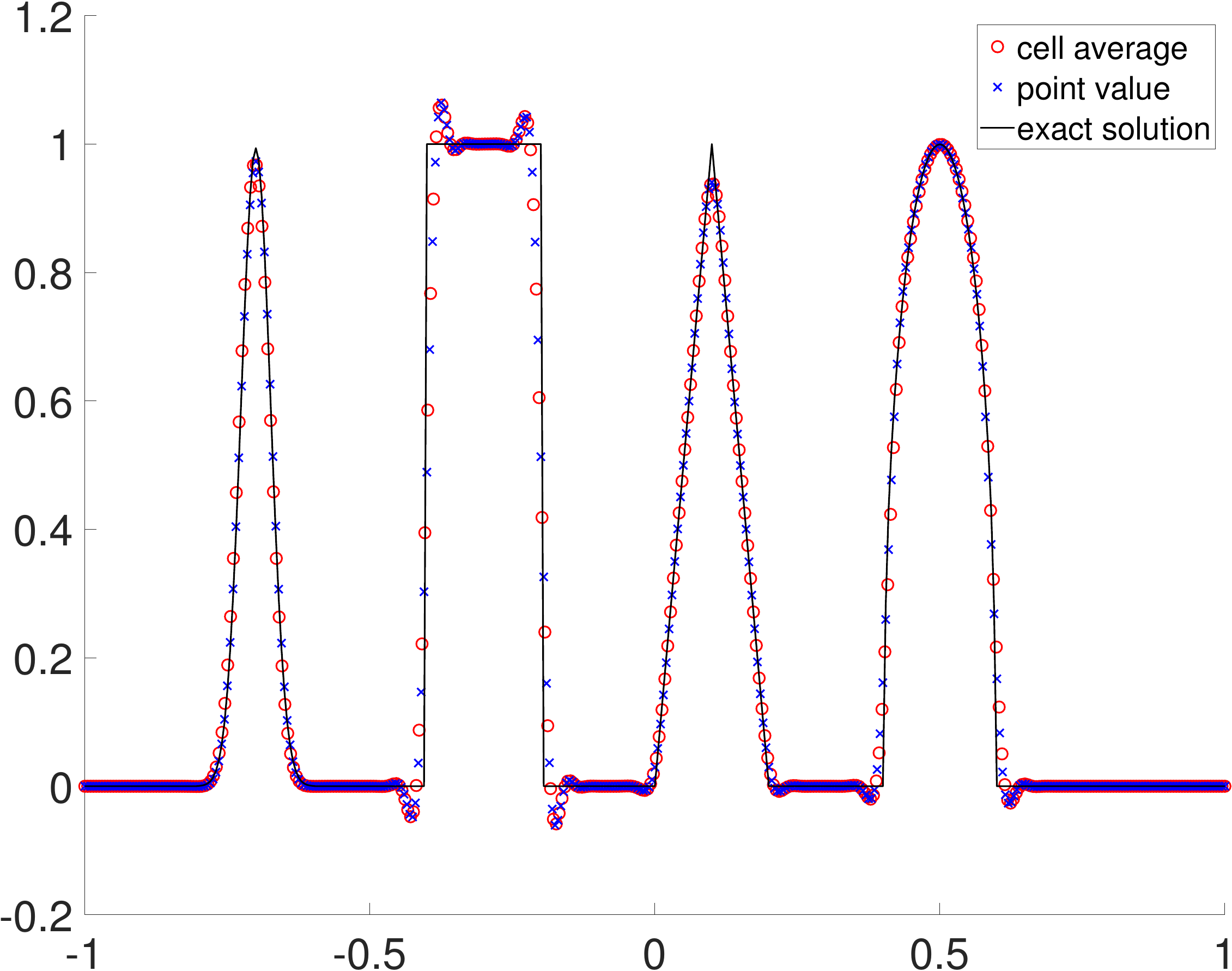}
        \caption{Original PAMPA without IDP}
    \end{subfigure}
    \hfill
    \begin{subfigure}[b]{0.48\textwidth}
        \centering
        \includegraphics[width=\textwidth]{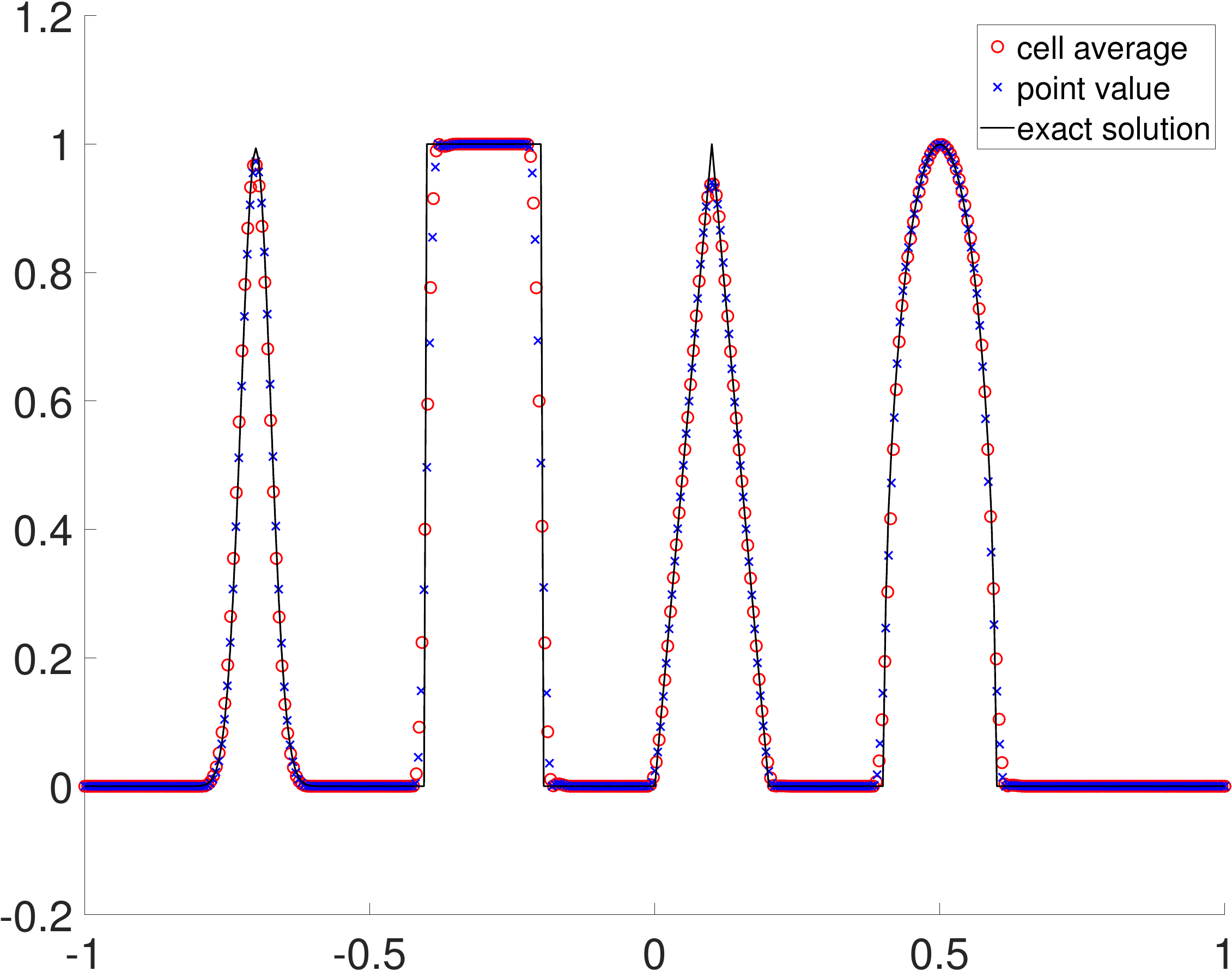}
        \caption{IDP PAMPA}
    \end{subfigure}
    \caption{Example 7.2: Numerical solutions computed by PAMPA schemes with and without IDP technique. }
    \label{adv1fig}
\end{figure}

\subsection{1D Burgers' Equation}
This subsection considers the 1D Burgers' equation 
$
u_t + \left(\frac{u^2}{2} \right)_x = 0
$ 
on the domain $[-1, 1]$ with periodic boundary conditions.

\begin{ex}[Self-Steepening Shock]
\rm
This example is used to test the ability of the IDP PAMPA scheme to handle nonlinear problems effectively. The 1D Burgers' equation is solved with 400 uniform cells until $t = 0.5$, with the initial conditions:
$$
u_0(x) = \begin{cases}
2, & \text{if } |x| < 0.2, \\
-1, & \text{otherwise.}
\end{cases}
$$
As illustrated in Figure \ref{burgersfig}, the original PAMPA scheme generates a spike at the initial discontinuity due to an inaccurate estimation of the upwind direction at the cell interface, as also observed in \cite{duan2024activefluxmethodshyperbolic}. In contrast, our IDP PAMPA scheme effectively addresses these spikes, demonstrating its capability to enhance stability and accurately capture discontinuities. 
\end{ex}

\begin{figure}[h]
\centering
    \begin{subfigure}[b]{0.48\textwidth} 
        \centering
        \includegraphics[width=\textwidth]{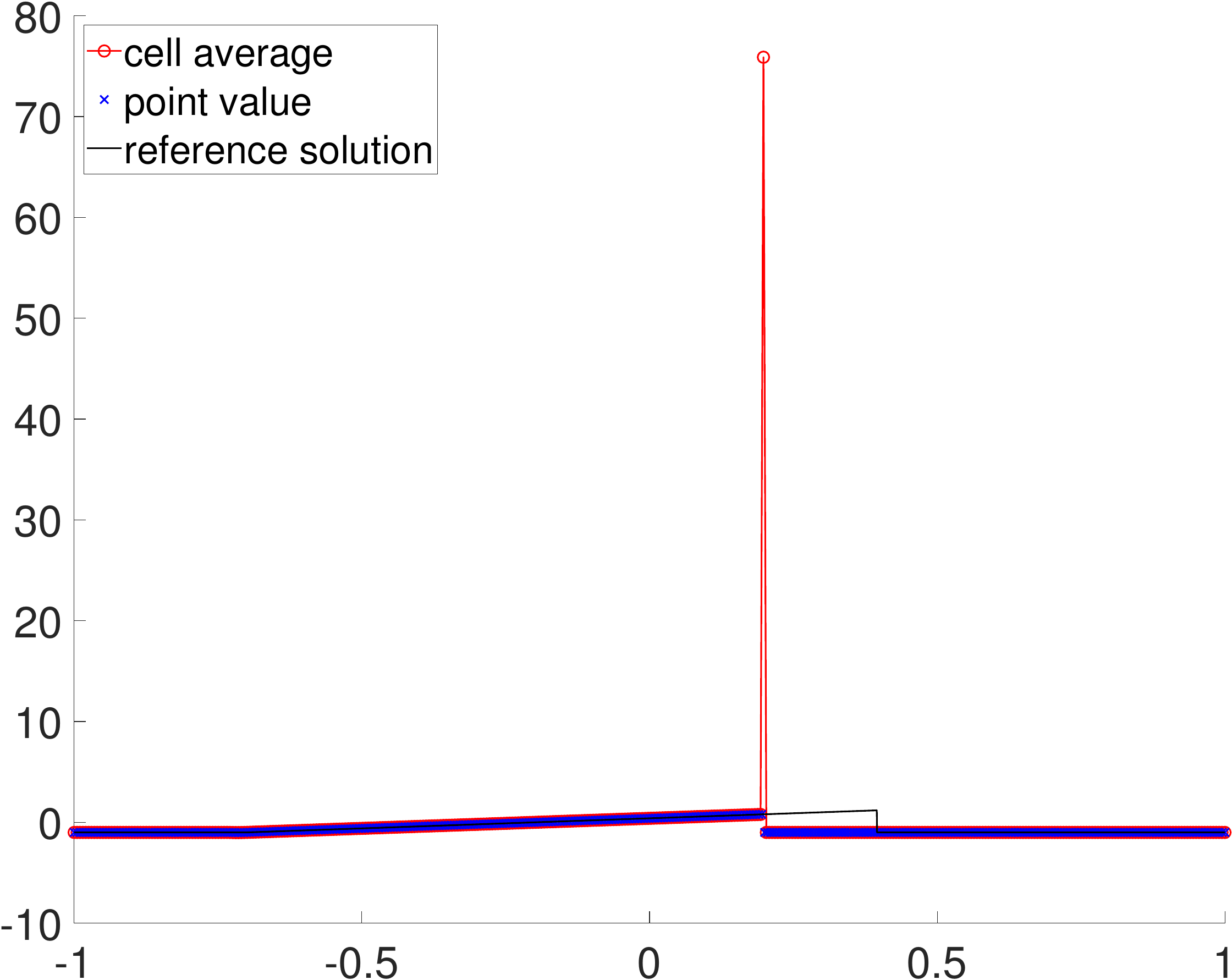}
        \caption{Original PAMPA without IDP}
    \end{subfigure}
    \hfill
    \begin{subfigure}[b]{0.48\textwidth}
        \centering
        \includegraphics[width=\textwidth]{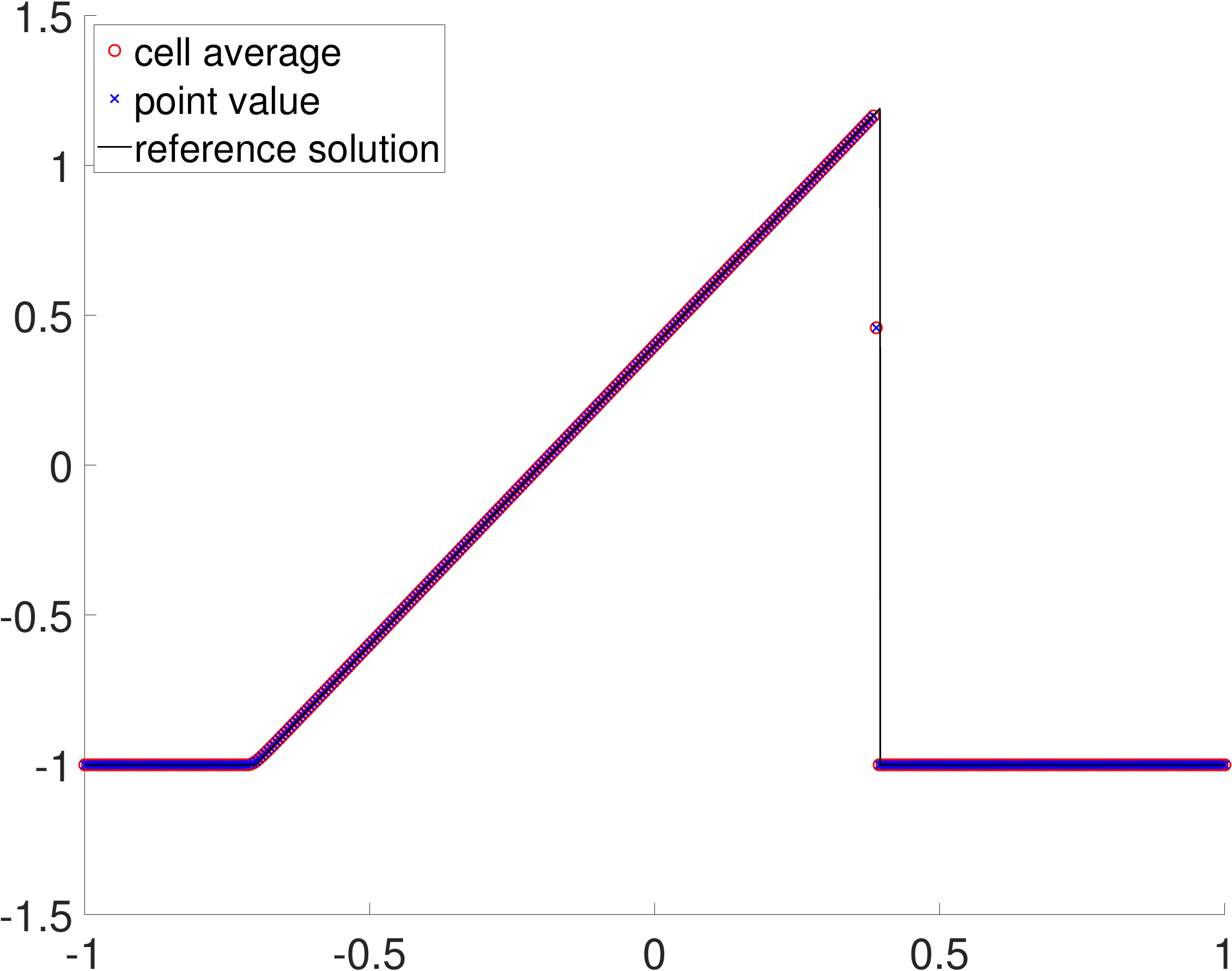}
        \caption{IDP PAMPA}
    \end{subfigure}
    \caption{Numerical results of PAMPA scheme with or without our IDP technique.}
    \label{burgersfig}
\end{figure}

\subsection{Compressible Euler Equations}

This subsection presents several challenging examples of the 1D compressible Euler equations \eqref{Eulerconservation}. As discussed in subsection 5.3, we use the variables $\mathbf{W} = \left( \ln ( e^{\frac{\rho}{\rho_{\rm ref}}} - 1 ), v, s \right)^T$ to evolve point values, thereby ensuring the automatic IDP property. The adiabatic index is taken as $\gamma = 1.4$, except for Example 7.10, where $\gamma = \frac{5}{3}$.

\begin{ex}[Smooth problem]
\rm
This example considers a smooth problem with the exact solution:
\[
\rho(x,t) = 1 + 0.999 \sin(x - t), \quad v(x,t) = 1, \quad p(x,t) = 10^{-8},
\]
which involves very low density and pressure. 
The computational domain is $[0, 2\pi]$ with periodic boundary conditions. The simulations are conducted up to $t = 0.1$. The failure to preserve the positivity of density and pressure can lead to blow-ups in this example if the original PAMPA scheme is used with a coarse mesh. Table \ref{eulerordertest} presents the $l^1$ errors in density and corresponding convergence rates. For time discretization, we employ the third-order SSP multi-step method. We clearly observe the designed third-order accuracy, which is not compromised by the IDP, OE, and MP techniques.

\end{ex}

\begin{table}[ht]
\centering
\caption{Errors and convergence rates for cell averages (top) and point values (bottom) computed with $N$ uniform cells for the 1D compressible Euler equations.}\label{eulerordertest}
\begin{tabular}{lcccccccc}
\toprule
 &   & \multicolumn{2}{c}{{IDP PAMPA}} & \multicolumn{2}{c}{{IDP PAMPA with OE}} & \multicolumn{2}{c}{{IDP PAMPA with MP}} \\
\cmidrule(lr){3-4} \cmidrule(lr){5-6}  \cmidrule(lr){7-8} 
 &$N$ & Error & Order& Error & Order& Error & Order \\
\midrule
\multirow{4}{*}{} 
& 20   & 1.85e-4 & - & 4.70e-4 & - & 1.79e-4 & - \\
& 40  & 6.67e-5 & 1.47 & 6.82e-5 & 2.78 & 6.29e-5 & 1.51 \\
& 80  & 1.07e-5 & 2.64 &1.09e-5& 2.65 & 1.31e-5 & 2.61 \\
& 160  & 1.65e-6 & 2.70 & 1.65-6 & 2.72 & 1.63e-6 & 2.66 \\
& 320  & 2.69e-7 &2.62 & 2.69e-7 & 2.62 & 2.70e-7 &2.60\\
& 640  & 3.76e-8 &2.83 & 3.76e-8 &2.83 & 3.78e-8 &2.84\\
& 1280 & 4.89e-9&2.94 & 4.89e-9 &2.94 & 4.91e-9&2.94\\
\midrule
\multirow{4}{*}{} 
& 20   & 7.18e-4 & - & 4.97e-4 & - & 7.21e-4 & - \\
& 40  & 6.28e-5 & 3.52 & 1.05e-4 & 2.24 & 5.63e-5 & 3.68 \\
& 80  & 1.80e-5 & 1.80 & 1.79e-5 & 2.55 & 1.75e-5 & 1.69 \\
& 160 & 2.77e-6 & 2.70 & 2.77e-6 & 2.69& 2.77e-6 & 2.66 \\
& 320  & 3.85e-7 &2.85 & 3.85e-7 & 2.85 & 3.87e-7 &2.84\\
& 640 &5.12e-8 &2.91& 5.12e-8 &2.91 & 5.13e-8 &2.91\\
& 1280  & 6.57e-9&2.96 & 6.57e-9 &2.96 & 6.58e-9&2.96\\
\bottomrule
\end{tabular}
\end{table}

\begin{ex}[Sod problem]
\rm
This is a classical Riemann problem used to test the ability of numerical schemes to capture shock wave, contact discontinuity, and rarefaction wave. The computational domain is $[-5,5]$, and the initial conditions are
\[
(\rho,v,p) = 
\begin{cases} 
(1,~0,~1), & \text{if} \quad x \leq 0, \\
(0.125,~0,~0.1), & \text{otherwise}.
\end{cases}
\]
Figure \ref{eulerrhosod} presents the numerical solutions at $t = 1.3$ obtained using the IDP PAMPA scheme with 200 uniform cells. The MP limiter is applied to suppress nonphysical oscillations. As shown, all of the shock waves, contact discontinuities, and rarefaction waves are well captured without introducing numerical oscillations.
\end{ex}

\begin{figure}[h]
    \centering
    \begin{subfigure}[b]{0.32\textwidth} 
        \centering
        \includegraphics[width=\textwidth]{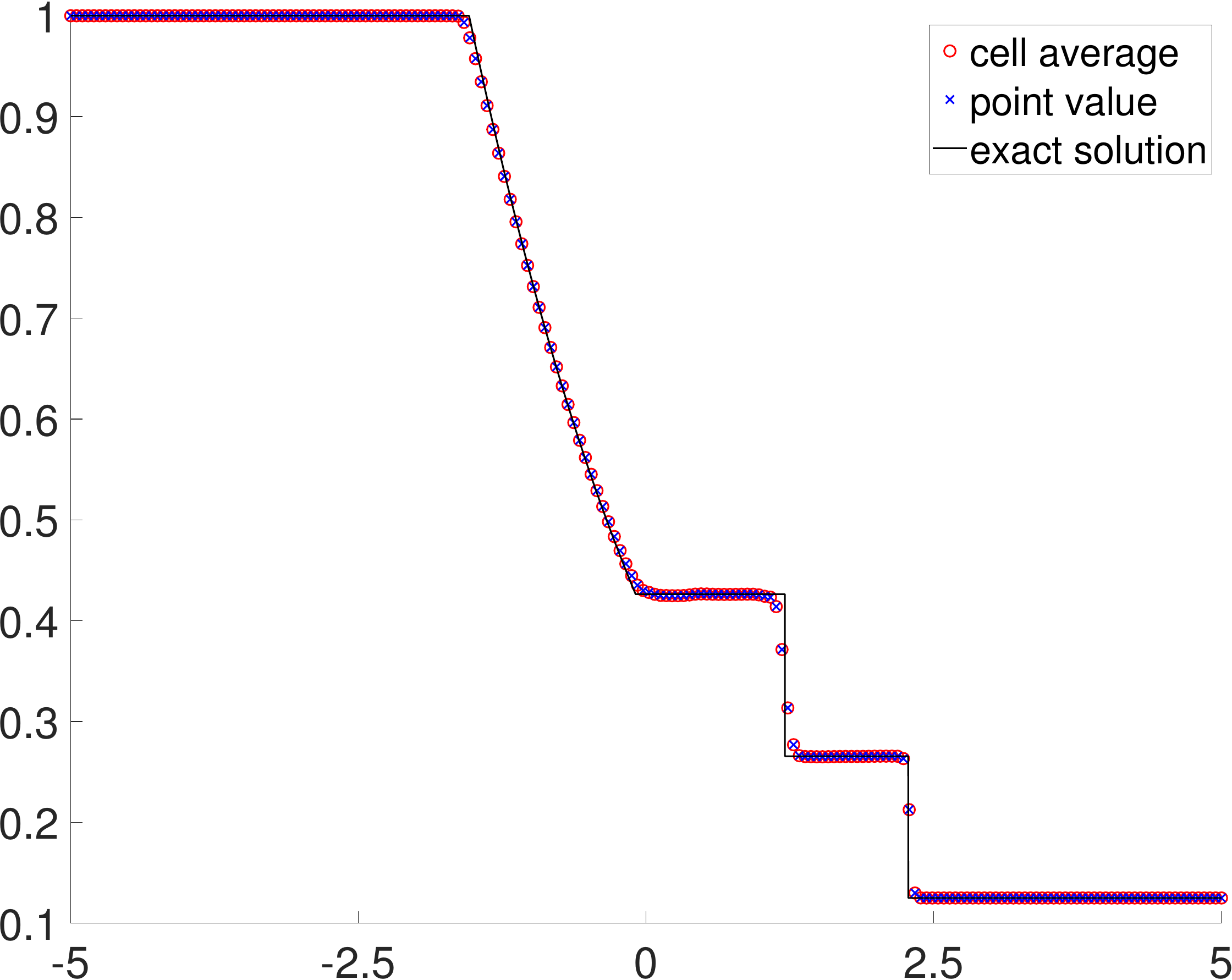}
        \caption{density}
    \end{subfigure}
    \hfill
    \begin{subfigure}[b]{0.32\textwidth}
        \centering
        \includegraphics[width=\textwidth]{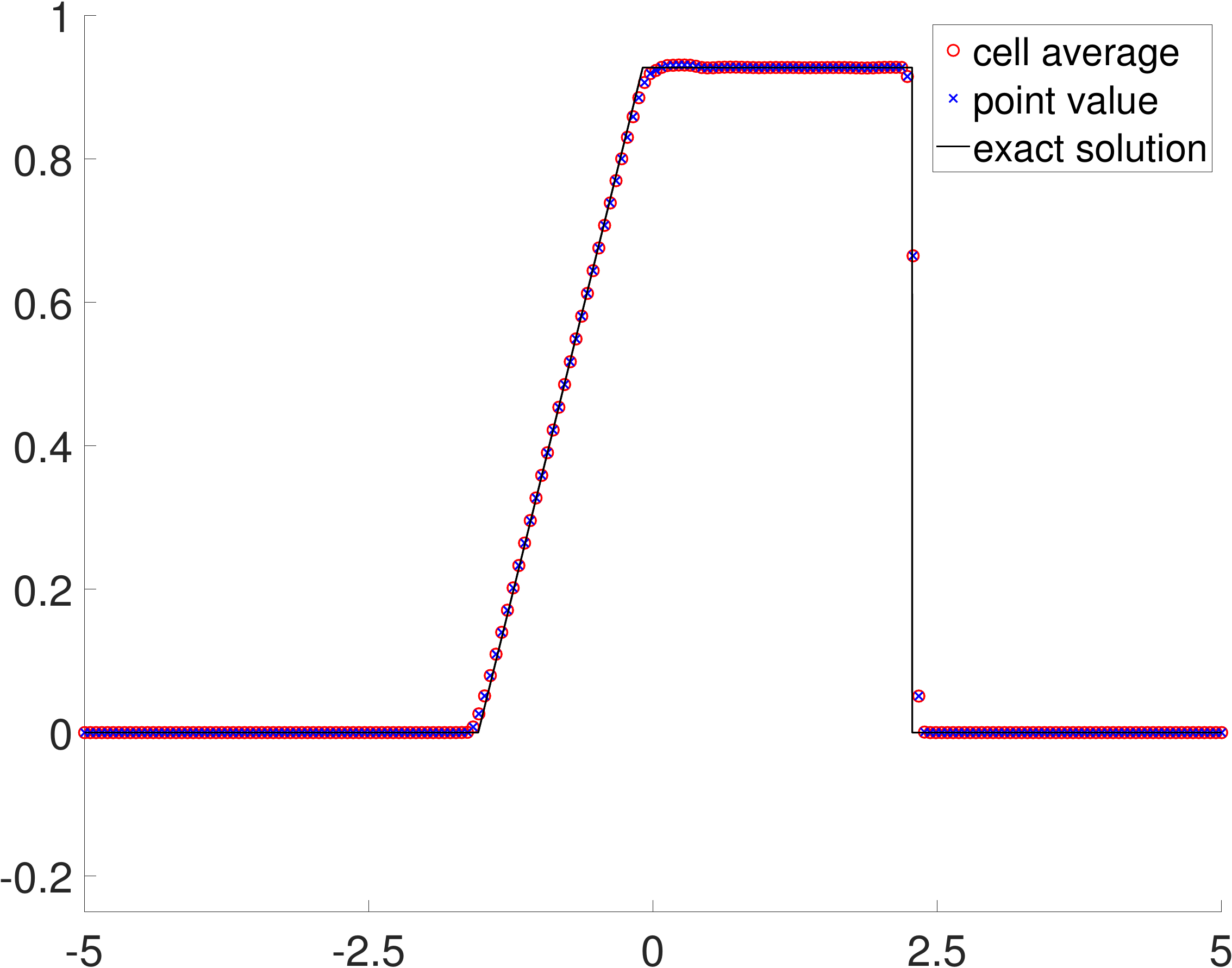}
        \caption{velocity}
    \end{subfigure}
    \hfill
        \begin{subfigure}[b]{0.32\textwidth}
        \centering
        \includegraphics[width=\textwidth]{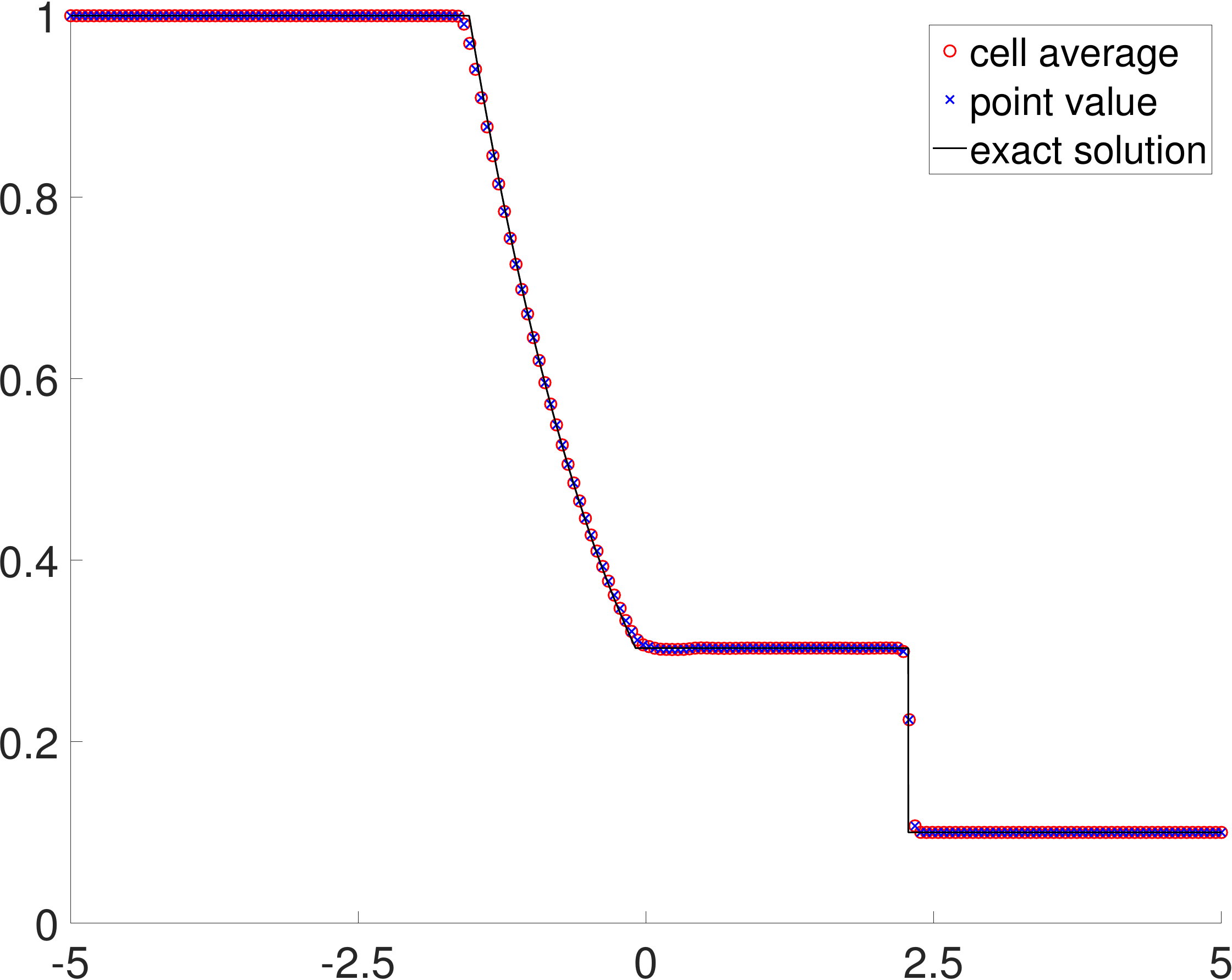}
        \caption{pressure}
    \end{subfigure} 
    \caption{Example 7.5: Numerical solutions of IDP PAMPA scheme with 200 cells.}
    \label{eulerrhosod}
\end{figure}

\begin{ex}[Interaction of two blast waves]
\rm
This example simulates the interaction of two blast waves in the domain $[0,1]$ with reflective boundary conditions. The initial setup consists of three distinct pressure regions: the left and right regions have high pressures, while the middle region has low pressure. The initial conditions are described as follows:
\[
(\rho,v,p) = 
\begin{cases} 
(1,~0,~10^3), & \text{if} \quad 0 \le x < 0.1, \\
(1,~0,~10^{-2}), & \text{if} \quad 0.1 < x < 0.9, \\
(1,~0,~10^2), & \text{if} \quad 0.9 < x \le 1.
\end{cases}
\]
The numerical simulation is conducted up to $t = 0.038$ using the IDP PAMPA scheme with the OE procedure on a mesh of 800 uniform cells, while the reference solution is computed using the local Lax--Friedrichs scheme with 10,000 uniform cells. The scheme  without the IDP technique would generate negative pressure and breakdown. As shown in Figure \ref{eulerrhoblast}, the blast waves are correctly resolved with high resolution, and the numerical solution obtained from the IDP PAMPA scheme agrees well with the reference solution.
\end{ex}

\begin{figure}[h]
    \centering
    \begin{subfigure}[b]{0.32\textwidth} 
        \centering
        \includegraphics[width=\textwidth]{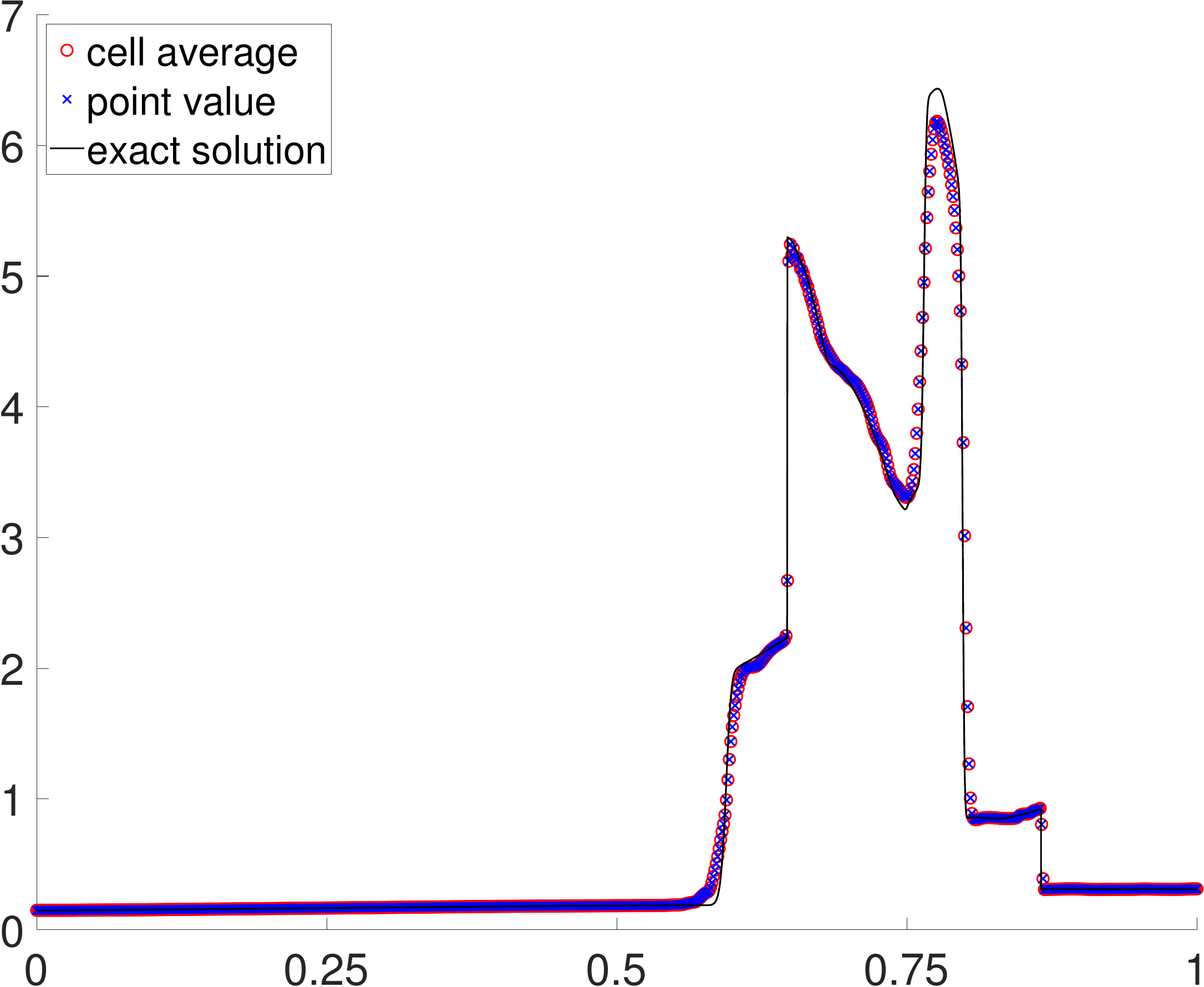}
        \caption{density}
    \end{subfigure}
    \hfill
    \begin{subfigure}[b]{0.32\textwidth}
        \centering
        \includegraphics[width=\textwidth]{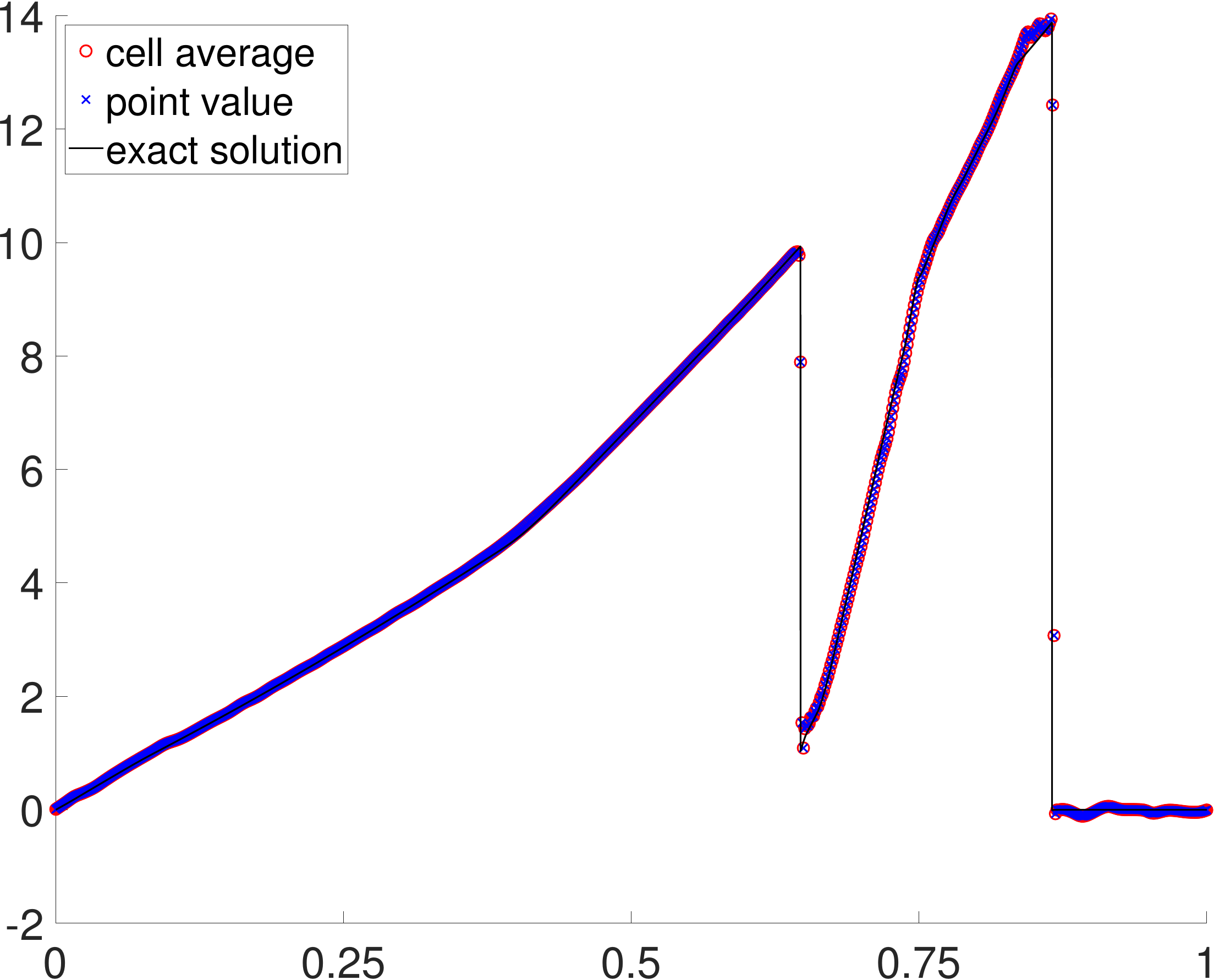}
        \caption{velocity}
    \end{subfigure}
    \hfill
        \begin{subfigure}[b]{0.32\textwidth}
        \centering
        \includegraphics[width=\textwidth]{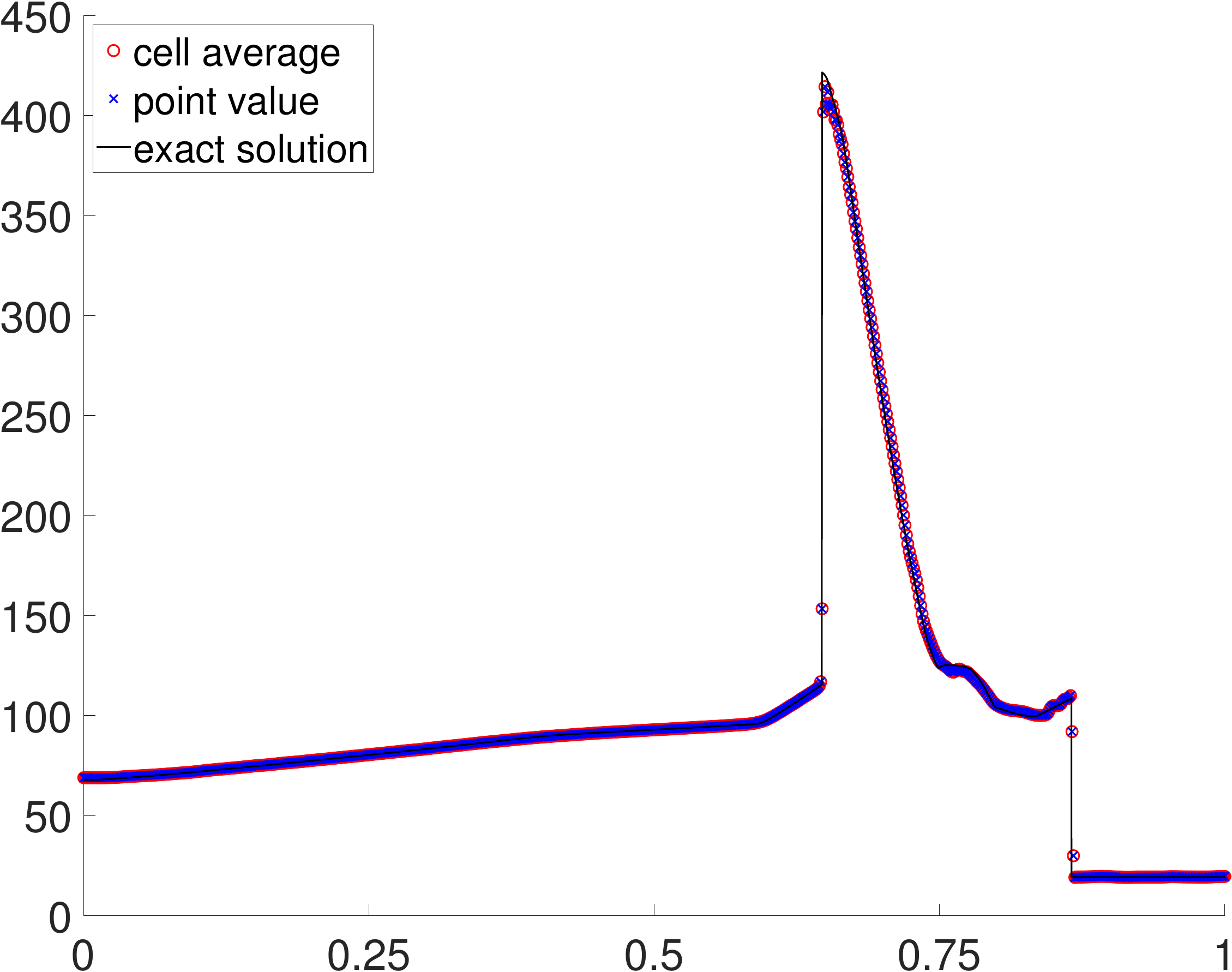}
        \caption{pressure}
    \end{subfigure}    
    \caption{Example 7.6: Numerical solutions of IDP PAMPA scheme with 800 cells.}
    \label{eulerrhoblast}
\end{figure}

\begin{ex}[Double rarefaction problem]
\rm
This example is initialized symmetrically with two rarefaction waves propagating outward from the center of the domain $[-1,1]$:
$$
(\rho,v,p)=\begin{cases}(7,~-1,~0.2), & {\rm if} ~~ -1<x<0,\\
(7,~1,~0.2),& {\rm if} ~~ 0<x<1.
\end{cases}
$$
Outflow boundary conditions are applied. The simulation is performed using the IDP PAMPA scheme with 400 uniform cells, and the MP limiter is used to suppress nonphysical overshoots. The numerical result at $t=0.6$ is shown in Figure \ref{eulerrhodoublerarefaction}. 
Since the exact solution involves a vacuum region, the IDP technique is necessary to prevent the appearance of negative pressures. Without the IDP technique, the original PAMPA scheme would fail, due to the production of negative numerical pressures. As observed, the numerical solution obtained using the IDP PAMPA scheme is in good agreement with the exact solution, even in the vicinity of the vacuum region.
\end{ex}

\begin{figure}[h]
    \centering
    \begin{subfigure}[b]{0.32\textwidth} 
        \centering
        \includegraphics[width=\textwidth]{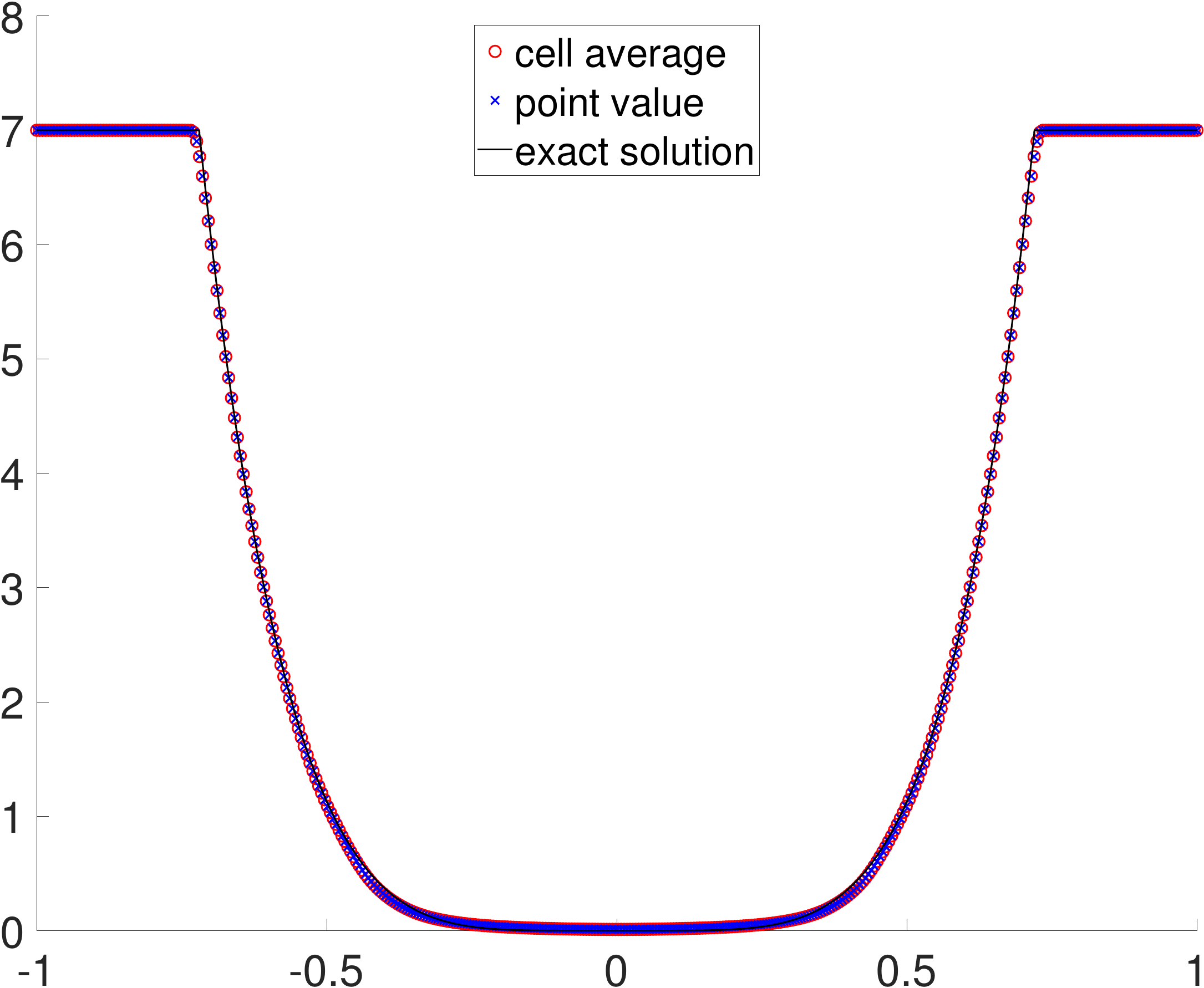}
        \caption{density}
    \end{subfigure}
    \hfill
         \begin{subfigure}[b]{0.32\textwidth}
        \centering
        \includegraphics[width=\textwidth]{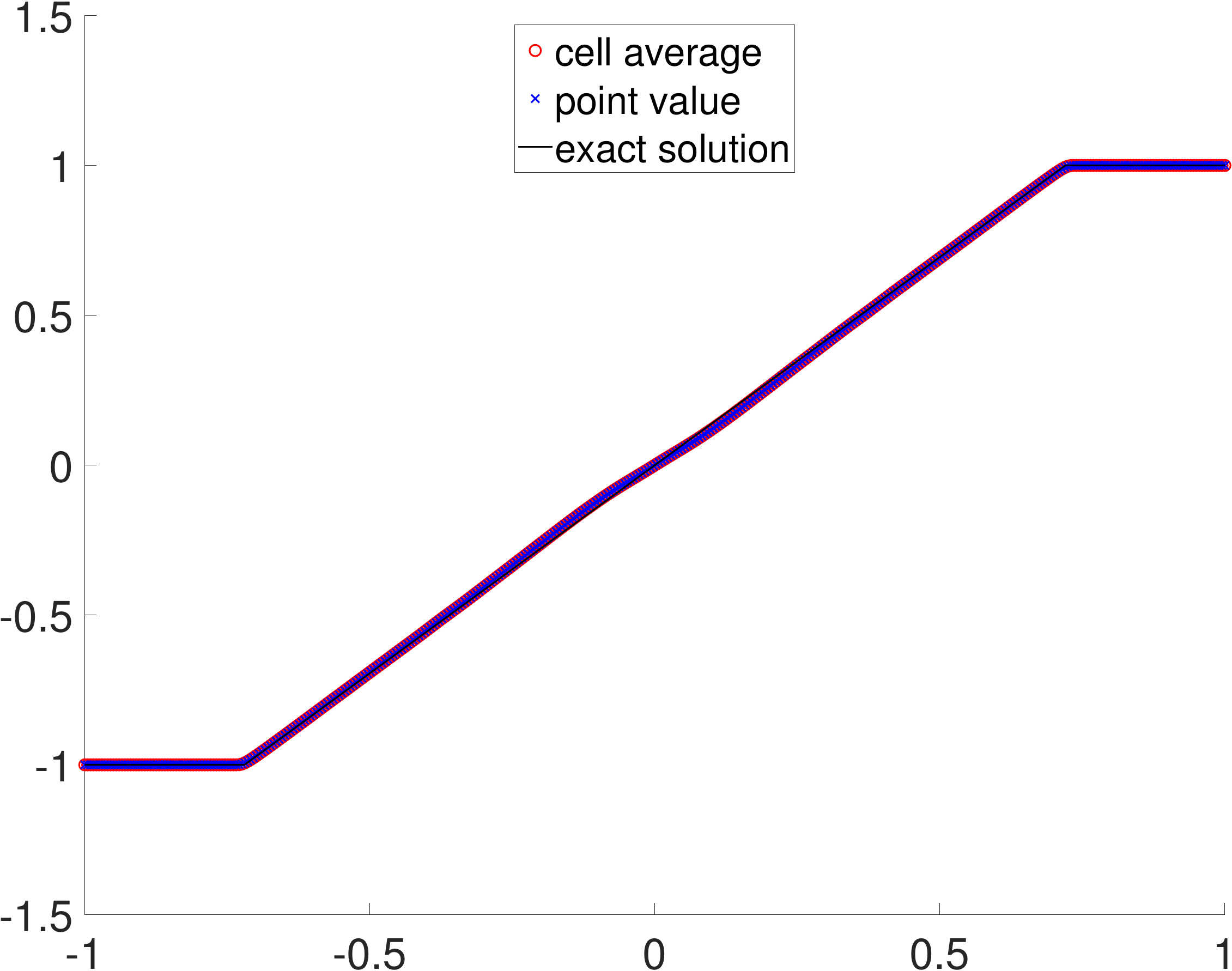}
        \caption{velocity}
    \end{subfigure}
    \hfill
        \begin{subfigure}[b]{0.32\textwidth}
        \centering
        \includegraphics[width=\textwidth]{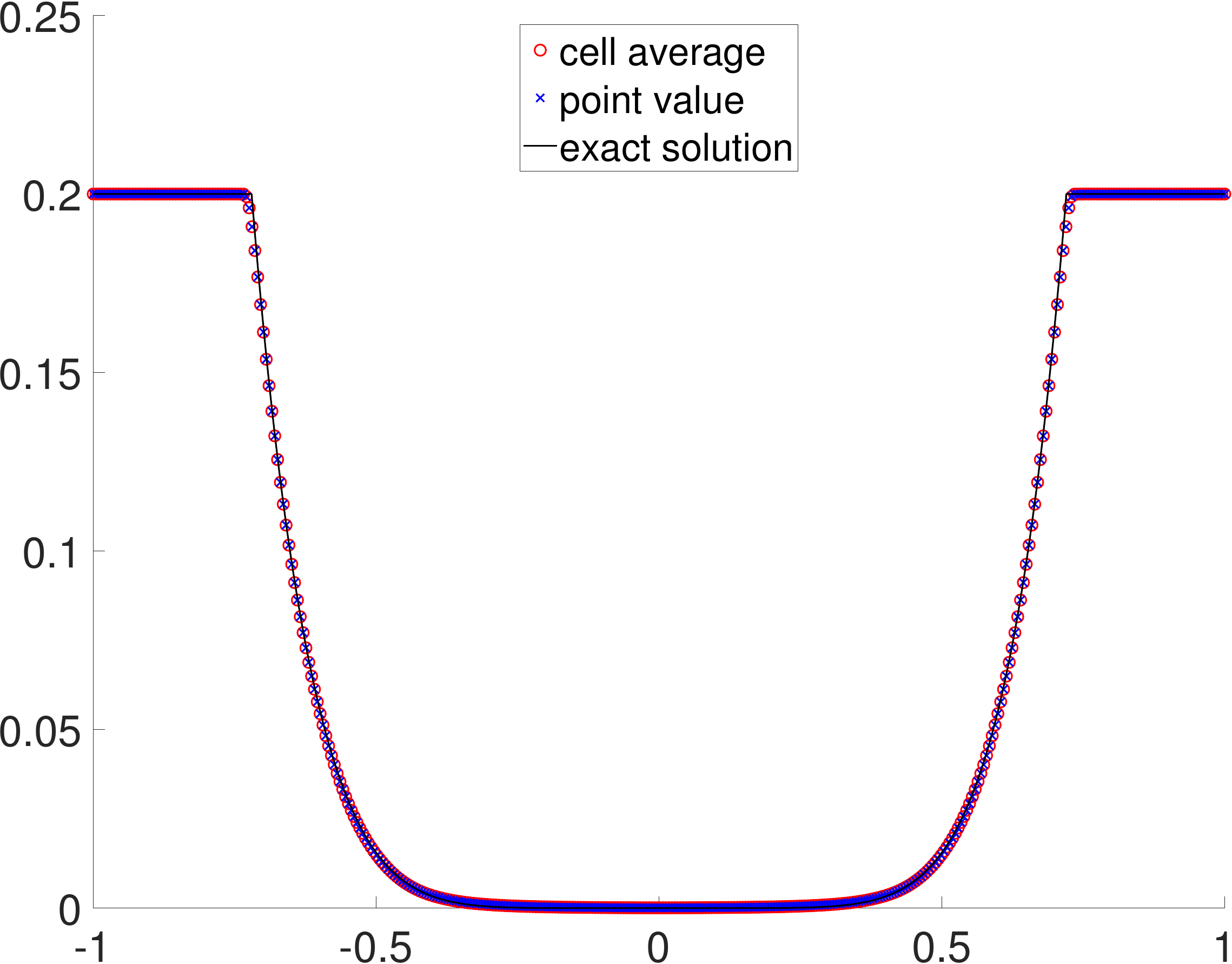}
        \caption{pressure}
    \end{subfigure}   
    \caption{Example 7.7: Numerical results of IDP PAMPA scheme with 400 cells.}
    \label{eulerrhodoublerarefaction}
\end{figure}

\begin{ex}[Shu--Osher problem]
\rm
This problem describes the interaction of sine waves and a right-moving shock. It is typically used to test the ability of high-order schemes to capture both the shock and the physical oscillations. The initial conditions are given by
\[
(\rho, v, p) = 
\begin{cases} 
(3.857143, 2.629369, 10.33333), & \text{if} \quad x < -4, \\
(1 + 0.2 \sin(5x), 0, 0.1), & \text{if} \quad x > -4.
\end{cases}
\]
The computational domain is chosen as $\Omega = [-5, 5]$, and the final time is set to $t = 1.8$. The solutions are computed using the IDP PAMPA scheme on a uniform mesh of 640 cells, and the results are shown in Figure \ref{eulerrhoshuosher}, where the OE procedure is applied to control potential nonphysical oscillations. The reference solution is computed using the local Lax–Friedrichs scheme on 300,000 uniform cells. It is observed that the IDP PAMPA scheme effectively captures both the shock and the high-frequency waves.
\end{ex}

\begin{figure}[h]
    \centering
    \begin{subfigure}[b]{0.32\textwidth} 
        \centering
        \includegraphics[width=\textwidth]{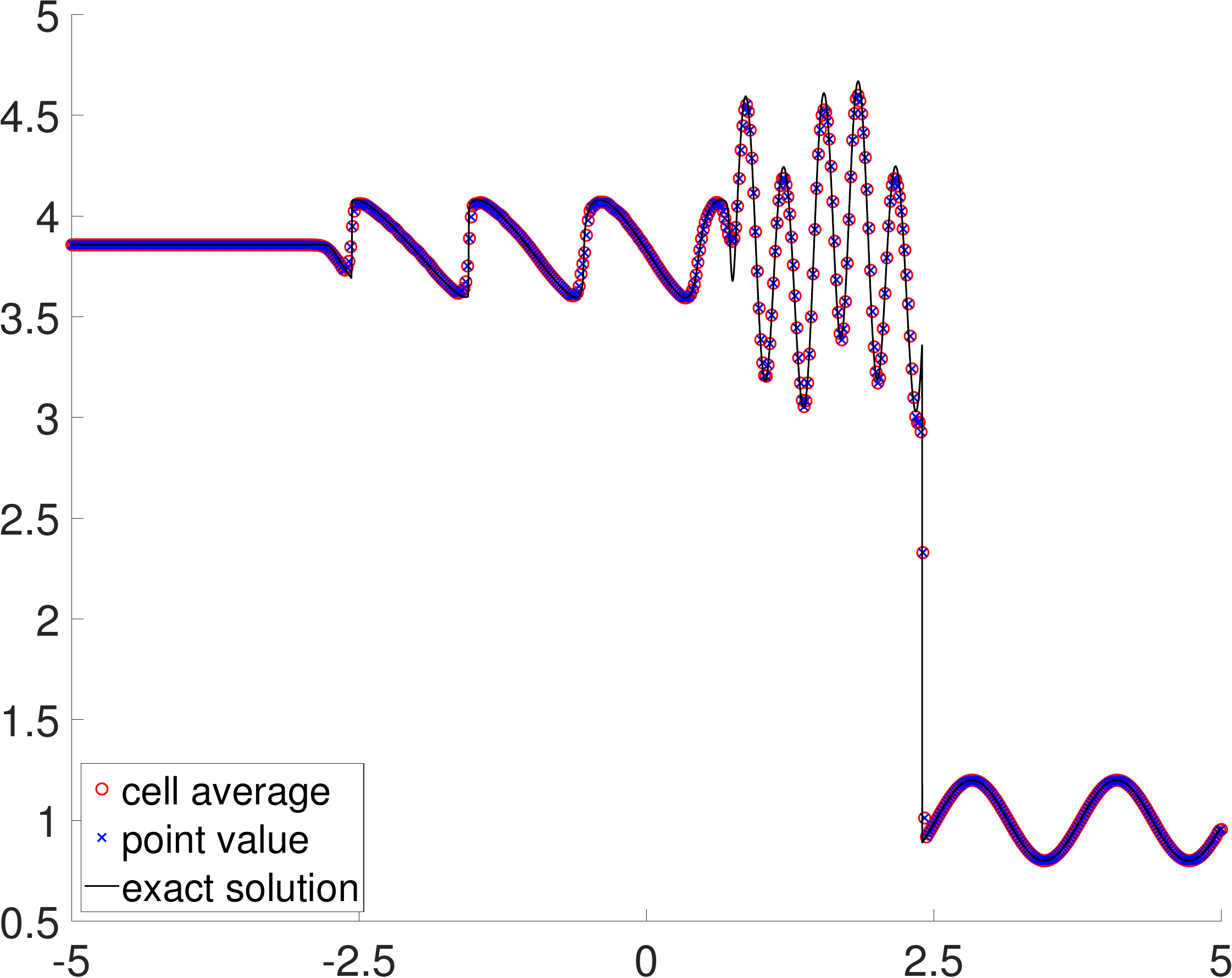}
        \caption{density}
    \end{subfigure}
    \hfill
    \begin{subfigure}[b]{0.32\textwidth}
        \centering
        \includegraphics[width=\textwidth]{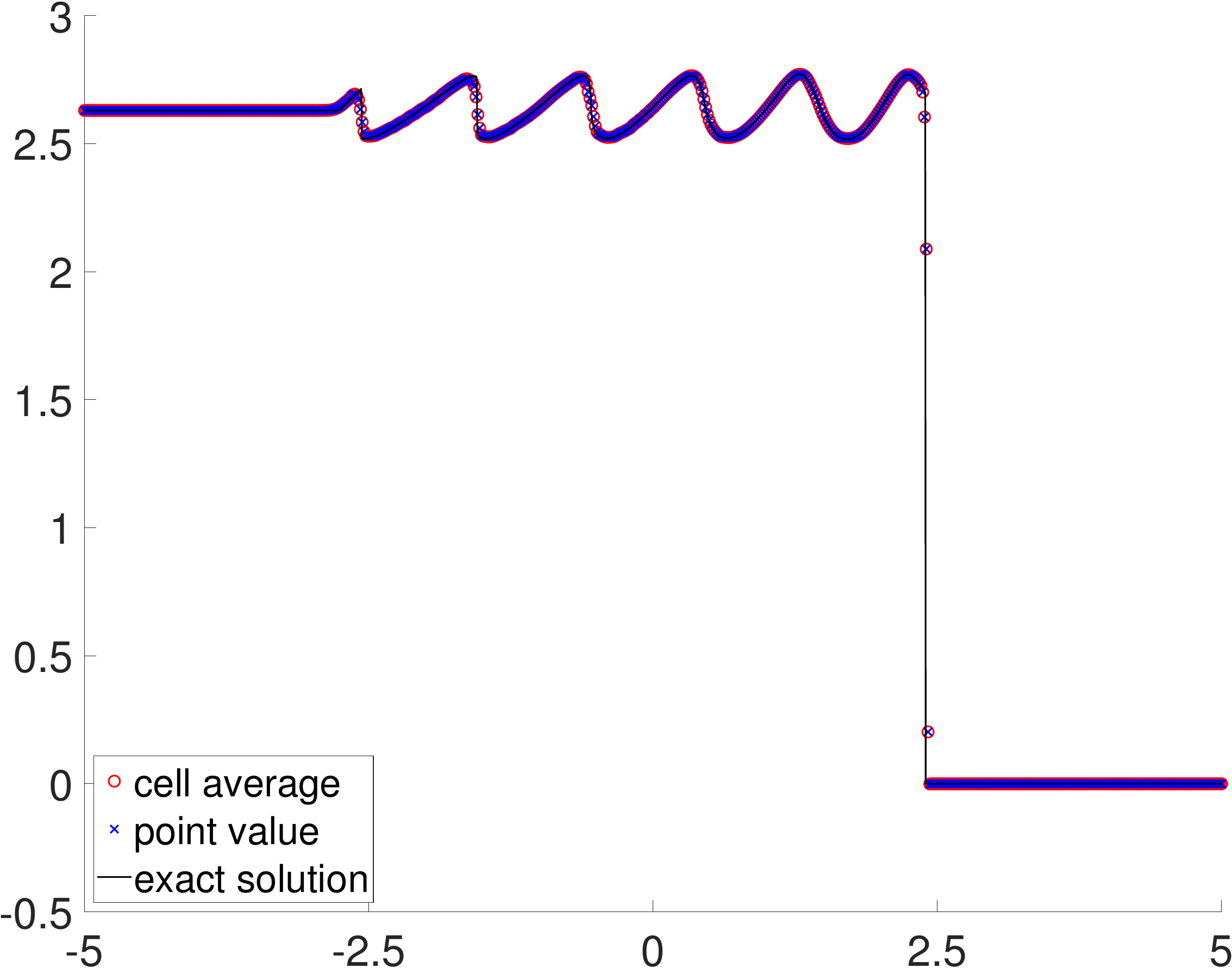}
        \caption{velocity}
    \end{subfigure}
    \hfill
        \begin{subfigure}[b]{0.32\textwidth}
        \centering
        \includegraphics[width=\textwidth]{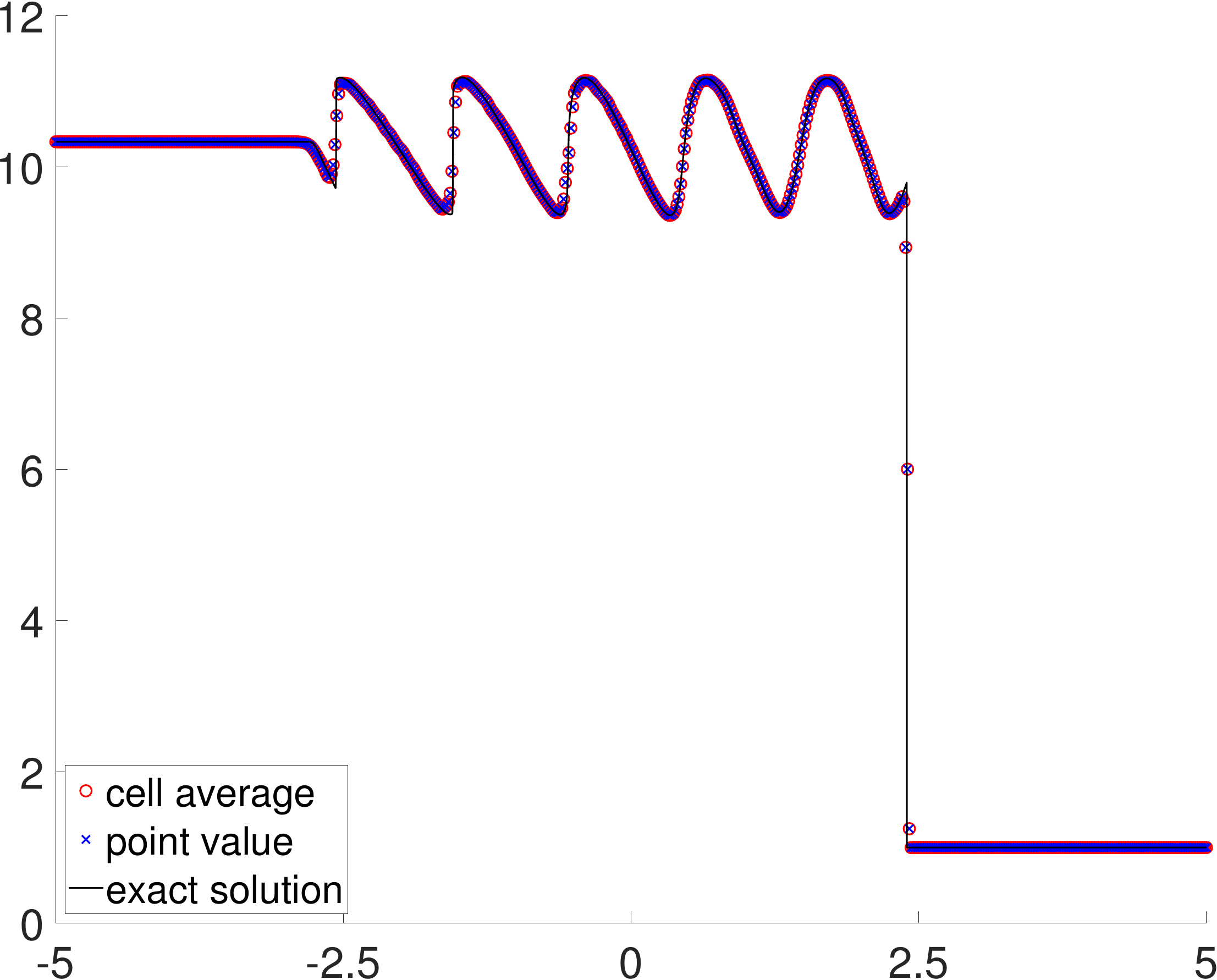}
        \caption{pressure}
    \end{subfigure}    
    \caption{Example 7.8: Numerical solutions of IDP PAMPA scheme with 640 cells.}
    \label{eulerrhoshuosher}
\end{figure}

\begin{ex}[Sedov problem]
\rm
This is a highly demanding problem defined on the domain $[-2,2]$, which involves very strong shocks and low-density regions. The exact solution is available in \cite{sedov1959similarity, korobeinikov1991problems}. The initial conditions are set with a uniform density of 1, zero velocity, and a total energy of $10^{-12}$ everywhere, except in the center cell, where the total energy is specified as $3.2 \times 10^{6}  \Delta x$, where $\Delta x$ is the mesh size. The simulation is conducted up to $t = 0.001$ using the IDP PAMPA scheme with 801 cells. The MP limiter is applied to suppress spurious oscillations near shocks. 
As shown in Figure \ref{eulerrhosedov}, the IDP PAMPA scheme effectively resolves both the strong shocks and the low-density profile. Without the IDP technique, the original PAMPA scheme would fail immediately for this simulation, highlighting the necessity of the IDP approach for handling such challenging problems.
\end{ex}

\begin{figure}[h]
    \centering
    \begin{subfigure}[b]{0.32\textwidth} 
        \centering
        \includegraphics[width=\textwidth]{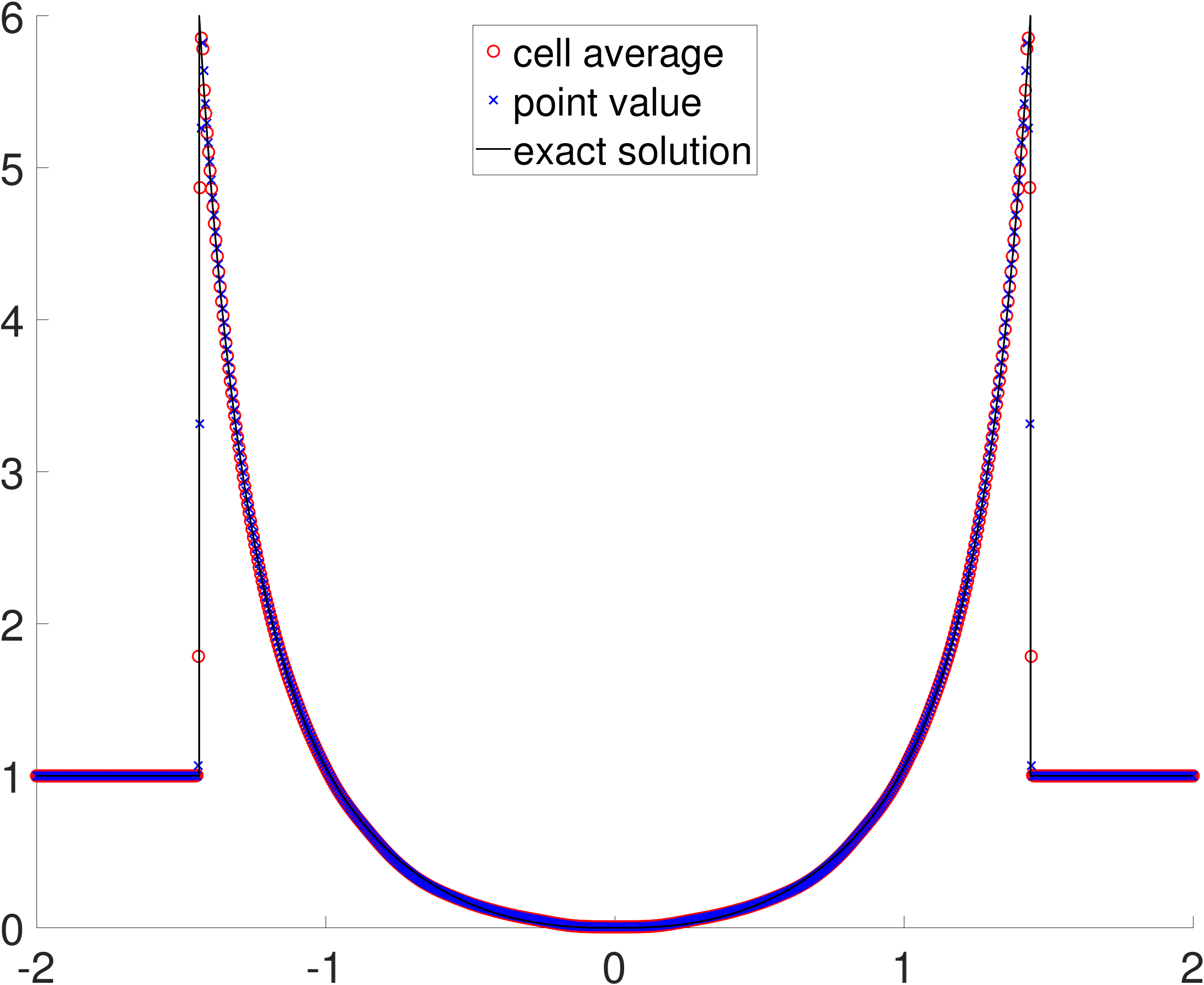}
        \caption{density}
    \end{subfigure}
    \hfill
        \begin{subfigure}[b]{0.32\textwidth}
        \centering
        \includegraphics[width=\textwidth]{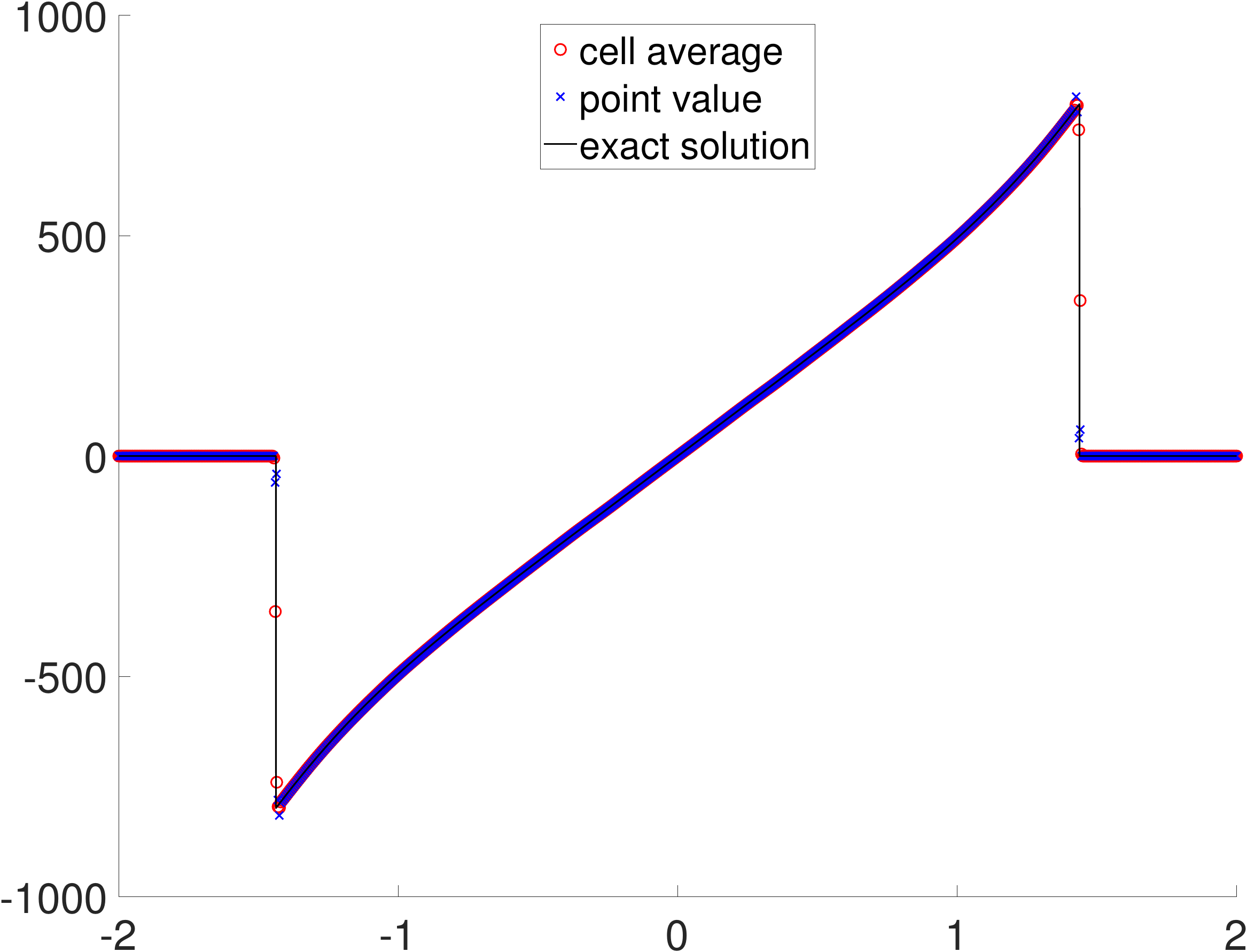}
        \caption{velocity}
    \end{subfigure}
    \hfill
        \begin{subfigure}[b]{0.32\textwidth}
        \centering
        \includegraphics[width=\textwidth]{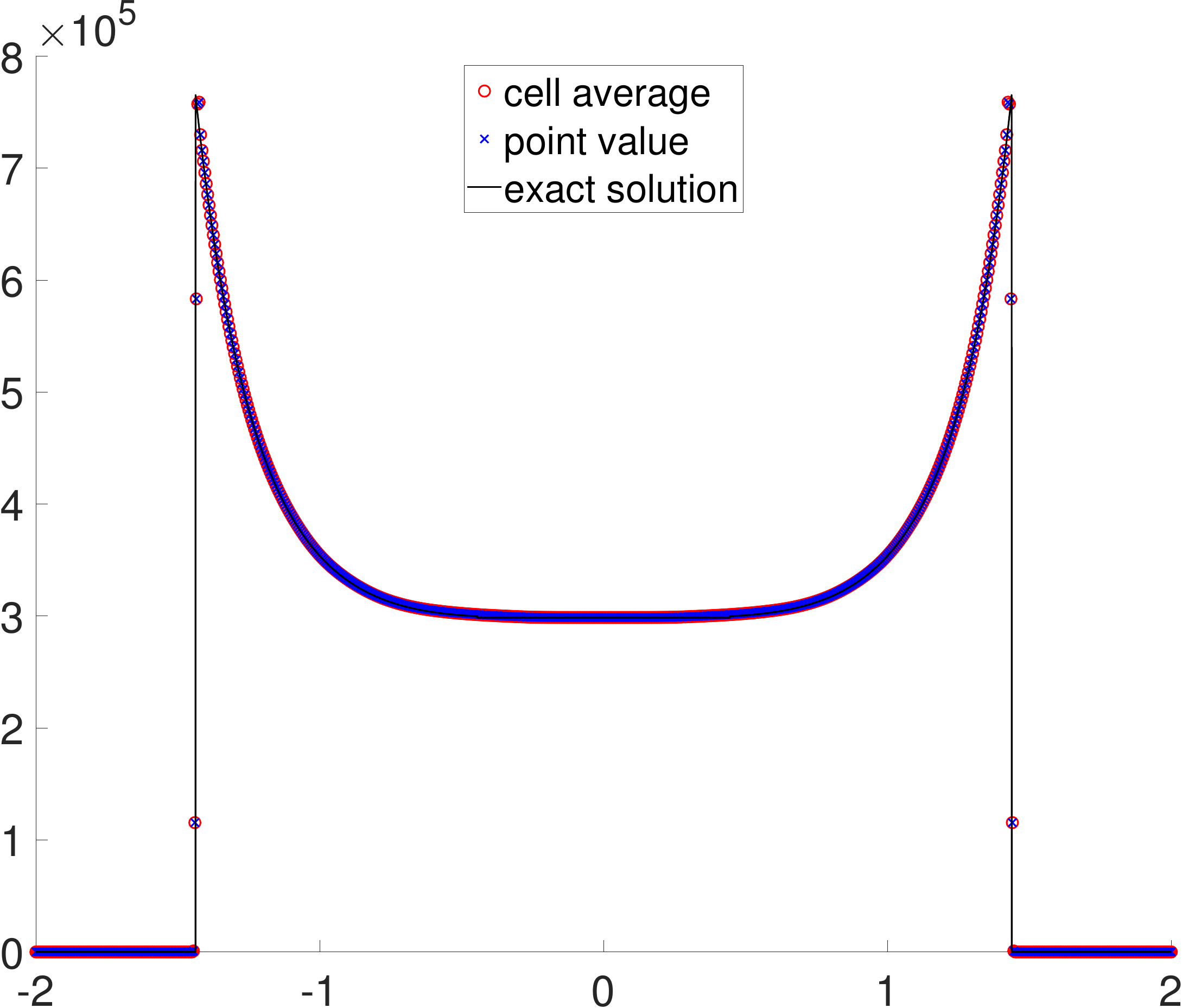}
        \caption{pressure}
    \end{subfigure}
    \caption{Example 7.9: Numerical results of IDP PAMPA scheme with 801 cells.}
    \label{eulerrhosedov}
\end{figure}

\begin{ex}[Leblanc problem]
\rm
The test is conducted on domain $[0,9]$ with the initial data:
$$
(\rho,v,p)=\begin{cases}(1,0,0.1(\gamma-1)), & {\rm if} ~~ x\leq 3,\\
(10^{-3},0,10^{-7}(\gamma-1)),& {\rm otherwise.}
\end{cases}
$$
The exact solution involves a strong shock wave propagating from the left (high-density high-pressure region) into the right (low-density low-pressure region), and a rarefaction wave moving towards the left. It involves very large gradients in pressure and density, particularly in the right-hand side near the vacuum region, posing significant challenges for the numerical simulation. 
Figure \ref{eulerrhoLeblancproblem} presents results at $t=6$ computed on a uniform mesh of 800 cells using the IDP PAMPA scheme. Both the MP limiter and OE procedure are tested for suppressing nonphysical oscillations. 
We observe that the strong shock is well captured by our scheme without spurious oscillations, even if there are some slight undershoots near the contact discontinuity, which were also noted in \cite{abgrall2024BPPAMPA, duan2024activefluxmethodshyperbolic}. Moreover, we observe that the code would break down due to negative pressure if the proposed IDP technique is not used in this test, confirming the necessity of the IDP approach for handling such challenging scenarios.
\end{ex}

\begin{figure}[h]
    \centering
    \begin{subfigure}[b]{0.32\textwidth} 
        \centering
        \includegraphics[width=\textwidth]{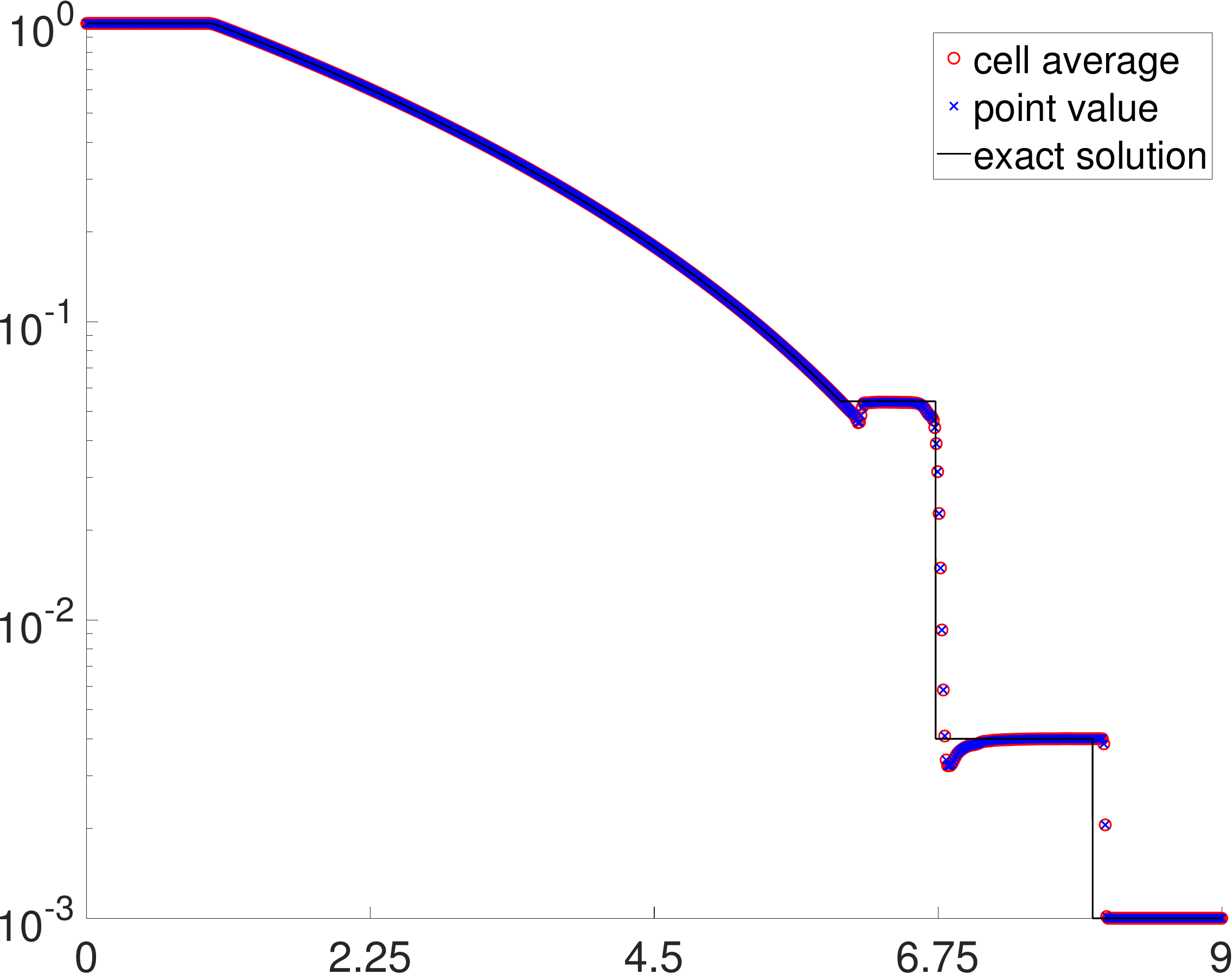}
        \caption{density}
    \end{subfigure}
      \hfill
        \begin{subfigure}[b]{0.32\textwidth}
        \centering
        \includegraphics[width=\textwidth]{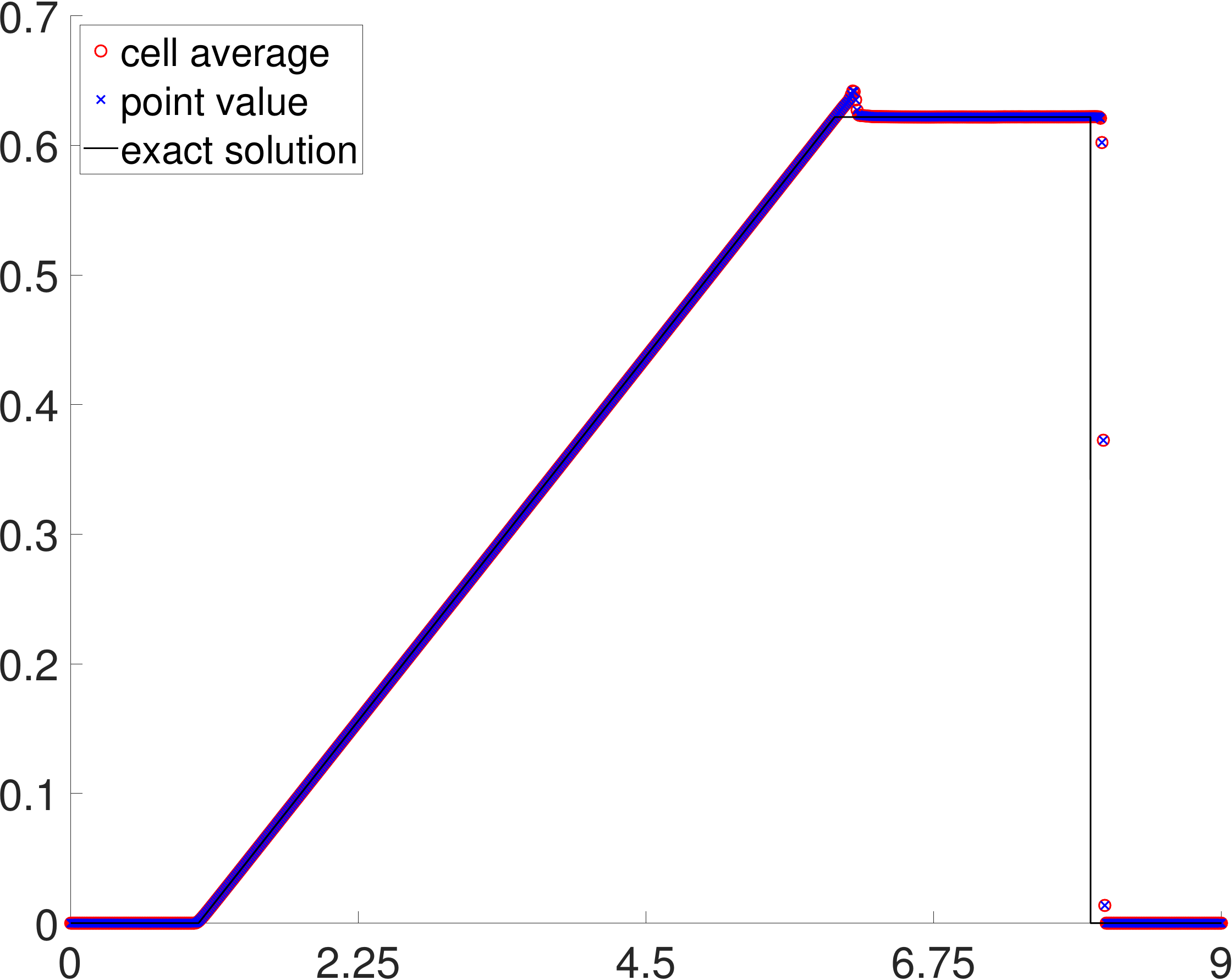}
        \caption{velocity}
    \end{subfigure}
    \hfill
    \begin{subfigure}[b]{0.32\textwidth}
        \centering
        \includegraphics[width=\textwidth]{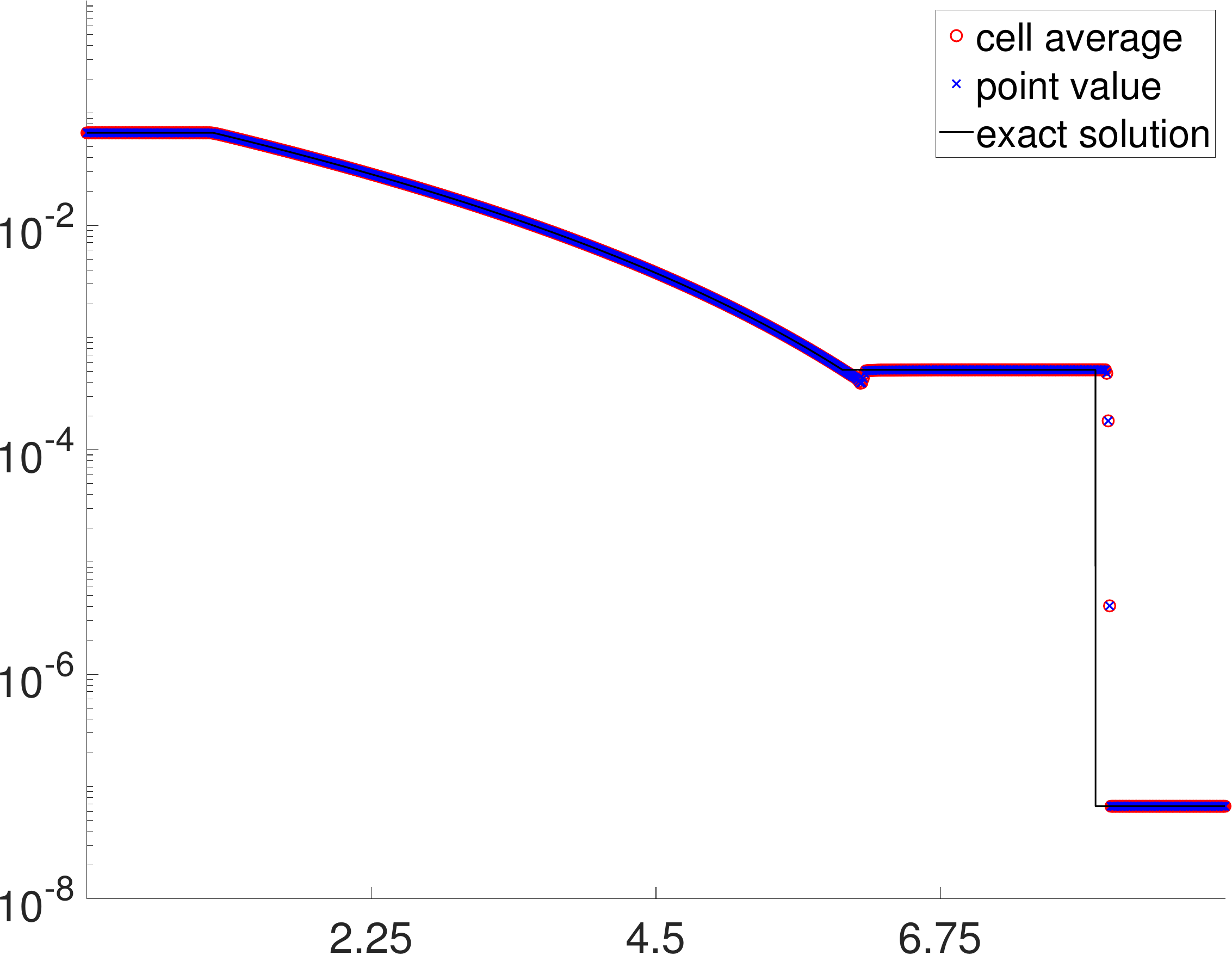}
        \caption{pressure}
    \end{subfigure}
       \vskip\baselineskip
    \centering
    \begin{subfigure}[b]{0.32\textwidth} 
        \centering
        \includegraphics[width=\textwidth]{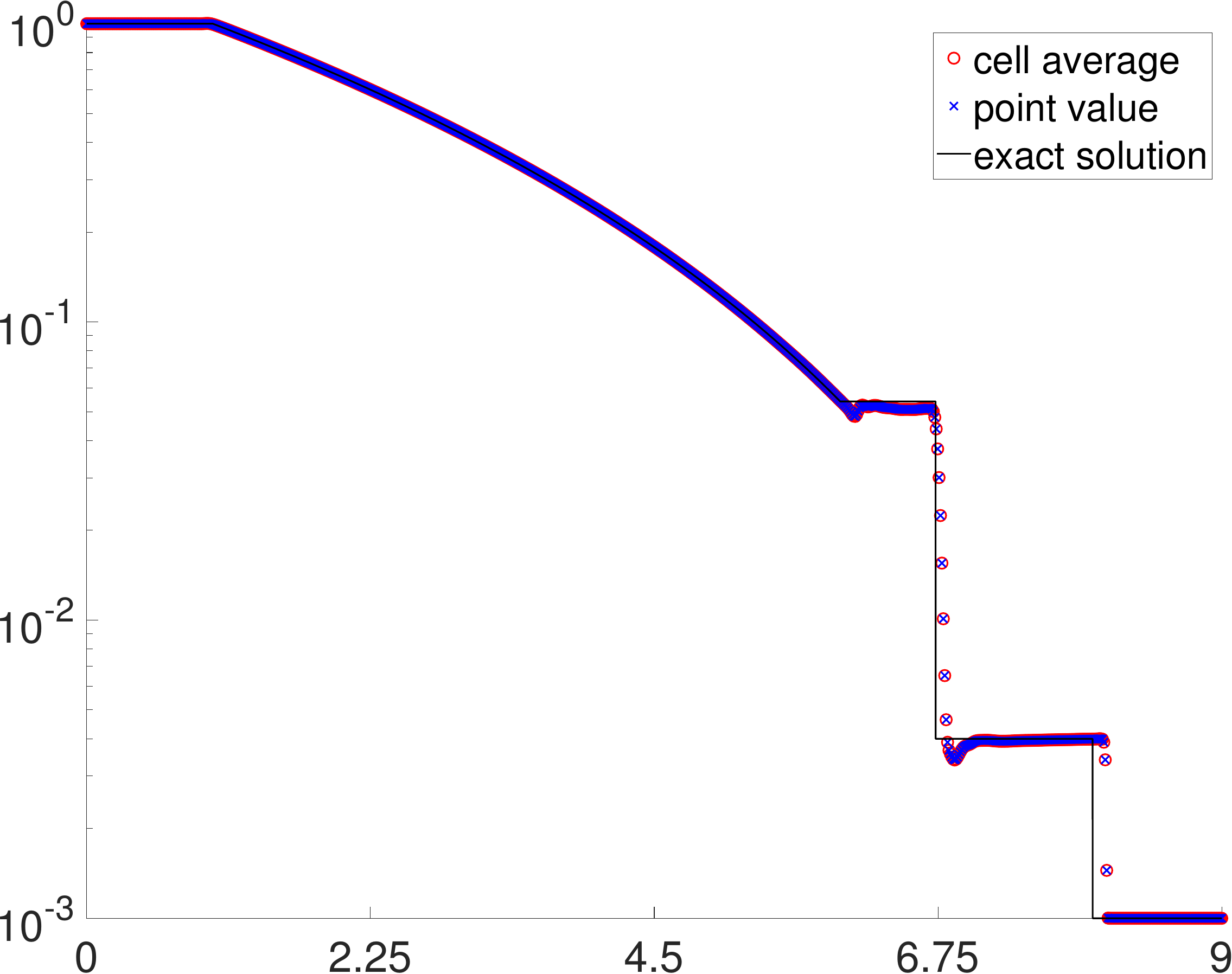}
        \caption{density}
    \end{subfigure}
    \hfill
    \begin{subfigure}[b]{0.32\textwidth}
        \centering
        \includegraphics[width=\textwidth]{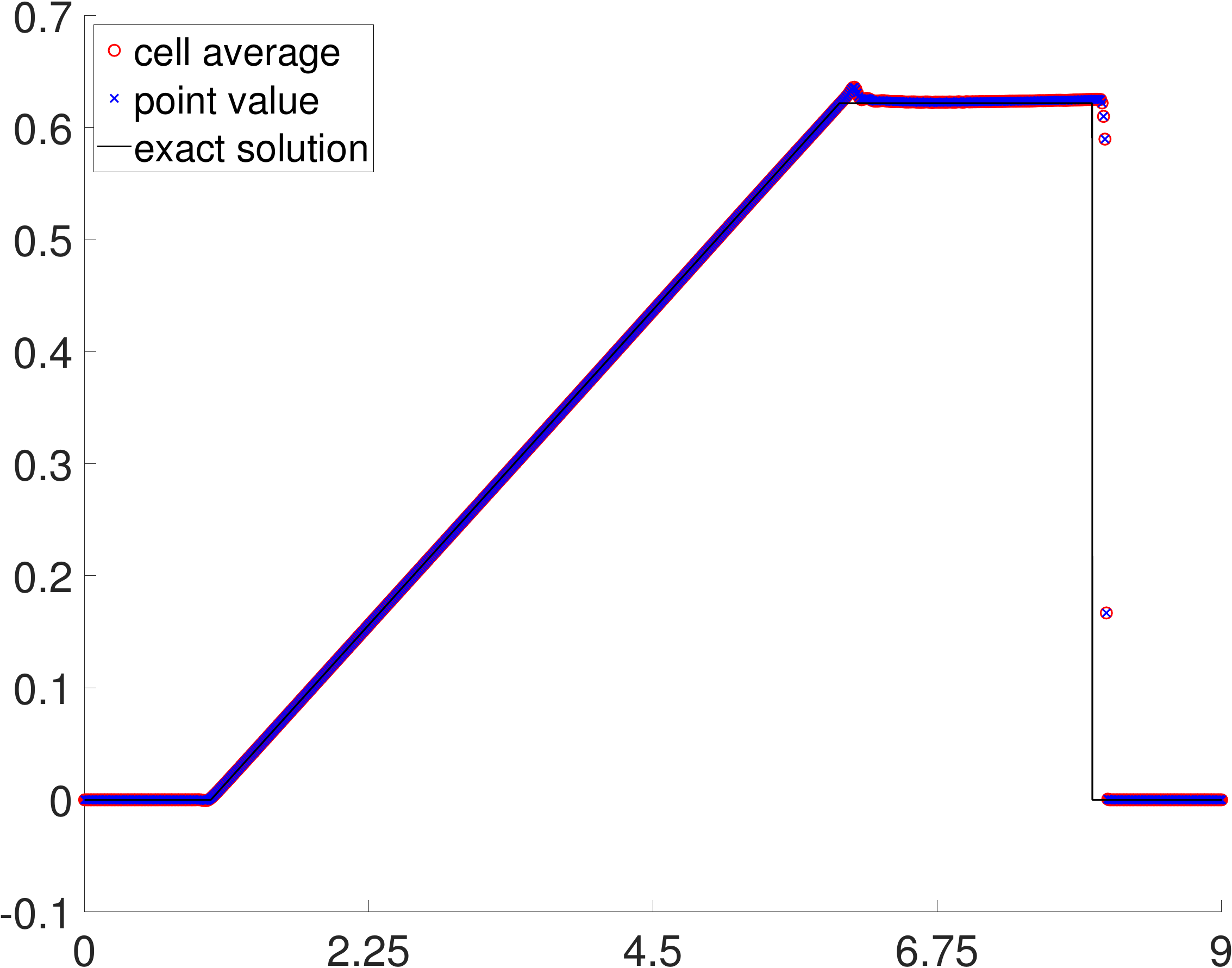}
        \caption{velocity}
    \end{subfigure}
    \hfill
        \begin{subfigure}[b]{0.32\textwidth}
        \centering
        \includegraphics[width=\textwidth]{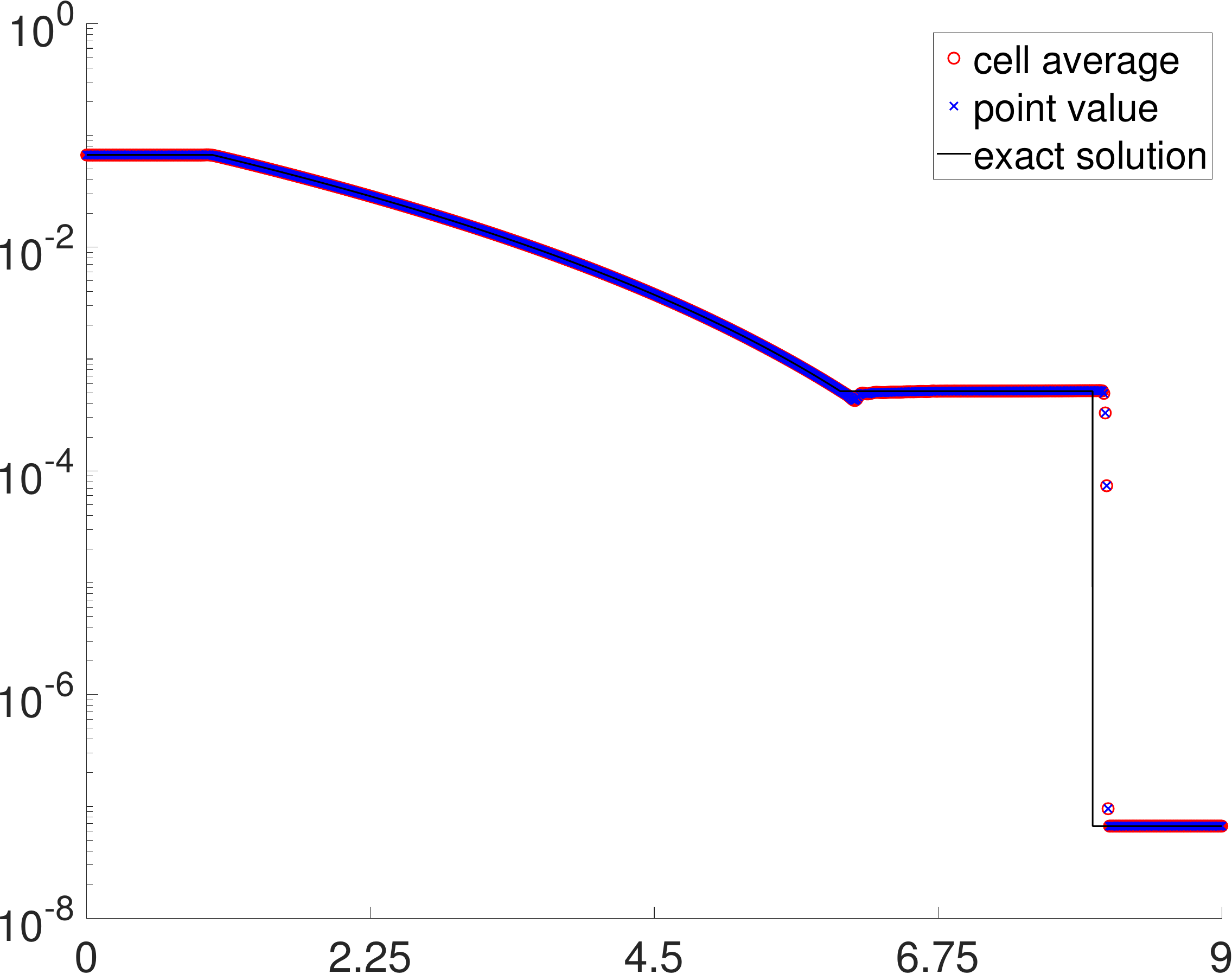}
        \caption{pressure}
    \end{subfigure} 
    \caption{Example 7.10: Numerical solutions computed by IDP PAMPA scheme with MP limiter (top) and OE procedure (bottom).}
    \label{eulerrhoLeblancproblem}
\end{figure}

\subsection{Ideal MHD Equations}

This subsection provides two examples of the 1D ideal MHD system \eqref{MHD}. Similar to the Euler equations, we evolve the point values using the variables  
$$
\mathbf{W} = \left(\ln\left(e^{\frac{\rho}{\rho_{\rm ref}}} - 1\right), v_x, v_y, v_z, B_y, B_z, s\right)^T,
$$ 
which ensures the updated point values are automatically IDP without any limiter. 

\begin{ex}[MHD shock tube problem]
    \rm
    To assess the shock-capturing capability of the IDP PAMPA scheme, we simulate a shock tube problem introduced in \cite{ryu1994numerical}. The adiabatic index is set as $\gamma = \frac{5}{3}$. The initial conditions are  
    $$
    (\rho, \mathbf{v}, \mathbf{B}, p) = 
    \begin{cases} 
        (1, 0, 0, 0, 0.7, 0, 0, 1), & {\rm if} ~~ x < 0.5, \\ 
        (0.3, 0, 0, 1, 0.7, 1, 0, 0.2), & {\rm if} ~~ x > 0.5. 
    \end{cases}
    $$
    The density and magnetic pressure at $t = 0.2$ computed using the IDP PAMPA scheme on 800 uniform cells are presented in Figure \ref{mhdshocktubeproblem}. The OE procedure is applied to eliminate spurious oscillations. 
    The reference solution is computed by using IDP PAMPA scheme with MP limiters on 4000 uniform cells. 
    We see that the waves and discontinuities are clearly resolved by the IDP PAMPA scheme. 
\end{ex}

\begin{figure}[h]
    \centering
    \begin{subfigure}[b]{0.48\textwidth} 
        \centering
        \includegraphics[width=\textwidth]{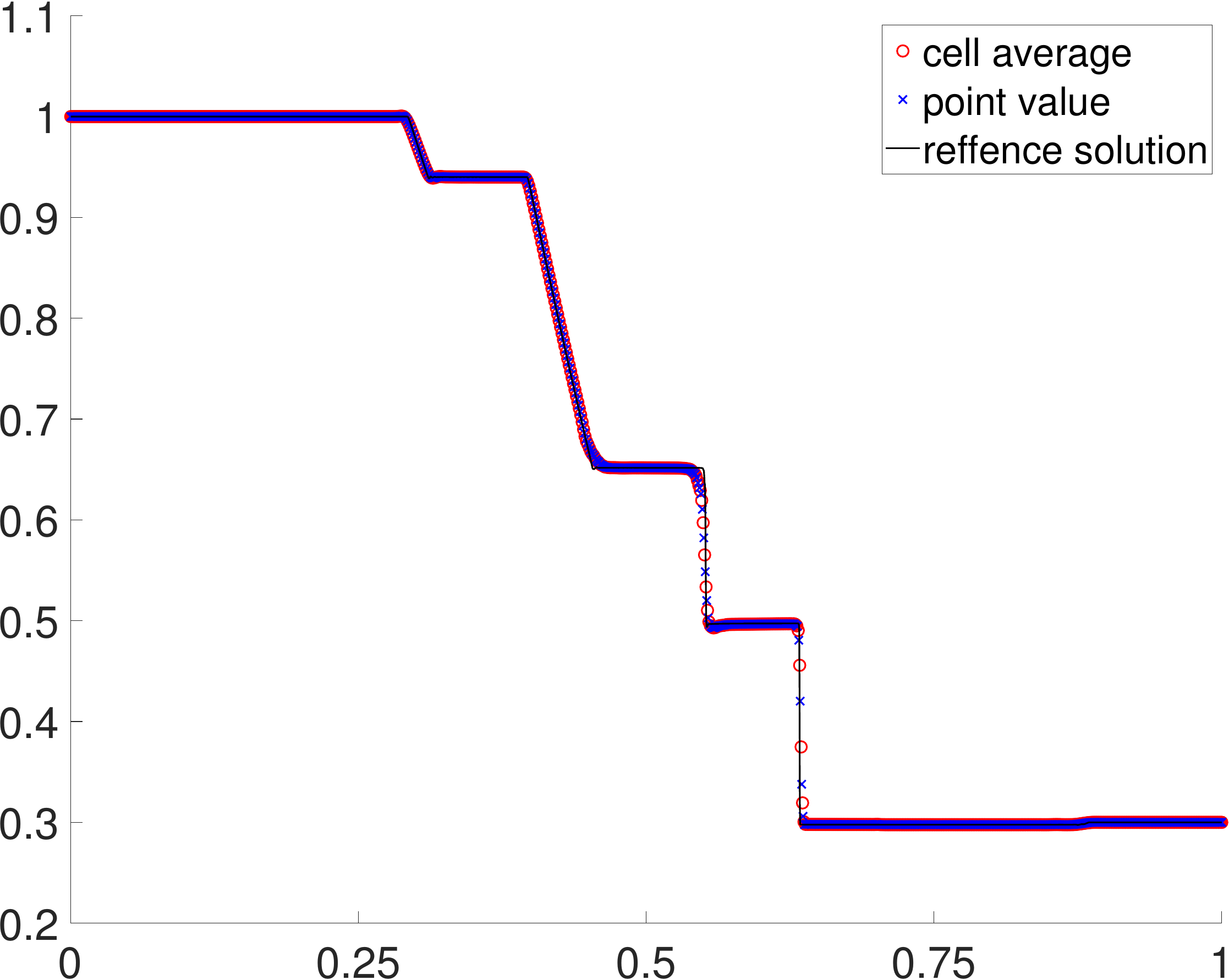}
        \caption{density}
    \end{subfigure}
    \hfill
    \begin{subfigure}[b]{0.48\textwidth}
        \centering
        \includegraphics[width=\textwidth]{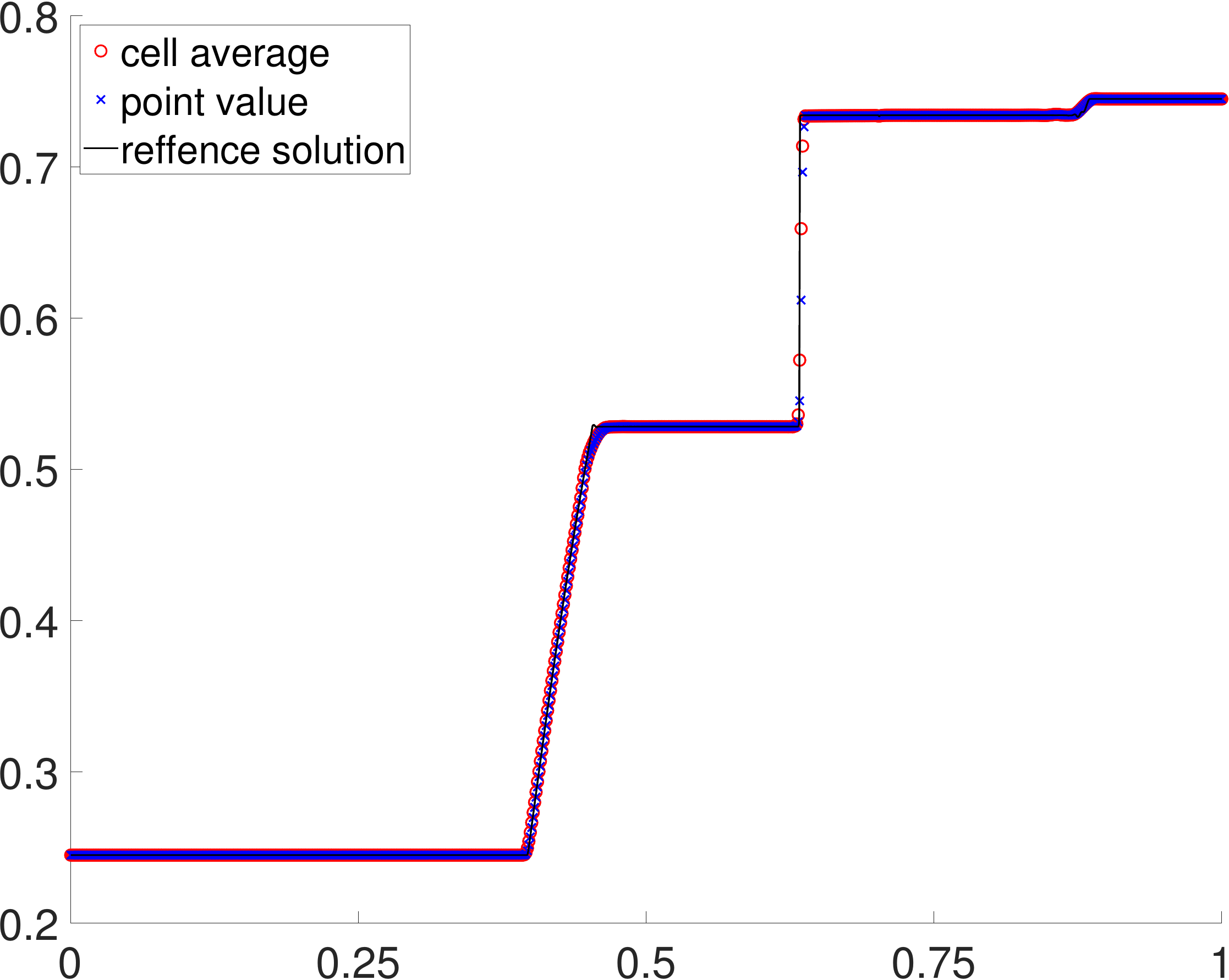}
        \caption{magnetic pressure}
    \end{subfigure} 
    \caption{Example 7.11: The density and magnetic pressure computed by our IDP PAMPA scheme with 800 uniform cells.}
    \label{mhdshocktubeproblem}
\end{figure}

\begin{ex}[MHD Leblanc problem with strong magnetic field]
    \rm
    To assess the robustness of the IDP PAMPA scheme for ideal MHD, we consider a challenging Riemann problem with a strong magnetic field. This test is a variant of the classical Leblanc problem in gas dynamics, adapted to include magnetic effects as introduced in \cite{Wu2019PPMHD}. The initial conditions are  
    $$
    (\rho, \mathbf{v}, \mathbf{B}, p) = 
    \begin{cases} 
        (2, 0, 0, 0, 0, 5000, 5000, 10^9), & {\rm if} ~~ x < 0, \\ 
        (0.001, 0, 0, 0, 0, 5000, 5000, 1), & {\rm if} ~~ x > 0. 
    \end{cases}
    $$
    The computational domain is $[-10, 10]$, and the adiabatic index $\gamma$ is 1.4. The problem features a huge initial pressure jump and a very low plasma-beta ($\beta = 4 \times 10^{-8}$) in the right state, making the numerical computation highly challenging. The density and magnetic pressure at $t = 0.00003$ obtained by using IDP PAMPA scheme  with 2000 uniform cells are shown in Figure \ref{mhdLeblancproblem}. The MP limiter is used to suppress spurious oscillations. The reference solutions are computed by using IDP PAMPA scheme on 10,000 uniform cells. One can see that the strong discontinuities are well captured even in such extreme test case. Again, we observe that the original PAMPA scheme without our IDP technique fails in this test, due to the generation of negative pressure. The proposed IDP PAMPA scheme keeps the positivity of density and pressure, and works very robustly throughout the simulation. 
\end{ex}

\begin{figure}[h]
    \centering
    \begin{subfigure}[b]{0.48\textwidth} 
        \centering
        \includegraphics[width=\textwidth]{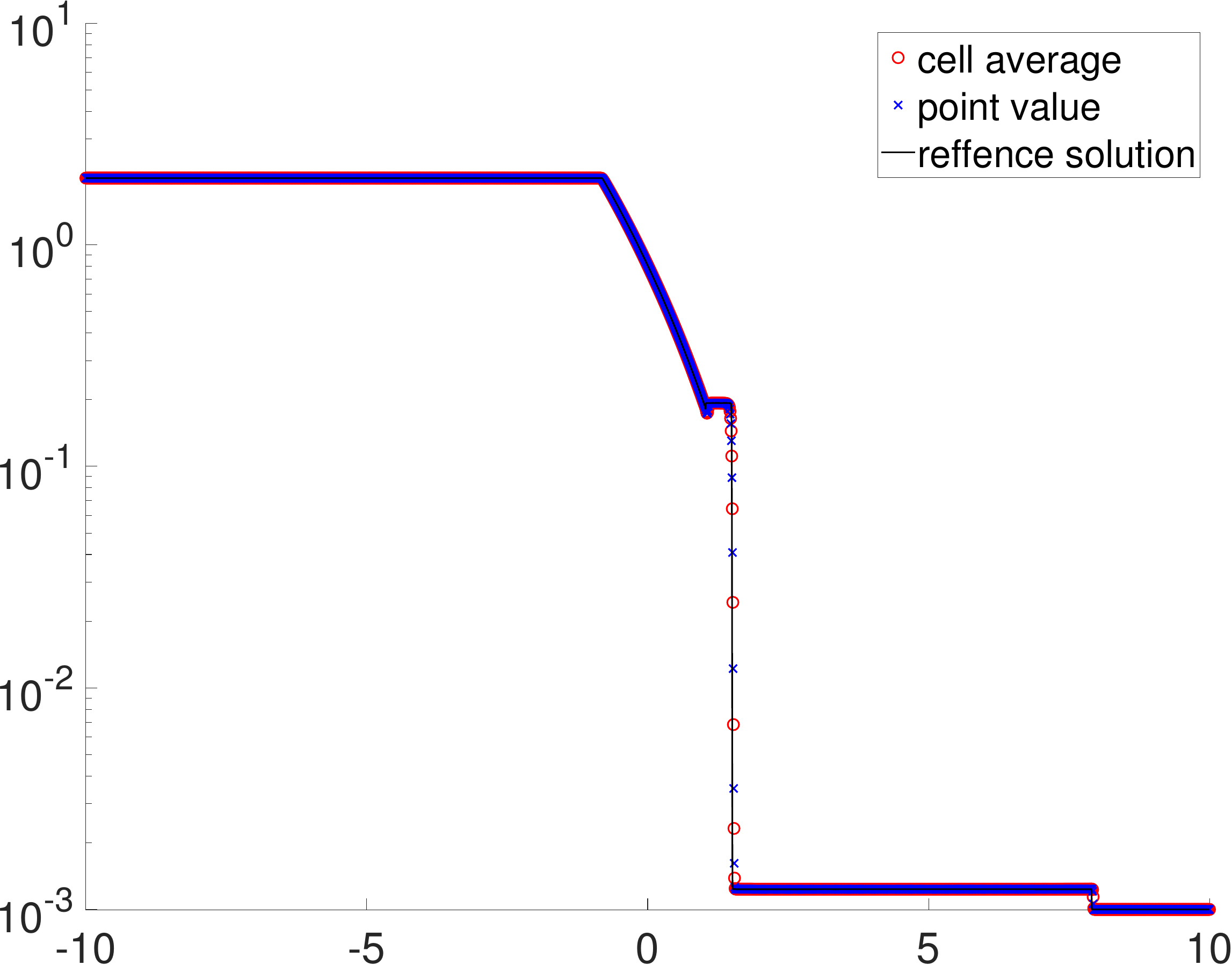}
        \caption{density}
    \end{subfigure}
      \hfill
        \begin{subfigure}[b]{0.48\textwidth}
        \centering
        \includegraphics[width=\textwidth]{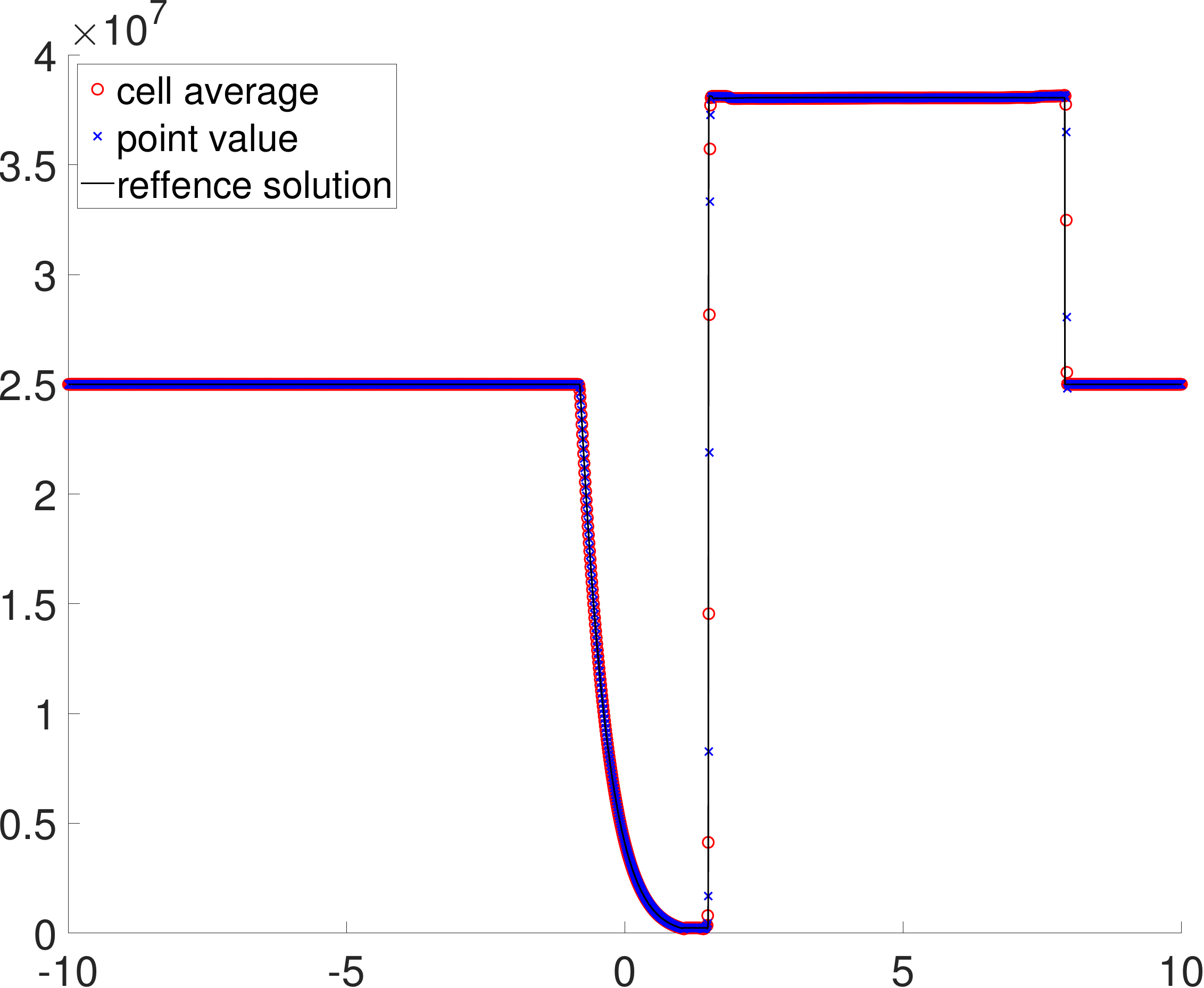}
        \caption{magnetic pressure}
    \end{subfigure}
    \caption{Example 7.12: The density and magnetic pressure computed by our IDP PAMPA scheme with 2000 uniform cells.}
    \label{mhdLeblancproblem}
\end{figure}

\section{Conclusions}

This paper proposed a novel and efficient IDP framework for the PAMPA scheme applied to general hyperbolic systems of conservation laws. We provided a rigorous theoretical analysis of the IDP property for the original PAMPA scheme, identifying key challenges and motivating the development of a simple IDP limiter to ensure midpoint values remain within the invariant domain. Additionally, we presented a provably IDP PAMPA scheme for cell averages that avoided the need for additional convex limiting procedures, distinguishing it from existing bound-preserving PAMPA schemes in the literature. By reformulating the governing equations, we developed an unconditionally limiter-free IDP scheme for evolving point values, drawing inspiration from techniques in machine learning. To further enhance stability, we introduced new oscillation-eliminating and monotonicity-preserving techniques that effectively suppressed spurious oscillations, allowing the PAMPA scheme to capture strong shocks. The accuracy and robustness of the proposed IDP PAMPA scheme were validated through a series of 1D numerical tests, including the linear advection equation, Burgers' equation, compressible Euler equations, and MHD equations. These results demonstrated that the IDP framework significantly improved the stability and accuracy of the PAMPA scheme for a wide range of hyperbolic equations. Future work will extend this IDP framework to multidimensional problems.

\bibliographystyle{siamplain}
\bibliography{references_article}

\end{document}